\documentclass{article}

\usepackage[a4paper,
 total={140mm,237mm},
 left=35mm,
 top=30mm,
 ]{geometry}
\usepackage[utf8]{inputenc}
\usepackage{mathpazo}
\usepackage{hyperref}
\usepackage{doi}
\usepackage{relsize}

\usepackage[T1]{fontenc}
\usepackage{amssymb, amsmath}
\usepackage[dvipsnames]{xcolor}
\usepackage[english]{babel}
\usepackage{amsthm}
\usepackage{enumitem}
\usepackage{tikz}
\usepackage{bm}
\usepackage{csquotes}
\usepackage[shortcuts]{extdash}
\usepackage{stmaryrd}
\usepackage[colorinlistoftodos]{todonotes}
\usepackage{braket}
\usepackage{authblk}

\usepackage[format=plain,
            labelfont=it,
            textfont=it]{caption}
\usepgflibrary{shapes.geometric}
\usepackage{float}
\allowdisplaybreaks

 \hypersetup{breaklinks=true, 
           colorlinks=true,   
        }

\usepackage[maxnames=20]{biblatex} 
\bibliography{Random.bib}

\usepackage{subfiles}

\AtBeginBibliography{\small}
 
\newtheorem{theorem}{Theorem}[section]
\newtheorem{innercustomthm}{Theorem}
\newenvironment{customthm}[1]
  {\renewcommand\theinnercustomthm{#1}\innercustomthm}
  {\endinnercustomthm}
\newtheorem{corollary}[theorem]{Corollary}
\newtheorem{lemma}[theorem]{Lemma}
\newtheorem{prop}[theorem]{Proposition}
\newtheorem{fact}[theorem]{Fact}
\theoremstyle{remark}
\newtheorem{remark}[theorem]{Remark}
\newtheorem{notation}[theorem]{Notation}

\theoremstyle{definition}
\newtheorem{example}[theorem]{Example}
\newtheorem{examples}[theorem]{Examples}
\newtheorem{definition}[theorem]{Definition}
\newtheorem{question}[theorem]{Question}

\title{When invariance implies exchangeability \\ (and applications to invariant Keisler measures)}
\author[1]{Samuel Braunfeld}
\author[2]{Colin Jahel}
\author[3]{Paolo Marimon}
\affil[1]{Computer Science Institute, Charles University. Prague, Czechia \\ The Czech Academy of Sciences, Institute of Computer Science, Pod Vod\'{a}renskou v\v{e}\v{z}\'{\i} 2, 182 00 Prague, Czech Republic.}
\affil[2]{Institut f\"{u}r Algebra, TU Dresden. Dresden, Germany.}
\affil[3]{Institut f\"{u}r Diskrete Mathematik \& Geometrie, TU Wien. Vienna, Austria.}

\newcommand{\G}{\mathbf{G}}
\newcommand{\HH}{\mathbf{H}}
\newcommand{\A}{\mathbf{A}}
\newcommand{\B}{\mathbf{B}}
\renewcommand{\P}{\mathbb{P}}
\newcommand{\N}{\mathbb{N}}
\newcommand{\R}{\mathbb{R}}
\newcommand{\FF}{\mathcal F}
\newcommand{\CC}{\mathcal C}
\newcommand{\DD}{\mathcal D}
\newcommand{\LL}{\mathcal L}
\newcommand{\MM}{\mathcal M}
\newcommand{\K}{\mathbb K}
\newcommand{\closed}{$k$-overlap closed}
\newcommand{\abs}[1]{\left| #1 \right|}
\newcommand{\bs}{\backslash}

\newcommand{\indep}[2]{%
  \mathrel{
    \mathop{
      \vcenter{
        \hbox{\oalign{\noalign{\kern-.3ex}\hfil$\vert$\rlap{$^\mathrm{#2}$}\hfil\cr
              \noalign{\kern-.7ex}
              $\smile$\cr\noalign{\kern-.3ex}}}
      }
    }\displaylimits_{#1}
  }
}

\newcommand{\notindep}[2]{%
  \mathrel{
    \mathop{
      \vcenter{
        \hbox{\oalign{\noalign{\kern-.3ex}\hfil$\not\vert$\rlap{$^\mathrm{#2}$}\hfil\cr
              \noalign{\kern-.7ex}
              $\smile$\cr\noalign{\kern-.3ex}}}
      }
    }\displaylimits_{#1}
  }
}

\begin{document}

\maketitle

\begin{abstract} 
We study the problem of when, given a countable homogeneous structure $\mathcal{M}$ and a space $S$ of expansions of $\mathcal{M}$, every $\mathrm{Aut}(\MM)$-invariant probability measure on $S$ is exchangeable (i.e. invariant under all permutations of the domain). We show, for example, that if $\mathcal{M}$ is a finitely bounded homogeneous $3$-hypergraph with free amalgamation (including the generic tetrahedron-free $3$-hypergraph), all $\mathrm{Aut}(\MM)$-invariant random expansions by graphs are exchangeable. Moreover, we extend and recover both the work of Angel, Kechris, and Lyons on invariant random orderings and some of the work of Crane and Towsner, and Ackerman on relative exchangeability.\\

In the second part of the paper, we apply our results to the study of invariant Keisler measures, which we prove to be particular invariant random expansions. Thus, we describe the spaces of invariant Keisler measures of various homogeneous structures, obtaining the first results of this kind since the work of Albert and Ensley. We also show there are $2^{\aleph_0}$ supersimple homogeneous ternary structures for which there are non-forking formulas which are universally measure zero.

\end{abstract}

\textbf{MSC classes:} 0C315, 0C345, 60G09, 60C05, 37A05.

\tableofcontents
\section{Introduction}
\newcommand{\Aut}{\mathrm{Aut}}
\newcommand{\Q}{\mathbb{Q}}
\newtheorem{problem}{Problem}

We study random expansions of a countable structure $\MM$ {with domain $\N$, i.e. probability distributions on the space of expansions of $\MM$ to a chosen language,} whose distribution is invariant under the action of $\mathrm{Aut}(\MM)$. For example, given a hypergraph $\MM$, we want to understand in which ways one can build a random graph on the vertices of $\MM$ whose distribution is invariant under automorphisms of $\MM$. 
These expansions will include random expansions invariant under \textit{all} permutations of the domain, i.e., $S_\infty := \Aut(\mathbb{N},=)$, which we call \textbf{exchangeable structures}. The aim of this paper is to give criteria under which exchangeable structures are the only $\mathrm{Aut}(\MM)$-invariant random expansions of a prescribed type.\\

This is an instance of the following problem \cite[Problem 6.2]{Aldous}, which was inspired by the result of \cite{ryll1957stationary} that {$\mathrm{Aut}((\Q, <))$-}invariant random unary expansions of $(\Q, <)$ are exchangeable:

\begin{problem}\label{q1} What hypotheses \textit{prima facie} weaker than exchangeability do in fact imply exchangeability? 
\end{problem}

Once we know $\Aut(\MM)$-invariant random expansions are exchangeable, we may use the Aldous-Hoover Theorem \cite{AldousT, HooverT}, which gives a description of exchangeable structures (cf. \cite[Section 2.1]{CraneTown}). So we also answer instances of \cite[Problem 12.18]{Aldous} asking (in somewhat more general formalism, cf. Subsection~\ref{subsub:relationlit}) for characterizations of $\Aut(\MM)$-invariant random expansions for various $\MM$.\footnote{See also \cite[Question 5.3]{CraneTown}, \cite[Research Problem 8.4]{Cranebook}, and \cite[p.31]{Kallenbergsome} for related questions.}\\


Exchangeability first arose in probability with de Finetti's theorem \cite{deFinetti1}, which can be seen as describing the exchangeable unary structures. De Finetti's theorem and the Aldous-Hoover theorem have found many applications in probability and statistics, where suitable indistinguishability assumptions on a population are natural \cite{kingman1978uses, Aldous_2010}. Exchangeable graphs appear in extremal combinatorics as graphons \cite{janson2007graph}, the central objects of the theory of graph limits \cite{lovasz2006limits, lovasz2012large}, and exchangeable structures have appeared in model theory when studying random constructions of countable structures \cite{Oligomperm, PetrovVershik, AFP, AFKrP, AFKwP}. The generalization from $S_\infty$-invariant measures to $\mathrm{Aut}(\MM)$-invariant measures again naturally arises in the study of symmetry and invariance principles in probability \cite{Kallbooksym, Aldous}, with applications to statistical networks, where $\MM$ may encode some heterogeneity in the population in a model for a large network \cite[Chapter 8]{Cranebook}. In model theory, this generalization arises in the study of invariant Keisler measures, as we will show in the second part of this paper.\\

Our main interest lies in the case where $\MM$ is a homogeneous structure, i.e. every isomorphism between finite substructures of $\MM$ extends to an automorphism of the whole structure. Such structures have suitably rich automorphism groups, and questions about their invariant random expansions can be reduced to purely finitary questions about certain expansions of their age, i.e. their class of finite substructures.\\ Examples of homogeneous structures are $(\mathbb{N}, =), (\mathbb{Q}, <)$, the countable random graph, and countable random hypergraphs. More generally, Fra\"{i}ss\'{e}'s Theorem yields that for many interesting classes $\mathcal{C}$ of finite structures there is a countable homogeneous structure $\MM$ whose age is $\mathcal{C}$. Indeed, invariant random expansions of homogeneous structures naturally appear in many problems, and their study resulted in applications to tame hypergraph regularity \cite{chernikov2020hypergraph}, to spin-glass models in physics \cite{austin2014hierarchical, crane2018relative}, and to probabilistic programming \cite{jung2021generalization}. Moreover, if we consider homogeneous structures in a possibly infinite language, the study of $\mathrm{Aut}(\MM)$-invariant random expansions recovers in full generality the study of $\{0,1\}$-valued sequences and arrays of random variables invariant under some group $G\leq S_\infty$ (Subsection~\ref{subsub:relationlit}).\\ In this paper, we will be especially interested in homogeneous structures whose ages are obtained by omitting particular substructures, such as the universal homogeneous tetrahedron\-/free $3$-hypergraph, a ternary analogue of the universal homogeneous triangle\-/free graph (see Example \ref{ex:hypergraphs} and Figure \ref{fig:tet}).\\

{
 Towards Problem \ref{q1}, we prove the following results (restricted to special cases for the introduction):

\begin{customthm}{A}[Theorem \ref{thm:cre si}]\label{thm:Intro cre si}
Let $\MM$ be the universal homogeneous tetrahedron-free $3$-hypergraph. Then every $\mathrm{Aut}(\MM)$-invariant random expansion of $\MM$ by a graph is $S_\infty$-invariant. 

More generally, let $\MM$ be a $3$-hypergraph that is the Fra\"{i}ss\'{e} limit of a free amalgamation class (Definition~\ref{def:DAPFAP}) determined by finitely many forbidden structures. Then every $\mathrm{Aut}(\MM)$-invariant random expansion of $\MM$ by a graph is $S_\infty$-invariant.
 \end{customthm}
Let us give a heuristic for our result. When we try to randomly expand a graph $\MM$ (with edge relation $E$) by another graph $E'$, it is easy to see that we can use the graph structure of $\MM$ to build non-exchangeable $\mathrm{Aut}(\MM)$-invariant graph expansions of $\MM$: for example, we could place an $E'$-edge between two vertices if and only if they already form an $E$-edge. This yields an $\mathrm{Aut}(\MM)$-invariant probability measure on the space of graphs on top of $\MM$ which concentrates on a copy of $\MM$, and so it is non-exchangeable as long as $\MM$ is not the countable complete graph or its complement. However, in several cases when $\MM$ is a homogeneous $3$-hypergraph and we randomly expand it by a graph, there is no binary structure that our random expansion can cling on to build a non-exchangeable graph. 

This phenomenon is quite subtle, as there are examples of $2$-transitive homogeneous $3$-hypergraphs $\MM$ with non-exchangeable $\mathrm{Aut}(\MM)$-invariant graph expansions (Subsection \ref{sec:kaygraphs}) using some ``hidden'' binary structure. Hence, for our results, we need $\MM$ to satisfy some condition which implies that its relations do not give rise to such hidden structure. A sufficient condition is the combinatorial closure property that we call $2$-overlap closedness, and $k$-overlap closedness when working in $(k+1)$-ary languages (Definition~\ref{def:closed}). 

Inverse to the issue of hidden lower-arity structure in $\MM$, it may be that although the relations of the structures we are expanding by do not have lower arity than the relations of $\MM$, they are so constrained that they behave as if they did. For example, linear orders are so constrained that they almost behave like unary structures, and so for many binary structures $\MM$, it is still true that every $\mathrm{Aut}(\MM)$-invariant random expansion of $\MM$ by a linear order must be $S_\infty$-invariant \cite{AKL}. This extended applicability appears naturally in our proofs, which only require a condition on the labelled growth rate of the expanding structure (Definition~\ref{def:growthrate}) rather than on its maximum arity.

Combining these considerations leads to the following more general theorem (still restricted in scope for this introduction). Together with the result that the structures considered in Theorem \ref{thm:Intro cre si} have 2-overlap closed age (Lemma~\ref{lemma:k irred}), it implies Theorem \ref{thm:Intro cre si}.


 \begin{customthm}{B}[Theorem \ref{thm:appliedforIKM}] \label{thm:IntroappliedforIKM}
 Let $k \geq 1$ and $\mathcal{M}$ be a homogeneous structure with {\closed} age. Let $\CC'$ be a hereditary class of $\LL'$-structures with arity at most $k$ (or more generally, with labelled growth rate $O(e^{n^{k+\delta}})$ for every $\delta > 0$). Then every $\mathrm{Aut}(\MM)$-invariant random expansion of $\mathcal{M}$ by $\LL'$-structures whose age is in $\CC'$ is $S_\infty$-invariant.
\end{customthm}
In the main text, we work with the more general notion of consistent random expansion (Definition~\ref{CRE}), which describes randomly expanding a class of finite structures $\CC$ by another $\CC'$. In fact, as mentioned above, we can reduce the study of invariant random expansions of homogeneous structures to purely combinatorial statements about their ages.}\\

Existing results obtaining descriptions (or exchangeability) for $\mathrm{Aut}(\MM)$-invariant random expansions employ one of two strategies, which Theorem \ref{thm:IntroappliedforIKM} synthesizes. The first strategy consists in looking at a base structure $\MM$ which is well-behaved enough that we can understand $\mathrm{Aut}(\MM)$-invariant expansions of $\MM$ to arbitrary classes $\mathcal{C}'$. This is the philosophy underlying the original Aldous-Hoover Theorem, concerned with expansions of $(\mathbb{N}, =)$, while a theorem of Kallenberg \cite{Kallsym}, can be seen as giving a representation to the invariant random expansions of $(\mathbb{Q}, <)$. Meanwhile, Crane and Towsner \cite{CraneTown}, and Ackerman \cite{AckerAut} provide a more elaborate representation for invariant random expansions of structures with disjoint $n$-amalgamation ($n$-DAP) for all $n$. For the purposes of this introduction, these may be considered as structures whose age has ``no interesting omitted substructures'' such as the random graph and hypergraphs (Definition~\ref{def:nDAP}). The results of \cite{CraneTown, AckerAut} yield full exchangeability for the invariant random expansions of $k$-transitive structures with $n$-DAP for all $n$ by classes $\mathcal{C}'$ with relations of arity $\leq k$. This strategy, up to now, has been able to only deal with structures which are particularly well-behaved, either by having very few substructures in each size, such as $(\mathbb{Q}, <)$ and $(\mathbb{N}, =)$, or by 
forbidding no interesting substructures, i.e. the case of structures with $n$-DAP for all $n \in \N$. Theorem \ref{thm:IntroappliedforIKM} yields exchangeability results also for many 
base structures with an age which is both rich and 
has interesting classes of forbidden substructures.\\

The alternative approach is that of choosing a class of structures $\mathcal{C}'$ such that, for a large class of base structures $\MM$, all invariant random expansions of $\MM$ by structures with age in $\mathcal{C}'$ are exchangeable. This is the philosophy underlying the original theorem of de Finetti \cite{deFinetti1}, which offers a representation of unary exchangeable structures. In the same vein, Tsankov and the second author \cite[Corollary 3.5]{JahelT} prove exchangeability for invariant random unary expansions of $\omega$-categorical $\MM$ with mild additional properties. A countable structure is $\omega$-categorical if its automorphism group has finitely many orbits on $n$-tuples for each $n$ (Definition~\ref{def:omegacat}), and it follows from the definitions that homogeneous structures in a finite relational language are $\omega$-categorical. Strong results have also been obtained for invariant random expansions by linear orderings. It is easy to see that there is a unique exchangeable linear ordering, i.e. the $S_\infty$-invariant measure $\mu$ on the space of linear orderings of $\mathbb{N}$ concentrating on the isomorphism type of $(\mathbb{Q}, <)$, where for each $k$-tuple $(a_1, \dots, a_k)$ from $\mathbb{N}$, $\mu(a_1<\dots <a_k)=\frac{1}{k!}$.
The first significant results on random expansions by linear orderings are due to Angel, Kechris, and Lyons \cite{AKL} who prove that for many homogeneous $\MM$, the unique $\Aut(M)$-invariant random order-expansion is the exchangeable one. This was further explored in \cite{JahelT}, which provides general conditions on $\omega$-categorical $\MM$ to obtain the same conclusion, and in \cite{CRO}, which studies consistent random orderings of hereditary classes of graphs and conditions which imply uniqueness or non-uniqueness of such orderings.\\

Our approach largely follows that of \cite{AKL}. Indeed, they prove that an asymptotic combinatorial statement about a hereditary class implies uniqueness of the random linear ordering \cite[Lemma 2.1]{AKL}, and then prove said combinatorial  statement for a variety of hereditary classes using a quantitative version of arguments from \cite{nesetril1978probabilistic}. We adapt Lemma 2.1 of \cite{AKL} to get that a weaker asymptotic combinatorial statement implies exchangeability of all invariant random expansions by a class of structures with small enough growth (Lemma \ref{lem:si}) and then we prove this combinatorial statement for structures which are $k$-overlap closed (Theorem \ref{theorem:Chernoff}). Finally, we prove that many structures are indeed $k$-overlap closed in Section \ref{sec:koverlapclosed}. Several challenges appear in adapting techniques for linear orderings to the context of expansions by arbitrary hereditary classes. Notably, proving that a class is $k$-overlap closed is a problem about constructing certain hypergraphs (partial Steiner systems) with many edges that avoid certain configurations. This topic is intensively studied in extremal combinatorics \cite{delcourt2022finding, glock2024conflict}.\\

On top of yielding sufficient conditions for exchangeability, Theorem \ref{thm:IntroappliedforIKM} also suggests where to look for interesting homogeneous structures with non-exchangeable invariant random expansions. Subsection \ref{sec:kaygraphs} looks carefully at one such example: the universal homogeneous parity $k$-hypergraphs (for $k\geq 3$). These are reducts of the random $(k-1)$-hypergraphs \cite{Thomashyp}. For example, the universal homogeneous parity $3$-hypergraph $\mathcal{G}_3$ is a $3$-hypergraph which is a reduct of the random graph obtained by drawing a $3$-hyperedge whenever three vertices have an odd number of edges between them. We prove that there is a unique invariant random expansion of $\mathcal{G}_3$ by the space of its graph expansions which induce it as a reduct. The slow growth rate of the class of parity $3$-hypergraphs compared to the class of graphs plays a role in the non-exchangeability of this measure. 

\subsection{Invariant Keisler measures}
The final part of our paper is dedicated to applications of our results to the study of invariant Keisler measures for homogeneous structures. Given a first-order structure $\mathcal{M}$, a Keisler measure in the variable $x$ is a regular Borel probability measure on the type space $S_x(M)$. We say a Keisler measure is invariant if it is $\mathrm{Aut}(\MM)$-invariant with respect to the action of $\mathrm{Aut}(\MM)$ on $S_x(M)$. Keisler measures are an active topic of research in model theory with many fruitful connections to combinatorics, such as model-theoretic approaches to variants of Szemer\'edi regularity  \cite{MalliarisReg, CherStarchNIP, RegularityNIP, StarchNIP, PillayStar, AlexisARL}.\\

Our study begins with a problem addressed in some of the earliest work on Keisler measures by Albert \cite{Albert} and Ensley \cite{Ensley, EnsleyPhD}:
\begin{problem}
	Given a countable $\omega$-categorical (or homogeneous) structure $\MM$, can we describe its space of invariant Keisler measures (in the singleton variable $x$)?
\end{problem}
In many cases, our results on exchangeability of invariant random expansions for $\MM$  immediately yield a classification of its invariant Keisler measures. In Section \ref{sec:IKMconnection}, we show that (invariant) Keisler measures over $\MM$ are in natural correspondence with certain ($\Aut(\MM)$-invariant) random expansions. In particular, if the relations of $\MM$ have arity $(k+1)$, then the relations of the expansion will have arity $\leq k$ and so the expansion will satisfy the hypotheses of Theorem \ref{thm:IntroappliedforIKM}. In light of this correspondence, for example, Albert's classification of invariant Keisler measures for the countable random graph $\mathcal{R}$ in \cite{Albert} can be seen as equivalent to the result that invariant random unary expansions of $\mathcal{R}$ are exchangeable~\cite{JahelT} together with de Finetti's classification of exchangeable unary structures~\cite{deFinetti1}.\\

Our correspondence and Theorem \ref{thm:IntroappliedforIKM} allow us to describe the spaces of invariant Keisler measures for many classes of homogeneous structures which were beyond the reach of previous techniques. On one hand, Albert's work \cite{Albert} can at best be refined to deal with binary $\omega$-categorical structures (see Section \ref{sec:consequences} and \cite{me2}). Indeed, the difficulties that arise for the random $3$-hypergraph are discussed at length in the final part of Ensley's PhD thesis \cite{EnsleyPhD}. On the other hand, most results on Keisler measures are obtained in the context of NIP structures \cite{NIPinv, KeislerNIP}. These may be considered as "very non-random" structures for the purpose of our exposition since their Keisler measures can be locally approximated by types (Dirac Keisler measures) \cite{NIPinv}. Some stronger characterisations can be obtained for invariant Keisler measures in an $\omega$-categorical NIP context as we note in Section \ref{sec:nip}, elaborating on the work of Ensley \cite{Ensley}. Nevertheless, outside the NIP context, the consensus is that Keisler measures are poorly understood \cite{lots_of_authors, Keislerwild}. In this case, most positive results are obtained under additional assumptions on the measure (e.g. satisfying a version of Fubini's Theorem \cite{MS, EM, AER}) and under higher amalgamation properties (cf. \cite{Measam} and \cite[Theorem B.8, Theorem B.11]{AER}). See Section \ref{sec:consequences} for a discussion of this. Our work allows us to understand the spaces of invariant Keisler measures for many homogeneous structures of high arity outside of an NIP context for which such higher amalgamation properties may fail (e.g. the universal homogeneous tetrahedron-free $3$-hypergraph), yielding a picture where Keisler measures behave very differently from the previously understood examples.\\

As one application of these results, we give many examples of {\em simple} (a formal model-theoretic property) structures where two notions of smallness for definable sets disagree. These notions of smallness are the ideals of non-forking sets $\mathrm{F}(\emptyset)$ and of universally measure zero sets $\mathcal{O}(\emptyset)$. In general, $F(\emptyset)\subseteq \mathcal{O}(\emptyset)$, and equality holds for many natural examples of simple theories \cite{MS}, including the subclass of stable theories \cite{lots_of_authors}. Recently, it was shown this containment could be strict in simple theories \cite{lots_of_authors}, even assuming $\omega$-categoricity \cite{me2}. However, this phenomenon looked somehow rare in the simple world since it was only proved for purpose-built counterexamples, and it seemed plausible that additional tameness conditions could rule this pathological behavior out.\\

Our results flip this picture, suggesting that most homogeneous simple structures have non-forking formulas which are universally measure zero. Below we give a slightly weaker version of the result:

\begin{customthm}{C}[Theorem \ref{IKMdichotomy}] Let $\mathcal{M}$ be a simple $k$-transitive homogeneous structure in a finite $(k+1)$-ary language whose age has free amalgamation and is $k$-overlap closed. Then, any invariant Keisler measure for $\mathcal{M}$ in $x$ is $S_\infty$-invariant. Moreover, 
	\begin{enumerate}
		\item {EITHER: $\MM$ has $n$-DAP for all $n \in \N$}. In this case, \[F(\emptyset)= \mathcal{O}(\emptyset).\]
		\item {OR: $\MM$ does not have $n$-DAP for all $n \in \N$.} In this case,
		\[F(\emptyset)\subsetneq \mathcal{O}(\emptyset).\]
	\end{enumerate}  
\end{customthm}
In particular, for a fixed $(k+1)$-ary language with $k\geq 2$, there are continuum many structures satisfying the hypotheses and falling in the second case~\cite{Kopconstr}, while only finitely many fall in the first case. Furthermore, all these structures are supersimple of $SU$-rank 1 with trivial forking (thus one-based), so these additional properties cannot ensure $F(\emptyset)= \mathcal{O}(\emptyset)$.\\

\subsection{Structure of the paper and notation}

The paper has the following structure. In Section \ref{sec:prelim} we introduce the main notions and definitions relevant for the first half of this paper. In Subsection \ref{sec:fraisseprelim}, we introduce some basic definitions and results regarding Fra\"{i}ss\'{e} classes and homogeneous structures. Subsection \ref{sec:examples} is dedicated to introducing some examples of homogeneous structures relevant to this paper, which the reader may want to keep in mind moving forwards. In Subsection \ref{sec:randomexp}, we introduce the main concepts of this paper: consistent random expansions \cite{AKL, CRO}, and invariant random expansions \cite{IREs, CraneTown, Cranebook, AckerAut}. These definitions capture, respectively, the idea of randomly expanding a class of finite structures $\mathcal{C}$, or a structure $\mathcal{M}$, by a hereditary class of structures $\mathcal{C}'$.\\

Section \ref{sec:invexch} contains the main results of the paper, and is entirely finite combinatorics. In Subsection \ref{sec:wheninvimpliesexch}, we isolate the property of $k$-overlap closedness for a hereditary class of structures $\mathcal{C}$ and prove that it guarantees exchangeability of consistent random expansions by classes $\mathcal{C}'$ with sufficiently slow growth rate. The main results are Lemma \ref{lem:si} and Theorem \ref{theorem:Chernoff}. Subsection \ref{sec:koverlapclosed} is dedicated to proving that many homogeneous structures naturally have $k$-overlap closedness for some reasonable choice of $k$. These results are summarised in Theorem \ref{thm:cre si}. We proceed with Subsection \ref{sec:applications}, which gives several applications of our results and discusses various examples. We conclude with Subsection \ref{sec:kaygraphs}, where we study the universal homogeneous parity $k$-hypergraphs. In Theorem \ref{thm:kaygraphs}, we prove that for a particular space of $k$-hypergraph expansions, there is a unique invariant random expansion of the universal homogeneous parity $k$-hypergraph, which is not exchangeble.\\

Section \ref{sec:IKMconnection} marks the shift towards invariant Keisler measures in this paper. Subsection \ref{sec:modelprelims} gives some model theoretic preliminaries and basic facts relevant to the rest of the article. Subsection \ref{howto} shows how we may consider invariant Keisler measures as a special case of invariant random expansions. In particular, Definition \ref{Mpstar} gives a way to consider the type space of a given structure $\mathcal{M}$ as represented by a particular space of expansions of $\mathcal{M}$, and having such a representation yields a way to view invariant Keisler measures on $\mathcal{M}$ as a particular kind of invariant random expansions in Corollary \ref{meascorr}. Subsection \ref{homcont} studies this correspondence in a homogeneous context.\\

Section \ref{sec:consequences} explores the consequences of Section \ref{sec:invexch} and Section \ref{sec:IKMconnection} for invariant Keisler measures of homogeneous structures. In Subsection \ref{classtime}, we describe the spaces of invariant Keisler measures of various homogeneous structures in higher arity, giving the first such results since the work of Albert \cite{Albert} and Ensley \cite{Ensley, EnsleyPhD}. We also use each example to exhibit unusual behaviour of invariant Keisler measures which could not be observed in a binary or NIP context. In Subsection \ref{FOtime}, we apply our techniques to prove Theorem \ref{IKMdichotomy}, which gives $2^{\aleph_0}$ many examples of model theoretically tame simple homogeneous structures with $F(\emptyset)\subsetneq\mathcal{O}(\emptyset)$.\\

Section \ref{sec:nip} focuses on invariant Keisler measures in an $\omega$-categorical NIP setting. Building on the work of Ensley \cite{Ensley}, we isolate a property of NIP $\omega$-categorical theories, shared by some other examples with the independence property, which (under very mild assumptions) implies that all invariant Keisler measures are weighted averages of invariant types. We also prove that NIP $\omega$-categorical theories have $F(\emptyset)=\mathcal{O}(\emptyset)$. In spite of our positive results, the question of whether our mild assumptions can be removed remains open.\\

We conclude with Section \ref{sec:finq}, which points out several questions and directions for further research which our work suggests.\\

Regarding notation, throughout the paper we work with a "base" structure $\mathcal{M}$, or class of finite structures $\mathcal{C}$, in a language $\mathcal{L}$ which we expand by another class of structures $\mathcal{C}'$ in a disjoint language $\mathcal{L}'$. We will consistently use $'$ to denote what we are expanding by. We also use the supercript $*$ to denote the class, structure, or language obtained by considering both the "base" and "expanding" part. So $\mathcal{L}^*$ will denote $\mathcal{L}\cup\mathcal{L}'$, $\mathcal{C}^*$ denotes a class of finite $\mathcal{L}^*$-structures and $\mathcal{M}^*$ some $\mathcal{L}^*$-structure. With the exception of Subsection \ref{howto}, $\mathcal{M}$ will denote a countable structure with domain $\mathbb{N}$ (often referred to as $M$). We use the letters $A, B, H,$ and $ G$, to denote finite structures, possibly adding the superscripts $'$ or $^*$ following the conventions mentioned above. We use boldface to distinguish the structure $\HH$ from its domain $H$, when the distinction is important. We use the letters $E$ and $R$ to denote relations, usually in some graph or hypergraph. Whenever we work with hypergraphs we take them to be uniform. \\

This paper requires some background in model theory, especially for its second part, such as Chapters 1-4 in \cite{TZ}. Some more advanced knowledge may be needed for the folklore arguments in Subsection \ref{sec:modelprelims}, such as \cite{CasSimp}. The main arguments in the paper require some basic knowledge of probability, covered in the initial chapters of most books on the probabilistic method, such as Chapters 1-2 of \cite{alon2015probabilistic}.
\\

\textbf{Acknowledgements:}
The authors would like to thank Nate Ackerman, David Evans, Alberto Miguel G\'{o}mez, Jarik Ne\v{s}et\v{r}il, Anand Pillay, Lionel Nguyen Van Th\'e, Misha Tyomkyn, and Julia Wolf for helpful comments and discussions, and Douglas Ensley for finding and forwarding a copy of his thesis. Paolo Marimon is funded by the European Union (ERC, POCOCOP, 101071674). Views and opinions expressed are however those of the authors only and do not necessarily reflect those of the European Union or the European Research Council Executive Agency. Neither the European Union nor the granting authority can be held responsible for them. This paper is part of a project that has received funding from the 
	European Research Council (ERC) under the European Union's Horizon 2020 
	research and innovation programme (grant agreement No 810115 - Dynasnet). Samuel Braunfeld is further supported by Project 24-12591M of the Czech Science Foundation (GA\v{C}R), and by the long-term strategic development financing of the Institute of Computer Science (RVO: 67985807).

\section{Preliminaries}\label{sec:prelim}

\subsection{Fra\"{i}ss\'{e} classes, free amalgamation, and disjoint \texorpdfstring{$n$}{n}-amalgamation}\label{sec:fraisseprelim}
Our main interest in this paper are invariant random expansions of homogeneous structures with free amalgamation, such as the random $r$-hypergraphs or the generic tetrahedron-free $3$-hypergraph. In this section, we cover some basic facts and definitions regarding homogeneous structures and free amalgamation. We also define disjoint $n$-amalgamation. Structures with disjoint $n$-amalgamation for all $n$ include the random $r$-hypergraphs and are mainly relevant for the final part of the paper in Subsections \ref{sec:modelprelims} and \ref{FOtime}. Subsection~\ref{sec:examples} is dedicated to giving examples of homogeneous structures.\\

Throughout this paper we work with countable (often finite) relational languages.

\begin{definition} Let $\LL$ be a relational language. A \textbf{class} $\CC$ is a set of 
$\LL$-structures with domain some finite subset of $\mathbb{N}$ which is closed under isomorphism. {We will always think about these as \emph{labelled} structures, where each point is labelled by the element of $\mathbb{N}$ coming from the domain of the structure. Although we do not require embeddings to respect the labels, we will still consider isomorphic structures that differ only by their labellings as distinct.}

{We say $\CC$ is a \textbf{hereditary class} if it is additionally closed under taking substructures.}
\end{definition}

\begin{definition}\label{def:Fra} Let $\CC$ be  a hereditary class.
\begin{itemize}
    \item[(JEP)] We say $\CC$ has the \textbf{joint embedding property} (JEP) if for any $B_0, B_1\in\CC$, there is some $C\in\CC$ such that $B_0$ and $B_1$ both embed in $\CC$;
    \item[(AP)]\label{eq:ap}  We say $\CC$ has the \textbf{amalgamation property} (AP) if for every $A, B_0, B_1\in\CC$ with embeddings $f_i:A\to B_i$ for $i\in\{0,1\}$, there is $C\in\CC$ and embeddings $g_i:B_i\to C$ for $i\in\{0,1\}$ such that $g_1\circ f_1=g_2\circ f_2$. 
\end{itemize}
A \textbf{Fra\"{i}ss\'{e} class} is a countable hereditary class $\CC$ with the joint embedding property and the amalgamation property.
\end{definition}

\begin{definition} We say that a relational structure $\mathcal{M}$ is \textbf{homogeneous} if every isomorphism between finite substructures of $\mathcal{M}$ extends to an automorphism of $\mathcal{M}$.  
\end{definition}

In some of the literature the word "ultrahomogeneous" is used instead.

\begin{definition}\label{def:omegacat} A countable structure $\mathcal{M}$ in a countable language is $\bm{\omega}$\textbf{-categorical} if the action of $\mathrm{Aut}(\MM)$ on $n$-tuples of $M$ has finitely many orbits for each $n\in\mathbb{N}$.
\end{definition}

Being $\omega$-categorical has several desirable model-theoretic consequences due to the Ryll-Nardzewski Theorem \cite[Theorem 7.3.1]{Hodges}. In particular, if $\mathcal{M}$ is $\omega$-categorical for $A\subseteq M$ finite, any $\mathrm{Aut}(\MM/A)$-invariant set is also definable in $M$ using parameters from $A$. The term $\omega$-categoricity comes from the fact that the first-order theory of an $\omega$-categorical structure has a unique countable model up to isomorphism. Structures which are homogeneous in a finite relational language are $\omega$-categorical.\\

For any relational structure $\mathcal{M}$, we write $\mathrm{Age}(\MM)$ for its class of finite substructures. Fra\"{i}ss\'{e}'s Theorem tells us that Fra\"{i}ss\'{e} classes correspond to the ages of homogeneous structures. 

\begin{fact}[Fra\"{i}ss\'{e}'s Theorem] Let $\mathcal{M}$ be a countable homogeneous structure. Then, $\mathrm{Age}(\MM)$ is a Fra\"{i}ss\'{e} class. Moreover, for any Fra\"{i}ss\'{e} class $\mathcal{C}$, there is a homogeneous structure $\mathcal{M}$ whose age is $\mathcal{C}$, and this structure is unique up to isomorphism. We call this structure the Fra\"{i}ss\'{e} limit of $\mathcal{C}$, $\mathrm{Flim}(\mathcal{C})$.
\end{fact}

\begin{definition}[disjoint amalgamation and free amalgamation]\label{def:DAPFAP} We say that a hereditary class $\mathcal{C}$ has the \textbf{disjoint amalgamation property} (DAP) if it has the amalgamation property and the functions $g_i:B_i\to C$ from Definition \ref{eq:ap} can be chosen so that $g_0(B_0)\cap g_1(B_1)=f_1(A)$. We say that $\mathcal{C}$ has \textbf{free amalgamation} if it has the disjoint amalgamation property and $C$ from Definition \ref{eq:ap} can be chosen so that it does not have any relation intersecting both $g_0f_0(B_0)\setminus f_0(A)$ and $g_1f_1(B_1)\setminus f_1(A)$. 
\end{definition}

The age of a homogeneous structure $\mathcal{M}$ has the disjoint amalgamation property if and only if $\mathcal{M}$ has \textbf{trivial algebraic closure}, in the sense that for every finite $A\subseteq M$ the $\mathrm{Aut}(\MM/A)$-orbit of any $b\in M\setminus A$ is infinite, where $\mathrm{Aut}(\MM/A)$ is the pointwise stabilizer of $A$.\\

Now and in the rest of the paper we often say that a homogeneous structure has a certain property when its age does. For example, we say that a homogeneous structure has free amalgamation when its age does, and so on\dots

\begin{definition} Let $k\geq 2$. A finite $\mathcal{L}$-structure $\mathcal{A}$ is $k$-irreducible if for every $a_1, \dots, a_k\in A$ there is some $\mathcal{L}$-relation $R$ on $A$ and a tuple $\overline{b}\in R$ such that $a_1, \dots, a_k\in\overline{b}$.
\end{definition}

\begin{definition} Let $\mathcal{F}$ be a class of finite $\mathcal{L}$-structures. We write $\mathrm{Forb}(\mathcal{F})$ for the class of finite $\mathcal{L}$-structures which do not embed structures from $\mathcal{F}$.
\end{definition}

\begin{remark} If $\mathcal{C}=\mathrm{Forb}(\mathcal{F})$, we have that $\mathcal{C}=\mathrm{Forb}(\mathcal{F}')$, where $\mathcal{F}'$ is minimal in the sense that for all $A\in\mathcal{F}'$, no proper substructure of $A$ is in $\mathcal{F}'$. Any hereditary class $\mathcal{C}$ is of the form $\mathrm{Forb}(\mathcal{F})$ for some (possibly infinite) $\mathcal{F}$.
\end{remark}

\begin{fact}\label{fact:2irrfree} Let $\mathcal{L}$ be a finite relational language. The class $\mathcal{C}=\mathrm{Forb}(\mathcal{F})$ with $\mathcal{F}$ minimal has free amalgamation if and only if every structure in $\mathcal{F}$ is $2$-irreducible.
\end{fact}

Many natural examples of homogeneous structures have free amalgamation. For example, the random graph and the universal homogeneous $K_n$-free graphs all have free amalgamation. In Subsection \ref{sec:examples}, we give many other examples of homogeneous hypergraphs with free amalgamation.\\

Below, we introduce the notion of disjoint $n$-amalgamation. This can be thought of as a higher dimensional version of disjoint amalgamation. Structures whose age has disjoint $n$-amalgamation for all $n$ can be thought of as having essentially no ``interesting'' omitted substructures. Indeed, they are sometimes referred to in the literature as ``random structures'' \cite{Kopconstr, Palacinrandom}, since they can be built by a probabilistic construction similar to that yielding the random graph by tossing coins to decide whether each pair of vertices from a countable set forms an edge.

\begin{notation} For $n\in\mathbb{N}$, we write $[n]$ for the set $\{1,\dots, n\}$. For $\CC$ a hereditary class, we denote by $\CC[n]$ the class of structures in $\CC$ with domain $[n]$. For a set $C$, and $k\in\mathbb{N}$, we write $[C]^k$ for its set of $k$-element subsets.
\end{notation}

\begin{notation} We denote by $\mathcal{P}([n])^{-}$ the set of subsets $I\subsetneq [n]$. We call some $\mathcal{F}\subseteq\mathcal{P}([n])$ \textbf{downwards closed} if $I\in\mathcal{F}$ and $I'\subseteq I$ implies that $I'\in\mathcal{F}$. 
\end{notation}

\begin{definition}[Disjoint $n$-amalgamation] \label{def:nDAP} Let $\mathcal{C}$ be a hereditary class of relational structures. Given $n\geq 2, m\in\mathbb{N}$, and $A\in\mathcal{C}[m]$, a \textbf{partial disjoint} $\bm{n}$\textbf{-amalgamation problem} over $A$ is a class of structures $(A_I|I\in\mathcal{F})$ for a downwards closed family of subsets $\mathcal{F}\subseteq \mathcal{P}([n])^{-}$ such that there are disjoint finite subsets $K_1, \dots, K_n$ of $\mathbb{N}\setminus[m]$, satisfying that for each $I, J\in\mathcal{F}$,
\begin{itemize}
    \item  the domain of $A_I$ is $[m]\cup\bigcup\{K_i\vert i\in I\}$;
    \item $A_I\upharpoonright_{[m]}=A$;
    \item $A_I\upharpoonright {[m]\cup\bigcup\{K_i\vert i\in I\cap J\}}=A_J\upharpoonright {[m]\cup\bigcup\{K_i\vert i\in I\cap J\}}$.
\end{itemize}
A \textbf{solution} to a partial disjoint $n$-amalgamation problem over $A$ is some $B\in\mathcal{C}([m]\cup \bigcup\{K_i\vert i\in n\}$ such that for each $I\in\mathcal{F}$, $B\upharpoonright \{K_i\vert i\in I\}=A_I$. A \textbf{disjoint} $\bm{n}$\textbf{-amalgamation problem} is a partial disjoint $n$-amalgamation problem where  $\mathcal{F}=\mathcal{P}([n])^{-}$. In this case all of the information about the problem is given by  $\{I\vert I\in[n]^{n-1}\}$, and so we will write this set instead of the corresponding $\mathcal{F}$ for simplicity of notation. We say that a disjoint $n$-amalgamation problem is a $\bm{1}$\textbf{-point} disjoint $n$-amalgamation problem if  $|K_i|=1$ for all $i\leq n$. We say it is \textbf{basic} if it is a $1$-point disjoint $n$-amalgamation problem and $A=\emptyset$. We say that $\mathcal{C}$ has \textbf{disjoint} $\bm{n}$\textbf{-amalgamation} if every disjoint $n$-amalgamation problem has a solution. We say that that it has ($1$-point) basic disjoint $n$-amalgamation if every ($1$-point) basic disjoint $n$-amalgamation problem has a solution. Often, we abbreviate 'disjoint $n$-amalgamation' with '$n$-DAP'.
\end{definition}

 Note that disjoint $2$-amalgamation corresponds to disjoint amalgamation. 
 Below, we state some basic implications between different kinds of disjoint $k$-amalgamation problems having solutions:
 
 \begin{fact}[Basic facts about disjoint $n$-amalgamation, {\cite[cf. Lemma 3.5 \& Proposition 3.6]{Kruck}}]\label{basicndap} Let $\mathcal{C}$ be a hereditary class of relational structures. Suppose that $\mathcal{C}$ has disjoint $k$-amalgamation for all $2\leq k\leq n$. Then, every partial disjoint $n$-amalgamation problem has a solution.\\
Moreover, $\mathcal{C}$ has disjoint $n$ amalgamation for all $n\in\mathbb{N}$ if and only if it has basic disjoint $n$-amalgamation for all $n\in\mathbb{N}$.
\end{fact}

\begin{remark} Note that we also get that $\mathcal{C}$ has disjoint $n$-amalgamation for all $n\in\mathbb{N}$ if and only if it has $1$-point disjoint $n$-amalgamation for all $n$.  
\end{remark}
 



\subsection{Some examples to keep in mind}\label{sec:examples}
It will be helpful to keep in mind some examples of homogeneous structures for which we will be using our techniques. Whilst most results in Subsections \ref{sec:wheninvimpliesexch} and \ref{sec:koverlapclosed} are phrased in general terms, some of the original motivation for our work was understanding random binary expansions of ternary homogeneous structures. Moreover, some of these examples will be used in later sections to showcase different kinds of behaviour of such invariant random expansions.\\

\begin{examples}[Some homogeneous $3$-hypergraphs]\label{ex:hypergraphs}
We define some homogeneous uniform $3$-hypergraphs as the Fra\"{i}ss\'{e} limits of classes $\mathcal{C}$ of $3$-hypergraphs omitting some family $\mathcal{F}$ of induced substructures.

\begin{itemize}
\item[($\mathcal{R}_3$)] The \textbf{universal homogeneous} $\bm{3}$\textbf{-hypergraph} $\mathcal{R}_3$ is the Fra\"{i}ss\'{e} limit of the class of finite $3$-hypergraphs. For $r\geq 2$ denote by $\mathcal{R}_r$ the universal homogeneous $r$-hypergraph;
\item[($\mathcal{H}_4^3$)] The \textbf{universal homogeneous tetrahedron-free} $\bm{3}$\textbf{-hypergraph} $\mathcal{H}_4^3$ is the is the Fra\"{i}ss\'{e} limit of the class of finite $3$-hypergraphs omitting a tetrahedron. A \textbf{tetrahedron} (which we also denote as $\mathcal{K}_4^3$) consists of four vertices such that all four triplets of vertices form a $3$-hyperedge, see Figure \ref{fig:tet}. More generally, for $2\leq r<n$, we write $\mathcal{K}_n^r$ for the complete uniform $r$-hypergraph on $n$-vertices and $\mathcal{H}_n^r$ for the universal homogeneous $\mathcal{K}_n^r$-free $r$-hypergraph. Note that $\mathcal{H}_n^2$ denotes the generic $K_n$-free graph;
\begin{figure}[ht]
    \centering
    \includegraphics[width=0.2\textwidth]{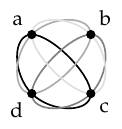}
    \caption{A pictorial representation of $K_4^3$, i.e. a tetrahedron. The different shades of gray for the hyperedges are aimed at making the picture more readable.}
    \label{fig:tet}
\end{figure}
\item[($\mathcal{H}_4^{-}$)] The \textbf{universal homogeneous} $\bm{\mathcal{K}_4^{-}}$\textbf{free} $\bm{3}$\textbf{-hypergraph} $\mathcal{H}_4^{-}$ is the is the Fra\"{i}ss\'{e} limit of the class of finite $3$-hypergraphs omitting a copy of $\mathcal{K}_4^{-}$, the unique (up to isomorphism) $3$-hypergraph on four vertices with three $3$-hyperedges. More generally, for $r\geq 3$, we denote by $\mathcal{K}_{r+1}^{-}$ the unique (up to isomorphism) $r$-hypergraph on $r+1$ vertices with $r$ many $r$-hyperedges and write $\mathcal{H}_{r+1}^{-}$ for the universal homogeneous  $\mathcal{K}_{r+1}^{-}$-free $r$-hypergraph;
\item[($\mathcal{P}_n^r$)] For $2<r\leq n$, we can consider the \textbf{universal homogeneous} $\bm{n}$\textbf{-petal free }$\bm{r}$\textbf{-hypergraph} $\mathcal{P}_n^r$. By $P_n^r$ we denote the $r$-hypergraph on $n+1$-many vertices consisting of $n$ vertices with no hyperedges between them and one distinguished vertex such that any $r-1$ vertices in the independent set form an $r$-hyperedge with it. Figure \ref{fig:Petal} gives a representation of $P_4^3$. By $\mathcal{P}_n^r$ we denote the universal homogeneous $r$-hypergraph omitting  $P_n^r$. Note that for $n=r$, $P_r^r=K_{r+1}^{-}$ and so $P_{r}^r$ is just $\mathcal{H}_{r+1}^{-}$;
\begin{figure}[ht]
    \centering
    \includegraphics[width=0.25\textwidth]{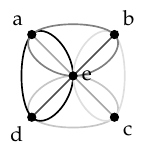}
    \caption{Pictorial representation of $P_4^3$. Note how the vertex $\mathrm{e}$ forms a relation with each pair of vertices in $\{\mathrm{abcd}\}$, but there are no relations containing only the latter four vertices.}
    \label{fig:Petal}
\end{figure}
\item[$(\mathcal{G}_3)$] A parity $3$-hypergraph\footnote{This is usually called a two-graph in the literature \cite{seidel1991survey}.} is a uniform $3$-hypergraph such that any four vertices have an even number of $3$-hyperedges. The \textbf{universal homogeneous parity $3$-hypergraph} $\mathcal{G}_3$ is the Fra\"{i}ss\'{e} limit of the class of parity $k$-hypergraphs. For $k>2$, a \textbf{parity $k$-hypergraph}\footnote{Similarily to the previous footnote, parity $k+1$-hypergraphs are sometimes called kay-graphs \cite{ApproxRamsey}.} is a $k$-uniform hypergraph such that the number of hyperedges on any $k+1$ many vertices has the same parity as $k+1$. The \textbf{universal homogeneous parity $k$-hypergraph} $\mathcal{G}_k$ is the Fra\"{i}ss\'{e} limit of the class of parity $k$-hypergraphs. 
\end{itemize} 
\end{examples}

\begin{remark}\label{rem:ndaphyper} The hypergraphs $\mathcal{R}_r$ satisfy disjoint $n$-amalgamation for all $n$. All of the structures from Example \ref{ex:hypergraphs} have free amalgamation except the parity $k$-hypergraphs, which do still have disjoint amalgamation. It is easy to see that $\mathcal{H}_4^3$ has free amalgamation and disjoint $3$-amalgamation, but fails disjoint $4$-amalgamation: the basic disjoint $4$-amalgamation problem where one specified that each triplet in a set of four vertices forms a $3$-hyperedge has no solution because tetrahedrons are omitted. A similar disjoint $4$-amalgamation problem fails for the universal homogeneous parity $3$-hypergraph $\mathcal{G}_3$, which still satisfies disjoint $3$-amalgamation. This is not entirely trivial and we refer the reader to \cite[Section 7.4.1]{myPhD} for a proof. A similar technique works to prove that the universal homogeneous parity $k$-hypergraph $\mathcal{G}_k$ satisfies disjoint $3$-amalgamation and fails disjoint $k+1$-amalgamation. Finally, the hypergraphs $\mathcal{P}_n^3$ fail disjoint $3$-amalgamation. We prove this later in Lemma \ref{lem:petalnsop}. 
\end{remark}

Whilst there are countably many homogeneous graphs, classified by Lachlan and Woodrow \cite{LachWood}, there are $2^{\aleph_0}$ homogeneous $3$-hypergraphs with free amalgamation, as shown in a construction by Akhtar and Lachlan \cite[Lemma 3]{Lachlanhyp}.


In the second part of the paper, we will need $2^{\aleph_0}$ homogeneous ternary structures satisfying some additional model theoretic tameness assumptions. We introduce basic model theoretic properties in Subsection \ref{sec:modelprelims}. The following examples are a slight modification from the $2^{\aleph_0}$ homogeneous ternary structures which are supersimple of $\mathrm{SU}$-rank $1$ and with free amalgamation constructed by Koponen in \cite[Section $7.3$]{Kopconstr}:

\begin{example}[$2^{\aleph_0}$ ternary homogeneous structures supersimple of $\mathrm{SU}$-rank $1$]\label{Kopconstr}
Let $\mathcal{L}$ consist of a single ternary relation $R$. For $n>3$, let $\mathcal{S}_n$ be the $\mathcal{L}$-structure on $[n]$ where $R$ is interpreted as a relation where for $a,b,c\in[n]$, $\mathcal{S}_n\vDash \neg R(a,b,c)$ if and only if either $a,b,c$ are not distinct or $a=1, b>1$ and either $b<n$ and $c=b+1$ or $b=n$ and $c=2$. For $3<n<m$, there is no embedding from $\mathcal{S}_n$ into $\mathcal{S}_m$ \cite[Lemma 7.4]{Kopconstr}. Let $S_0$ be a set of structures on one and two vertices such that no two of them embed in each other so that $\mathrm{Forb}(S_0)$ consists of all $\mathcal{L}$-structures where $R$ does not hold on tuples with repeated entries. Now, for every $I\subseteq \mathbb{N}\setminus [3], \mathcal{C}_I:=\mathrm{Forb}(S_0\cup\{S_n\vert n\in I\})$ forms a distinct Fra\"{i}ss\'{e} class with free amalgamation and whose Fra\"{i}ss\'{e} limit, $\mathbb{S}_I$ is supersimple with $\mathrm{SU}$-rank $1$. 
\end{example}

\subsection{Randomly expanding structures}\label{sec:randomexp}

In this section, we introduce two related notions for the study of random expansions of structures: consistent random expansions, which capture the idea of randomly expanding a class of finite structures; and invariant random expansions, which capture the idea of randomly expanding a given (usually infinite and highly symmetric) structure. Consistent random expansions have been previously studied in the context of expansions by linear orderings in \cite{AKL} and \cite{CRO}. Meanwhile, the notion of invariant random expansions was previously studied in \cite{CraneTown, AckerAut, Cranebook}, and \cite{IREs}, with the work of Crane and Towsner~\cite{CraneTown, crane2018relative, Cranebook} focusing on relatively exchangeable structures, which correspond to consistent random expansions of hereditary classes with the joint embedding property. The special case of invariant random expansions by colourings, graphs, or hypergraphs has a substantially older history, as these were studied as invariant sequences and arrays in probability \cite{deFinetti1, AldousT, HooverT, Aldous, Kallbooksym}.\\

Given a class $\CC$ (e.g., graphs) and a distinct $\CC'$ (e.g., linear orders), a \textbf{consistent random expansion} of $\mathcal{C}$ by $\mathcal{C}'$ is a family $(\P_\A \mid \A\in\CC)$ such that (i) each $\P_\A$ is a probability distribution on the space of structures in $\CC'$ with domain $A$, and (ii) the $\P_\A$ are compatible with respect to embeddings, so that if there is an embedding $\phi:\B\mapsto \A$, the probability distribution induced by $\P_\A$ on the domain $\phi(B)$ agrees with the probability distribution $\P_\B$. We give the formal definition of consistent random expansions later in Definition \ref{CRE}. Before that, we give some examples of consistent random expansions from \cite{AKL} in order to aid the reader's intuition.\\

\begin{examples} \label{ex:CREs}
$(a)$  For the first example, let $\CC$ be the class of all finite graphs. For each $n \in \mathbb{N}$, we randomly expand each graph $\G$ with domain of size $n$ by a linear order chosen uniformly at random, so each ordering is chosen with probability $\frac{1}{n!}$. We will see that this is a consistent random expansion of the class of graphs by that of linear orders.

    
    
 $(b)$   For a second example, let $\CC$ be the class of all finite path-graphs, i.e. connected graphs for which there are two vertices of degree one and every other vertex has degree two. Again, for each path-graph with domain of size $n$, we may take uniformly at random a linear ordering of the domain, and this would constitute a consistent random expansion. However, we may instead consider the random linear ordering obtained by picking uniformly at random  one of the degree-one vertices to be the ``leftmost'' vertex, and then ordering the vertices of the path from left to right. This is easily seen to be different from the uniformly random ordering, but will still satisfy the definition of a consistent random expansion.
\end{examples}

The second notion for randomly expanding a structure we will introduce in this section is an \textbf{invariant random expansion.} This term refers to an $\mathrm{Aut}(\MM)$-invariant measure on the space of expansions of a given deterministic (generally infinite) structure $\MM$. 
When $\MM$ is homogeneous, its invariant random expansions are in a natural bijection with consistent random expansions of $\mathrm{Age}(\MM)$. Relying on this correspondence, most of our results will be phrased in terms of consistent random expansions. \\

For this subsection, we will assume $\mathcal{L}$ and $\mathcal{L}'$ are disjoint relational signatures and let $\mathcal{L}^*=\mathcal{L}\cup\mathcal{L}'$.

 \begin{definition} 
Let $\A$ and $\B$ be respectively an $\mathcal{L}$-structure and and an $\mathcal{L}'$-structure with the same domain. We write $\A*\B$ for the $\mathcal{L}^*$-structure whose $\mathcal{L}$-reduct is $\A$ and whose $\mathcal{L}'$-reduct is $\B$. This is the \textbf{superposition of } $\A$ and $\B$.
\end{definition}

\begin{definition}
For two classes $\mathcal{C}, \mathcal{C}'$ of $\mathcal{L}$ and $\mathcal{L}'$-structures respectively, we write $\mathcal{C}*\mathcal{C}'$ for the class of structures $\mathbf{A}*\mathbf{B}$ where $\mathbf{A}\in\mathcal{C}$ and  $\mathbf{B}\in\mathcal{C}'$.
\end{definition}

\begin{notation} 
For a (finite or infinite) $\LL$-structure $\MM$, we write $\mathrm{Struc}_{\LL'}(M)$ for the space of $\LL'$-expansions of $\MM$. In particular, by $\mathrm{Struc}_{\LL'}[n]$ and $\mathrm{Struc}_{\LL'}(\mathbb{N})$ we mean the space of $\LL'$ structures on $[n]$ and $\mathbb{N}$ respectively.\\

Let $\CC$ and $\CC^*$ be classes of $\LL$ and $\LL^*$ structures respectively such that $\CC_{\upharpoonright\LL}^*=\CC$. We do not assume $\CC$ or $\CC^*$ to be hereditary.

For $\A\in\CC$, we let
\[\mathrm{Struc}(\A, \CC^*):=\{\A^*\in\CC^*\vert 
\A^*_{\upharpoonright\LL}=\A\}.\]
That is, $\mathrm{Struc}(\A, \CC^*)$ is the class of $\LL^*$-structures in $\CC^*$ which correspond to $\A$ when restricted to $\LL$.\\

We denote by $\mathcal{D}_\A^*$ the set of probability distributions on $\mathrm{Struc}_{\LL'}(\A)$. We write $\mathcal{D}_\A^*(\CC^*)$ for the subspace of $\mathcal{D}_\A^*$ of probability distributions concentrating on $\mathrm{Struc}(\A, \CC^*)$. For simplicity of notation, given $\P_\A\in\mathcal{D}_\A$ and $\A^*\in\mathrm{Struc}_{\mathcal{L}'}(\A)$, with $\A^*_{\upharpoonright \mathcal{L}'}=\A'$, we write $\P_\A(\A')$ instead of $\P_\A(\A^*)$.
\end{notation}

\begin{notation}\label{not:relabel} Given $\HH, \G\in\mathcal{C}$ with domains $H$ and $G$ respectively, $\HH'$ an $\mathcal{L}'$-structure with domain $H$, {and an injection $\phi \colon H \to G$}, we write $\HH'_\phi$ for the relabelling of the points of $\HH'$ by the map $i\mapsto \phi(i)$. So $\HH'$ and $\HH'_\phi$ are isomorphic, but their domains are possibly different subsets of $\N$.
\end{notation}

The following notation is only needed to define formally consistent random expansions and will not be needed later.
\begin{notation}
{Let $\A$ be an $\LL$-structure, let $\P_\A \in \mathcal{D}^*_\A$, and let $\B \subset \A$ with domain $B$. Then $(\P_{\A} \upharpoonright B) \in \mathcal{D}^*_\B$ denotes the distribution induced by $\P_\A$ on expansions of $\B$, i.e. 
\[(\P_{\A}\upharpoonright B)( \B') = \sum \{P_\A(\A') \vert  \;  \A' \upharpoonright B = \B'  \}.\]}

{
    Let $\A, \B$ be $\LL$-structures, let $\phi \colon B \to A$ be an injection on their domains, and let $\P_\B \in \mathcal{D}^*_\B$. Then we let $\phi(\P_\B) \in \mathcal{D}^*_\A$ denote the pushforward of $\P_\B$ by $\phi$, i.e., for $\A'$ with domain $A$, $\phi(\P_\B)(\A') = \P_\B( \A'_{\phi^{-1}})$.}
\end{notation}

\begin{definition}\label{CRE} Let $\mathcal{C}$ be a class of $\mathcal{L}$-structures. By a \textbf{consistent random expansion of} $\CC$ to $\LL'$, we mean a function assigning to each $\A\in\CC$ an element $\P_\A\in \mathcal{D}_\A^*$ such that for $\HH, \G\in\CC$ with domains $H$ and $G$ and $\phi \colon H \to G$ an embedding of $\HH$ into $\G$, we have that $\phi(\P_\HH)$ is the same distribution as $(\P_\G \upharpoonright \phi(H))$. That is, whenever $\HH'$  is an $\mathcal{L}'$-structure with domain $H$, we have 
    \[{\P_\HH(\HH')=(\P_\G {\upharpoonright \phi(H)})(\HH'_\phi).}\]
 In particular, note that for $\HH_1\cong\HH_2$ in $\mathcal{C}$, the isomorphism between them induces an isomorphism between the distributions $\P_{\HH_1}$ and $\P_{\HH_2}$.

\end{definition}
\begin{notation}
    A \textbf{consistent random expansion of} $\CC$ by $\CC'$ is a consistent random expansion of $\CC$ to $\LL'$ such that for all $A\in\CC$, $\P_A\in \mathcal{D}_A^*(\CC^*)$, where $\CC^*=\CC*\CC'$.
\end{notation}

{The reader can now revisit Example \ref{ex:CREs} and verify that the examples satisfy the conditions of Definition \ref{CRE}.}

\begin{definition}\label{def:strucl} 
For $\MM$ an $\LL$-structure on $\mathbb{N}$, we consider $\mathrm{Struc}_{\LL'}(M)$ as a Hausdorff compact topological space whose basis of clopen sets is given by 
\[ \llbracket \phi(\overline{a}) \rrbracket =\{N\in \mathrm{Struc}_{\mathcal{L}'}(M) \vert N\vDash \phi(\overline{a})\},\]
    where $\phi(\overline{x})$ is a quantifier-free $\mathcal{L}'$-formula and $\overline{a}$ is a tuple from $M$ of length $|\overline{x}|$.

There is a natural continuous group action of $\mathrm{Aut}(\MM)$ on $\mathrm{Struc}_{\LL'}(M)$: for $N\in \mathrm{Struc}_{\LL'}(M)$, and $g\in\mathrm{Aut}(\MM)$, the structure $g\cdot N$ is given by the $\LL'$-expansion of $\MM$, where for each $\LL'$-relation in the variable $\overline{x}$, $R(\overline{x})$ and tuple $\overline{a}$ from $M$ (of length $|\overline{x}|$),
\[g\cdot N\vDash R(\overline{a})\text{ if and only if } N\vDash R(g^{-1}(\overline{a})).\]

This is known as the \textbf{relativised logic action} \cite{IREs,Descriptive}, or just \textbf{logic action} when $\mathrm{Aut}(\MM)$ is $S_\infty$ \cite{AFP}.
\end{definition}

\begin{notation} For $\MM$ an $\LL$-structure on $\mathbb{N}$ and $\CC^*$ a hereditary class of $\LL^*$-structures such that $\CC^*_{\upharpoonright\LL}=\mathrm{Age}(\MM)$, we write
\[\mathrm{Struc}(M, \CC^*):=\{M^*\in \mathrm{Struc}_{\LL^*}(\mathbb{N})\vert M^*_{\upharpoonright\LL}=M, \mathrm{Age}(\MM^*)\subseteq \CC^*\}.\]
It is easy to see this is a closed (and hence compact) subspace of $\mathrm{Struc}_{\LL^*}(M)$. 

\end{notation}

\begin{remark}\label{Kolmocorr}
For $A\subset M$ finite and $A^*\in\mathrm{Struc}_{\LL'}(A)$, we write $\llbracket A^* \rrbracket$ for the closed set
\[\{N\in\mathrm{Struc}_{\LL'}(M)\vert N\vDash R(\overline{a}') \text{ if and only if } A^*\vDash R(\overline{a}') \text{ for all }R\in\LL', \overline{a}' \text{a tuple from } A\}.\]

Given a Borel probability measure $\mu$ on $\mathrm{Struc}_{\LL'}(M)$, for each finite substructure $\A\subset M$ $\mu$ induces a probability distribution $\P_\A^\mu\in\DD_\A^*$, where for $\A'\in\mathrm{Struc}_{\LL'}(A)$, and $A^*=A*A'$,
\begin{equation}\label{CREX}
  \P_\A^\mu(\A')=\mu(\llbracket \A^*\rrbracket).  
\end{equation}
Moreover, by the Hahn-Kolmogorov Theorem \cite[Theorem 1.7.8]{Taointro}, $\mu$ is entirely determined by the probability distributions it induces on each finite $A\subset M$.
\end{remark}


\begin{definition}\label{def:IRE} An \textbf{invariant random expansion} of $\mathcal{M}$ to $\LL'$ is a Borel probability measure on $\mathrm{Struc}_{\LL'}(M)$  invariant under the action of $\mathrm{Aut}(\MM)$. We write $\mathrm{IRE}_{\LL'}(M)$ for the space of invariant random expansions of $\mathcal{M}$ to $\LL'$. We say that an invariant random expansion of $\mathcal{M}$ is to $\CC^*$ if it concentrates on
$\mathrm{Struc}(M, \CC^*)$. We write $\mathrm{IRE}(M, \CC^*)$ to denote the space of invariant random expansions of $\mathcal{M}$ to $\CC^*$.
\end{definition}

Given (\ref{CREX}), it is easy to see that for a Fra\"{i}ss\'{e} class $\CC$, its consistent random expansions by $\CC'$ correspond to the invariant random expansions of its Fra\"{i}ss\'{e} limit $\mathrm{Flim}(\CC)$ to $\CC^*:=\CC*\CC'$.

\begin{notation} Invariant random expansions of $\mathbb{N}$ to $\LL'$ are called $\bm{S_\infty}$\textbf{-invariant measures on} $\bm{\mathrm{Struc}_{\LL'}(\mathbb{N})}$ or \textbf{exchangeable structures}. Hence, we write $S_\infty(\LL')$ for $\mathrm{IRE}_{\LL'}(\mathbb{N})$ and 
$S_\infty(\CC^*)$ for $\mathrm{IRE}(\mathbb{N}, \CC^*)$. 
\end{notation}

\begin{remark}\label{rem:finitaryperm}
{In the literature, exchangeable (hyper)graphs are usually required to be invariant only under the subgroup of $S_\infty$ consisting of finitely-supported permutations \cite{Aldous}. 
The latter condition is equivalent to being invariant under all permutations, as we require in our definition. Indeed, the group of finitely supported permutations is dense in $S_\infty$ and the action on the space of expansions is continuous; therefore open sets that are invariant under finitely supported permutation are also $S_\infty$-invariant. The regularity of the measure then ensures that the invariance of open sets implies invariance of all measurable sets.}
\end{remark}

\begin{examples}\label{ex:IREs} $(a)$ Consider the following invariant random expansion of $(\mathbb{N}, =)$ by graphs: toss a coin (i.e., independent and identically distributed $\mathrm{Bernoulli}(1/2)$ random variables) for each pair $\{a,b\}\in[\mathbb{N}]^2$ to decide whether it forms an edge or not, and consider the resulting probability distribution on the space of countable graphs. Clearly, this construction gives rise to an exchangeable structure since the probability of a given graph holding for some set of points does not depend at all on their placement on the natural numbers;\\
$(b)$ Consider instead a copy of $\mathcal{R}_2$ for $\mathcal{M}$. We can easily construct several non-exchangeable invariant random expansions of $\mathcal{R}_2$ by graphs. For example, we could toss a coin for each vertex, and place an edge (in the language $\mathcal{L}'$) between two vertices if and only if, both coin tosses for the vertices are heads AND those vertices already form an $\mathcal{L}$-edge in $\mathcal{R}_2$. The resulting probability distribution on the space of expansions of $\mathcal{R}_2$ by an $\mathcal{L}'$-graph relation is invariant under automorphisms of $\mathcal{R}_2$, but is clearly not exchangeable: if two vertices do not form an edge in $\mathcal{R}_2$, the probability that they form an $\mathcal{L}'$-edge is zero, but if they do form an edge, it is $1/4$.
\end{examples}

\begin{definition}[Exchangeability] A consistent random expansion $(\mathbb{P}_A\vert A\in\mathcal{C})\in\mathrm{CRE}_\mathcal{L'}(\mathcal{C})$ is \textbf{exchangeable} if for all $\HH'\in\mathrm{Struc}_{\mathcal{L}'}[r]$ and $\HH_1, \HH_2\in \mathcal{C}[r]$, 
\[\mathbb{P}_{\HH_1}(\HH')=\mathbb{P}_{\HH_2}(\HH').\]

\end{definition}

{Revisiting the expansions in Example \ref{ex:CREs}, we see that the first example is exchangeable while the second is not. In particular, take the path on the domain $[3]$ with two different labellings, first so that the vertex labelled 3 has degree 2 and second so that the vertex labelled 3 has degree 1. In the first labelling the vertex labelled 3 can never be the first point of the order expansion, while in the second labelling it is the first point of the order expansion with probability $\frac{1}{2}$.}

\subsubsection{Relations to other works in the literature}\label{subsub:relationlit}

We conclude this section discussing relations between our paper and other works in the literature.\\
 
Exchangeable structures and $S_\infty$-invariant measures have been heavily studied in probability, the theory of graph limits (graphons), and logic. As mentioned earlier, Aldous and Hoover \cite{AldousT, HooverT} give a representation for exchangeable graphs and hypergraphs (see Fact \ref{AldousHoover}), which can be adapted to arbitrary relational structures with some technical twists \cite{CraneTown, AFKwP}. There is a deep connection between the theory of exchangeable graphs and the theory of graph limits (i.e., graphons) developed by Lov\'{a}sz, Szegedy and coauthors \cite{lovasz2006limits, lovasz2012large, borgs2008convergent, borgs2012convergent} since ergodic (a.k.a. dissociated) exchangeable graphs correspond to graphons \cite{janson2007graph}. In logic, exchangeable structures have been heavily studied by Ackerman, Freer, and Patel \cite{AFP} and collaborators \cite{AFKrP, AFKwP}. In particular, \cite{AFP} describes for which countable structures we can find an exchangeable structure concentrating on its isomorphism type. We discuss this problem more in Section \ref{sec:consequences}. In the work of Albert and Ensley \cite{Albert, Ensley, EnsleyPhD}, which will be relevant to the second part of our paper, exchangeable structures appear under the name of Cameron measures, following \cite[Section 4.10]{Oligomperm}.\\

The generalisation from exchangeability to $\mathrm{Aut}(\MM)$-invariance is very natural, and was already discussed by Aldous \cite{Aldous}, who gives a representation for invariant random unary expansions for the infinitely branching infinite tree (see \cite{pemantle1992automorphism} for the case of the infinite $k$-ary branching tree). As mentioned above, our definition of consistent random expansion comes from generalising to arbitrary classes the notion of consistent random orderings from \cite{AKL} and  \cite{CRO}. To our knowledge, the work of Crane and Towsner \cite{CraneTown, crane2018relative, Cranebook} and Ackerman \cite{AckerAut} consitutes the first papers trying to systematically study invariant random expansions of homogeneous structures. In particular, the former authors \cite{CraneTown} study consistent random expansions of hereditary classes with the joint embedding property under the name of relatively exchangeable structures. We use the terminology ``invariant random expansions'' introduced in work of the second author and Joseph~\cite{IREs}.\\

As mentioned in the introduction, the formalism used in probability is somewhat different from the one we adopt (cf.~\cite[Section 12]{Aldous}, \cite{Kallbooksym}). Below we give a brief explanation of it and discuss how it relates to our study of invariant random expansions. Let $J$ be $[\mathbb{N}]^k$ or $\mathbb{N}^{(k)}$, where the former is the set of $k$-element subsets of $\mathbb{N}$ and the latter is the set of $k$-tuples from $\mathbb{N}$ consisting of distinct elements. Let $X:=(X_j\mid j\in J)$ be an $\mathcal{S}$-valued sequence of random variables where $\mathcal{S}$ is a Polish space. We call $X$ an $\mathcal{S}$-valued array of random variables, and a sequence when $k=1$. Let $G$ be a group of permutations of $\mathbb{N}$. We say that $X$ is $\bm{G}$\textbf{-invariant} if its distribution is invariant under the action of $G$ on $J$. That is, for each $g\in G$
\[(X_j\mid j\in J)\stackrel{d}{=}(X_{g j}\mid j\in J).\]
Hence, the classical problem in probability is, given a group $G$ of permutations of $\mathbb{N}$, to provide an informative description of $G$-invariant arrays. De Finetti's Theorem \cite{deFinetti1} and the Aldous-Hoover Theorem \cite{AldousT, HooverT}(cf. Fact~\ref{AldousHoover}) are clear examples of such informative descriptions when $G$ is all permutations of $\mathbb{N}$ (i.e., $S_\infty$).\\

If $\mathcal{S}=\{0,1\}$, and $G=\mathrm{Aut}(\MM)$ for a countable structure with domain $\mathbb{N}$, $G$-invariant arrays of the form $(X_i\mid i\in [\mathbb{N}]^k)$ correspond to invariant random expansions of $\MM$ by $k$-hypergraphs, and those of the form $(X_j\mid j\in \mathbb{N}^{(k)})$ correspond to invariant random expansions of $\MM$ by an injective $k$-ary relation. For example, for the first correspondence, $X$ induces an invariant random expansion of $\MM$ by a $k$-hypergraph as follows: let $\{a_1^i, \dots, a_k^i\}\in[\mathbb{N}]^k$ and $\xi_i\in\{0,1\}$ for $i< \ell$. Let $\mathcal{C}'$ be the class of $k$-hypergraphs with relation $R$, and write $R^0$ for $\neg R$ and $R^1$ for $R$. Then, define the Borel probability measure $\mu$ on $\mathrm{Struc}(M, \mathcal{C}')$ by specifying its values on the basic clopen sets as
\begin{equation}\label{eq:rvcorrespondence}
    \mu\left(\left\{\mathcal{N} \ \large\middle\vert \ \mathcal{N}\vDash\bigwedge_{i<\ell} R^{\xi_i}(a_1^i, \dots, a_k^i)\right\}\right)=\P\left(\bigwedge_{i<\ell} X_{\left\{a_1^i, \dots, a_k^i\right\}}=\xi_i\right).
\end{equation}
Clearly, this measure is $\mathrm{Aut}(\MM)$-invariant and so $\mu$ is an invariant random expansion. Conversely, any invariant random expansion of $\MM$ by a $k$-hypergraph induces an $\mathrm{Aut}(\MM)$-invariant $\{0,1\}$-valued array $(X_i\mid i\in [\mathbb{N}]^k)$ by the correspondence given in (\ref{eq:rvcorrespondence}).\\

When $\mathcal{S}$ is compact and Hausdorff, the question of understanding $G$-invariant $\mathcal{S}$-valued arrays for an arbitrary group of permutation of $\mathbb{N}$ reduces to the question of understanding $\mathrm{Aut}(\MM)$-invariant arrays where $\MM$ is a countable homogeneous relational structure. In particular, when $\mathcal{S}=\{0,1\}$, the question of understanding such arrays reduces to the question of understanding some particular invariant random expansions. To see this, note that the group $G$ acts continuously on $\mathcal{S}^I$ when endowed with the pointwise convergence topology, and the latter space is compact and Hausdorff. Hence, the same argument as in Remark~\ref{rem:finitaryperm} yields that $G$-invariance of $X$ implies $\overline{G}$-invariance of $X$, where $\overline{G}$ is the closure of $G$ in $S_\infty$ with respect to the topology of pointwise convergence. It is well-known that closed subgroups of $S_\infty$ correspond to the automorphism groups of countable homogeneous relational structures (possibly in an infinite language). This gives us a natural reason to focus only on $\mathrm{Aut}(\MM)$-invariant arrays for some countable homogeneous structure $\mathcal{M}$. Obviously, understanding $\mathrm{Aut}(\MM)$-invariant arrays for arbitrary countable (homogeneous) structures $\mathcal{M}$ is still an essentially hopeless task. However, given this framework, it is natural to focus on particular classes with large and well-behaved automorphism groups, such as homogeneous structures in a finite relational language or, more generally, $\omega$-categorical structures.\\

Going back to other connections, as pointed out in \cite[Chapter 8]{Cranebook} invariant random expansions are relevant to the study of large statistical networks, where one may want to reason probabilistically about a network given information about a base network. For example one might be given as a base network $\mathcal{M}$ a graph of friendships in an online social network, or a hypergraph joining people in a common chat by hyperedges and want to think of the actual friendship network $\mathcal{N}'$ in terms of $\mathcal{M}$ as a probability distribution on the space of expansions of $\mathcal{M}$ by a graph. It is natural to consider the given probability measure invariant under automorphisms of $\mathcal{M}$ and so to study invariant random expansions: since all knowledge we have of a set of individuals are their relationships in $\MM$, if there is an automorphism of $\MM$ sending $A$ to $A'$, the same friendship configuration must be assigned the same probability for the individuals in $A$ and for those in $A'$. It also makes sense to assume the stronger hypothesis that the measure induces the same distribution on isomorphic finite substructures as in Definition \ref{CRE} since it is natural  to work under the assumption that relations amongst non-sampled individuals do not interfere with those of the sampled individuals (cf \cite[Section 8.3.3]{Cranebook}). Indeed, some frequently studied models of large statistical networks, i.e. Erd\"{o}s-R\'{e}nyi-Gilbert models and stochastic block models \cite{SBM} can be thought of (in the limit) as invariant random expansions by graphs of homogeneous structures: respectively, $(\mathbb{N}, =)$ and an infinite set partitioned by finitely many unary predicates.\\

In light of this, Theorem~\ref{thm:IntroappliedforIKM} has some interesting implications since it yields that for various homogeneous hypergraphs, including many built by omitting non-trivial configurations, all invariant random expansions by graphs are exchangeable. This is too strong for what one would expect of real world networks: the homogeneous hypergraphs we work with are too dense and too symmetric, yielding, for example, that in all invariant random expansions of one of our $3$-hypergraphs, the probability of a triangle on top of a $3$-hyperedge is the same as that of it on top of a non-hyperedge: thinking of the $3$-hypergraph as a network of common chats, the probability that three people are friends will be independent of whether they are in a common chat or not in all distributions. Still, it is surprising that the only invariant random expansions by graphs of the hypergraphs we study are exchangeable. The issue of building good models which, unlike exchangeable graphs, adequately capture the sparsity of real-world networks, i.e. some form of "sparse exchangeable graphs", is a major challenge and topic of research in modern network analysis and in the theory of graphons \cite{BorgGraphons, Bollobassparse, BorgsparrseI, BorgAH}. Relatively exchangeable structures (and so invariant random expansions) were suggested by Crane \cite{Cranebook} as a potential way to simulate sparsity more accurately than exchangeable graph models. Given this, Theorem~\ref{thm:IntroappliedforIKM} tells us that a homogeneous structure $\mathcal{M}$ must fail the condition we call $2$-overlap closedness in order for its invariant graph expansions to yield statistical network models  which allow for non-exchangeable distributions.

\section{Invariance and exchangeability}\label{sec:invexch}
In this section, we provide some sufficient conditions on classes $\CC$ and $\CC'$ for all consistent random expansions of $\CC$ by $\CC'$ to be exchangeable, and show that these are satisfied in several contexts. Our approach is largely based on that of \cite{AKL}, which is concerned with the restricted setting where $\CC'$ is the class of finite linear orders. The full setting we work with brings some new complications. The approach has three main steps. First (Lemma \ref{lem:si}), we show that the exchangeability of consistent random expansions of $\CC$ by $\CC'$ is forced by the existence of suitable structures in $\CC$. Then (Theorem \ref{theorem:Chernoff}), we show that if $\CC'$ has a sufficiently slow growth rate, such structures can be found in $\CC$ using a ``random placement construction'' on suitably dense hypergraphs with controlled overlaps. Finally (Lemma \ref{lemma:Steiner}) we show that we can produce such hypergraphs, and check that they suffice for several classes, leading to Theorem \ref{thm:cre si}.\\

In this section, $\mathcal{C}$ and $\mathcal{C}'$ are hereditary classes of (relational) structures with (assumed to be disjoint) signature $\mathcal{L}$ and $\mathcal{L}'$ respectively. We will also assume that each relation in $\CC$ and $\CC'$ is injective, i.e. never holds with repeated entries. Since this can assumed up to quantifier-free interdefinability, it does not pose any real restriction. We denote by $\mathcal{L}^*$ the language $\mathcal{L}\cup \mathcal{L}'$ and by $\mathcal{C}^*$ the class of structures $\mathbf{A}$ such that $\mathbf{A}_{\upharpoonright \mathcal{L}} \in \mathcal{C}$ and  $\mathbf{A}_{\upharpoonright \mathcal{L}'} \in \mathcal{C}'$.

\begin{notation} For notational convenience, in the following sections, when talking about consistent random $\mathcal{C}'$-expansions of $\mathcal{C}$ we will mean $\mathrm{CRE}(\mathcal{C}, \mathcal{C}^*)$ for $\mathcal{C}^*=\mathcal{C}*\mathcal{C}'$.
\end{notation}


\begin{notation} Let $\HH, \G\in\mathcal{C}$. Let $\Theta$ be a family of embeddings of $\HH$ in $\G$. Let $\HH^*, \G^*\in\mathcal{C}^*$ be such that $\HH^*_{\upharpoonright\mathcal{L}}=\HH$ and $\G^*_{\upharpoonright\mathcal{L}}=\G$. Then, we write 
\[N_{\Theta}(\HH^*,\G^*)\]
for the number of embeddings in $\Theta$ that are also embeddings of $\HH^*$ in $\G^*$.
\end{notation}

\subsection{When invariance implies exchangeability}\label{sec:wheninvimpliesexch}

{Our first lemma in this section, that exchangeability of consistent random expansions of $\CC$ by $\CC'$ is forced by the existence of suitable structures in $\CC$, reduces verifying some universal statement about all possible expansions by $\CC'$ to an existential statement about $\CC$. 

While it may seem that a richer class should admit more consistent random expansions, the classes in Example \ref{ex:CREs} show this is not the case. When $\CC$ is the class of paths, we have a consistent random order-expansion that is not exchangeable, but when $\CC$ is the class of all graphs, \cite{AKL} shows that the only consistent random order-expansion is exchangeable. So the graphs we have added to $\CC$ have somehow forced exchangeability of consistent random expansions by linear orders. This is because for every target graph $\G$ into which a given graph $\HH$ can embed, the consistency condition yields that the distribution on expansions of $\G$ imposes additional restrictions on the distribution on expansions of $\HH$.

This in turn suggests the statement of Lemma \ref{lem:si}. We roughly want that for every pair of structures $\HH_1, \HH_2 \in \CC$ on $[n]$, there is some $\G \in \CC$ embedding both and such that for every $\CC'$-expansion of $\G$, the probability distribution induced on $\CC'$-expansions of some $n$ points of $\G$ does not depend on whether those $n$ points embed a copy of $\HH_1$ or a copy of $\HH_2$. Since the distributions of expansions of embeddings of $\HH_1$ and of $\HH_2$ are the same in $\G$, consistency will force $\P_{\HH_1} = \P_{\HH_2}$.}

The lemma is closely related to the ordering property from structural Ramsey theory (which is clearer in the setting of \cite{AKL}). A class $\CC$ satisfies the ordering property if for every $\HH \in \CC$, there is some $\G \in \CC$ such that for every linear ordering $\vec \HH$ of $\HH$ and $\vec \G$ of $\G$, there is some embedding of $\vec \HH$ into $\vec \G$. In the following lemma, rather than just requiring that a single embedding of $\HH$ into $
\G$ survives the expansion process, we make quantitative comparison of the fraction of embeddings that survive the expansion for two different base structures $\HH_1$ and $\HH_2$.

\begin{lemma}\label{lem:si} Let $\LL$ and $\LL'$ be disjoint relational languages. Let $\CC$ be a hereditary class of $\LL$-structures and $\CC'$ a hereditary class of $\LL'$-structures. Given $\mathbf{H}_1$, $\mathbf{H}_2\in\mathcal{C}[k]$, and $\epsilon>0$, suppose there is some $n$ and $\G \in \mathcal{C}[n]$ and non-empty families $\Theta_i$ of embeddings of $\mathbf{H}_i$ in $\G$ such that for all $\mathbf{H}'\in\mathcal{C}'[k]$ and $\G' \in \mathcal{C}'[n]$ we have
\[ \left|\frac{N_{\Theta_1}(\mathbf{H}_1^*,\mathbf{G}^*)}{|\Theta_1|}-\frac{N_{\Theta_2}(\mathbf{H}_2^*,\mathbf{G}^*)}{|\Theta_2|}\right| < \varepsilon,\]
where $\G^*:=\G*\G', \HH_i^*:=\HH_i*\HH'$ and $N_{\Theta_i}(\HH_i^*,\G^*)$ is the number of embeddings in $\Theta_i$ that are also embeddings of $\HH_i^*$ in $\G^*$.

Then for every consistent random expansion of $\CC$ by $\CC'$, $(\P_\A\mid \A\in\CC)$, we have that $\P_{\HH_1}(\mathbf{H}') = \P_{\HH_2}(\mathbf{H}') $ for every $\mathbf{H}'\in\mathcal{C}'[k]$. Thus, if the above is true for all $\mathbf{H}_1$, $\mathbf{H}_2\in\mathcal{C}[k]$, then every consistent random expansion of $\CC$ by $\CC'$ is exchangeable.
\end{lemma}
\begin{proof} This proof follows the same structure as \cite[Lemma 2.1]{AKL}.\\

Let $m\in \mathbb{N}$ and let $\mathbf{A}\in \CC[m]$. We denote by $\mathbf{A}^\mu$ a random element of $\mathcal{C}'[m]$ such that the expansion of $\mathbf{A}$ by $\mathbf{A}^\mu$ has distribution $\P_\A$. That is, for any $\mathbf{A}'\in\mathcal{C}'[m]$,
\[\P(\mathbf{A}^\mu=\mathbf{A}')=\P_\A(\mathbf{A}').\]
Since $(\P_\A\mid \A\in\CC)$ is a consistent random expansion, for $\mathbf{A}_1\cong \mathbf{A}_2$, $\mathbf{A}_1^\mu$ and $\mathbf{A}_2^\mu$ have the same distribution. 
We need to show that for $\HH_1, \HH_2\in\CC[k]$ and $\HH'\in\mathcal{C}'[k]$, 
\[\P_{\HH_1}(\HH')=\P(\HH_1^{\mu}=\HH')=\P(\HH_2^\mu=\HH')=\P_{\HH_2}(\HH').\]

Let $\HH_1, \HH_2\in\CC[k]$ and let $\G\in\CC[n]$ satisfy the hypotheses of the Lemma with respect to the class of embeddings $\Theta_i$ for $i\in\{1,2\}$.

Take $\phi_i$ a uniform random element of $\Theta_i$. In particular, $\phi_i$ is always an embedding of $\HH_i$ in $\G$. Let $B_i$ be the event ``$\phi_i $ is an embedding of $\mathbf{H}'$ into $\G^\mu$'', i.e. the event that in $\G^\mu$
\[\phi_i(H_i)\upharpoonright_{\mathcal{L}'}=\HH'_{\phi_i}\; ,\]
where $\HH'_{\phi_i}$ is the relabelling of $\HH'$ by $\phi_i$. Conditionally on a given $\psi_i \in \Theta_i, B_i$ is the event ``$\psi_i(\HH_i)^\mu=\HH'_{\psi_i}$'', which has the same probability as the event ``$\HH_i^\mu=\HH'$'', by the definition of a consistent random expansion. Thus we have, 
    \[\mathbb{P}(B_i|\phi_i=\psi_i)=\P(\HH_i^\mu= \HH').\] 
    Since this is true for any $\psi_i$, by the law of total probability, we get
    \begin{equation}\label{L1eq1}
        \mathbb{P}(B_i)=\P(\HH_i^\mu= \HH')=\P_{H_i}(\HH').
    \end{equation}

    On the other hand, since $\phi_i$ is uniform in $\Theta_i$, for a given expansion $\G'$ we have 
    \[ \P(B_i |\G^\mu=\G') =\frac{|N_{\Theta_i}(\mathbf{H}_i^*,\mathbf{G}^*)|}{|\Theta_i|} .\]

    Our hypothesis can therefore be rewritten as 
    \begin{equation}\label{L1eq2}
          |\P(B_1 | \G^\mu=\G') - \P(B_2|\G^\mu=\G')|\leq \varepsilon.
    \end{equation}

    We can conclude that,
    \begin{align*}
    \left| \ \P_{\HH_1}(\HH') - \P_{\HH_2}(\HH') \right| &= \left| \mathbb{P}(B_1)-\mathbb{P}(B_2) \right| \\
    &= \left| \sum_{G'\in\mathcal{C}'[n]} (\P(B_1 | \G^\mu=\G') - \P(B_2|\G^\mu=\G'))\P(\G^\mu=\G') \right|
    \\
    &\leq \sum_{G'\in\mathcal{C}'[n]} |\P(B_1 | \G^\mu=\G') - \P(B_2|\G^\mu=\G')|\P(\G^\mu=\G') \\
    &\leq \varepsilon \sum_{G'\in\mathcal{C}'[n]}\P(\G^\mu=\G')\\
    &= \varepsilon.\end{align*}
Here, the first equality is just the identity \ref{L1eq1}, whilst the second equality is just the law of total probability. Hence, we get the desired statement that $\P_{\HH_1}(\mathbf{H}') = \P_{\HH_2}(\mathbf{H}')$.
\end{proof}

{In order for $\G$ to satisfy the hypotheses of Lemma \ref{lem:si}, it will need to embed many copies of $\HH_1$ and $\HH_2$ with several overlaps between them.} We are now ready to define $k$-overlap closedness of a hereditary class $\mathcal{C}$, {which will allow us to build such $\G$ by overlapping structures from $\CC$,} and which will yield exchangeability of invariant random expansions by suitably slow growing classes $\mathcal{C}'$ in Theorem \ref{theorem:Chernoff}. In Subsection \ref{sec:koverlapclosed}, we give several examples of $k$-transitive homogeneous structures with $k$-overlap closed age, with Theorem \ref{thm:cre si} summarising some of the most interesting examples:

\begin{definition} \label{def:closed}
	Let $\CC$ be a class of relational structures with minimum arity strictly greater than $k$. We say that $\CC$ is {\bf \closed} if for every $r > k$ and arbitrarily large $n$, there exists an $r$-uniform hypergraph $\K$ on $n$ vertices satisfying the following conditions.
	\begin{enumerate}
		\item $\K$ has at least $C(r)n^{k+\alpha(r)}$ many hyperedges for some $\alpha(r) > 0$.
		\item No two $\K$-hyperedges intersect in more than $k$ points.
		\item For every $\HH_1, \HH_2 \in \CC[r]$, if copies of $\HH_1$ and $\HH_2$ are pasted into the $\K$-hyperedges, the resulting structure is completable to a structure in $\CC$ by adding relations such that no added relation is contained within a single $\K$-hyperedge.
	\end{enumerate}
To express the third point formally: for every $\HH_1, \HH_2\in\mathcal{C}[r]$ and every $\mathcal{L}$-structure $\G_0$ on $n$ vertices such that the induced substructure on each $r$-hyperedge of $\K$ is isomorphic to either $\HH_1$ or $\HH_2$, there is an injective homomorphism $g$ of $\G_0$ into a structure $\G\in\mathcal{C}[n]$ such that the restriction of $g$ to each $r$-hyperedge of $\K$ is an embedding. Figure \ref{fig:overlapclosed} may also help the reader understanding what we mean by $k$-overlap closedness.
\end{definition}

\begin{figure}[t]
\includegraphics[width=\textwidth]{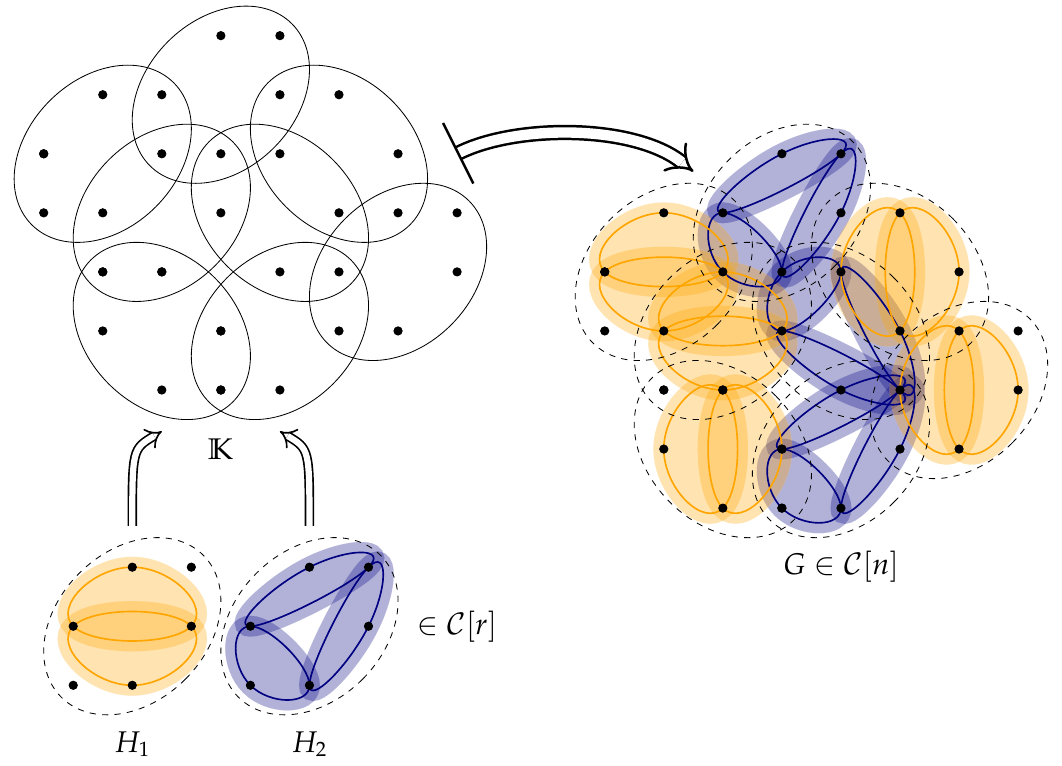}
    \caption{A pictorial representation of $k$-overlap closedness of $\mathcal{C}$: for some dense enough $r$-hypergraph $\mathbb{K}$ on $n$ vertices where no two hyperedges intersect in more than $k$ vertices, for any two $H_1, H_2\in\mathcal{C}[r]$, any way of pasting $H_1$ and $H_2$ into the $r$-hyperedges of $\mathbb{K}$ gives some $G\in\mathcal{C}[n]$ (after possibly adding some relations that do not change the induced structures on the $r$-hyperedges of $\mathbb{K}$).}
    \label{fig:overlapclosed}
\end{figure}

The last point in this definition is similar to the notion of multiamalgamation that occurs in structural Ramsey theory, as in \cite{hubicka2017ramsey}, which similarly splits into a free overlap stage and a completion stage. However, we will often consider classes that are closed under the removal of relations, in which case no completion is required in the last point.

\begin{fact}[Chernoff bound for a binomial random variable]
    Let $X$ be a binomial random variable with parameters $n, p$. Let $\mu = np$ and let $D \geq 0$. Then $\P\left(\abs{\frac{X}{\mu} -1} \geq D\right) \leq \exp{\left(-\frac1{3} D^2 \mu\right)}$.
\end{fact}

\begin{definition}\label{def:growthrate}
    Let $\CC$ be a hereditary class of finite structures. The {\bf labelled growth rate of $\CC$} is the function $f \colon \mathbb{N}^+ \to \mathbb{N} \cup \set{\infty}$ such that $f(n)$ is the number of structures in $\CC$ with domain $[n]$.
\end{definition}

The following theorem is an analogue of \cite[Theorem 5.2]{AKL}, and is similarly proved by an application of the ``random placement construction'' of \cite{nesetril1978probabilistic}, which uses the construction to prove the ordering property for hypergraphs with high girth.

\setcounter{equation}{0}
\begin{theorem} \label{theorem:Chernoff}
	Let $\CC$ be a {\closed} hereditary class and let $\CC'$ be a hereditary class of $\LL'$-structures with labelled growth rate $O(e^{n^{k+\delta}})$ for every $\delta > 0$. Then every consistent random expansion of $\CC$ by $\CC'$ is exchangeable.
\end{theorem}
\begin{proof}
	We show that $\CC$ satisfies the conditions of Lemma \ref{lem:si}. So fix $\HH_1, \HH_2 \in \CC[r]$ and $\epsilon > 0$. Then we wish to find $n$ and $\G \in \mathcal{C}[n]$ and non-empty families $\Theta_i$ of embeddings of $\mathbf{H}_i$ in $\G$ such that for all $\mathbf{H}'\in\mathcal{C}'[k]$ and $\G' \in \mathcal{C}'[n]$ we have
\[ \left|\frac{N_{\Theta_1}(\mathbf{H}_1^*,\mathbf{G}^*)}{|\Theta_1|}-\frac{N_{\Theta_2}(\mathbf{H}_2^*,\mathbf{G}^*)}{|\Theta_2|}\right| < \varepsilon.\] 
Let $\K$ be an $r$-uniform hypergraph satisfying the conditions of Definition \ref{def:closed} with $n$ vertices, for $n$ to be determined later, and $m \geq Cn^{k+\alpha(r)}$ edges. Let $\widetilde \G$ be the random $\CC$-structure obtained by pasting either $\HH_1$ and $\HH_2$ into every $\K$-hyperedge, each with probability 1/2 and via bijections chosen uniformly at random. Let $\Theta_i$ be the embeddings via which we have chosen to paste $\HH_i$ into a $\K$-hyperedge. Note that even if $\mathrm{Aut}(
\HH_i)$ is non-trivial, $\Theta_i$ still contains at most one embedding for each $\K$-hyperedge. 
	
	Now fix $\HH' \in \CC'[r]$ and $\G' \in \CC'[n]$ and let $\HH_i^* = \HH_i * \HH'$ and $\widetilde{\G}^* = \widetilde{\G} * \G'$. Let $\Theta'$ be the embeddings of $\HH'$ into $\G'$ whose image is a $\K$-hyperedge. We now consider two cases depending on the choice of $\G'$. In each case, we will show $\P$(There exists some choice of $\G'$ in that case so that our random $\widetilde \G$ fails) can be made arbitrarily small by taking $n$ sufficiently large. Thus by a union bound, $\P$(There exists some choice of $\G'$ so that our random $\widetilde \G$ fails) can also be made arbitrarily small by taking $n$ sufficiently large.

	Case 1: In $\G'$, the fraction of $\K$-hyperedges whose induced structure is isomorphic to $\HH'$ is less than $\epsilon/4$. 
	
	Note that if  $\HH_1$ and $\HH_2$ were each pasted into at least 1/4 of the $\K$-edges, then for every $\G'$ satisfying the case assumption, each of the terms in Lemma \ref{lem:si} is between 0 and $(mL_i^*\epsilon/4)/(mL_i/4) \leq \epsilon$, and so their difference is between $0$ and $\epsilon$ as well.
	
	By a Chernoff bound, the probability that the fraction of $\K$-edges replaced by copies of $\HH_1$ is less than 1/4 is at most $\exp{(-c_0 m)} \leq \exp(-c_1n^{k+\alpha(r)})$, for some constants $c_0$ and $c_1$, and similarly for $\HH_2$. So the probability that either event occurs goes to 0 as $n$ goes to infinity.

	Case 2: In $\G'$, the fraction of $\K$-hyperedges whose induced structure is isomorphic to $\HH'$ is at least $\epsilon/4$. 
	
Then $|\Theta_i|$ is binomial with parameters $m, 1/2$. Let $m' \geq m\epsilon/4$ be the number of $\K$-hyperedges in $\G'$ with induced structure isomorphic to $\HH'$. Let $\beta = |\mathrm{Aut}(\HH')|/r!$. Then $N_{\Theta_i}(\HH_i^*, \widetilde \G^*)$ is binomial with parameters $\mathrm{Binom}(m', 1/2), \beta$, which is equivalent to parameters $m', \beta/2$.
	
	
	 By Chernoff bounds
	    \begin{equation}\label{Chereq1}
     \P\left(\abs{\frac{|\Theta_i|}{m/2} -1} \geq D\right) \leq \exp{\left(-\frac1{12} D^2 m\right)}
     \end{equation}
	and
 \begin{equation}\label{Chereq2}
	\P\left(\abs{\frac{N_{\Theta_i}(\HH_i^*, \widetilde{\G}^*)}{m'\beta/2} -1} \geq D\right) \leq \exp{\left(-\frac{\beta}{12} D^2 m'\right)} \leq \exp{\left(-\frac\epsilon{48} D^2 m\right)}
 \end{equation}
	Thus
	\begin{align*}
			&\P\left(\abs{\frac{N_{\Theta_1}(\HH_1^*, \widetilde{\G}^*)}{|\Theta_1|} -\frac{N_{\Theta_2}(\HH_2^*, \widetilde{\G}^*)}{|\Theta_2|}} \geq  \frac{4D}{1-D^2} \right)\\
		\leq~ 	&\P\left(\abs{\frac{N_{\Theta_1}(\HH_1^*, \widetilde{\G}^*)}{|\Theta_1|} -\frac{N_{\Theta_2}(\HH_2^*, \widetilde{\G}^*)}{|\Theta_2|}} \geq \beta\frac{m'}{m}\frac{4D}{1-D^2} \right) \\
		=~& \P\left(\abs{\frac{N_{\Theta_1}(\HH_1^*, \widetilde{\G}^*)}{|\Theta_1|} -\frac{N_{\Theta_2}(\HH_2^*, \widetilde{\G}^*)}{|\Theta_2|}} \geq \frac{(1+D)m'\beta}{(1-D)m} - \frac{(1-D)m'\beta}{(1+D)m} \right) \\
		\leq ~ & c_0\exp{(-c_1D^2m)} \text{ for some constants $c_0, c_1$}
	\end{align*}
	where the first inequality uses that $\beta\frac{m'}{m} \leq 1$, the next equality uses that $4D/(1-D^2) = (1+D)/(1-D) - (1-D)/(1+D)$, and the last inequality uses a union bound over the failure of the events in (\ref{Chereq1}) and (\ref{Chereq2}).
	
	Let $D = \sqrt{n^{k + \delta}/m}$ for some $0< \delta < \alpha(r)$, recalling that $m \geq Cn^{k + \alpha(r)}$. Thus $\lim_{n \to \infty} D = 0$, and so $\lim_{n \to \infty} \frac{4D}{1-D^2} = 0$. 
	
	Taking another union bound (and using that there are only a constant number of options for $\HH'$)
	
	\begin{align*}
		& \P\left(\text{For some $\G'$ as in Case 2 and some $H'$ :} \abs{\frac{N_{\Theta_1}(\HH_i^*, \widetilde{\G}^*)}{|\Theta_1|} -\frac{N_{\Theta_2}(\HH_i^*, \widetilde{\G}^*)}{|\Theta_2|}} \geq  \frac{4D}{1-D^2} \right)\\
		\leq ~ & c_2 f(n) \exp{(-c_1D^2m)} = c_2f(n) \exp{(-c_1n^{k+\delta})},
	\end{align*}
 where $f(n)=|\mathcal{C}'[n]|$.
	By our assumption on $f(n)$, the last term goes to 0 as $n$ goes to $\infty$. Thus for sufficiently large $n$, we have both that$\frac{4D}{1-D^2} < \epsilon$ and the probability of failure for our random $\widetilde \G$ is arbitrarily small when considering $\G'$ in Case 2.
\end{proof}

\subsection{\texorpdfstring{{\closed}}{k-overlap closed} structures}\label{sec:koverlapclosed}

We now turn to the problem of identifying classes that are \closed. Lemma \ref{lemma:Steiner} will be used to construct the hypergraph $\mathbb{K}$ from Definition \ref{def:closed}, and afterwards we will show that this construction suffices for many classes (Corollary~\ref{cor:k+1 irred} and Lemma~\ref{lemma:k irred}).

\begin{remark}[Non-examples of $k$-overlap closedness] Let us begin with two examples of ages of $k$-transitive homogeneous $(k+1)$-ary structures which are not $k$-overlap closed. It is easy to see that linear orders are not $1$-overlap closed by looking at the case of $r=2$: any large enough graph $\mathbb{K}$ satisfying the density condition (1) in Definition \ref{def:closed} must contain a cycle, which makes it impossible to satisfy condition (3) of the definition (i.e. pasting copies of $``a<b''$ into the edges and then adding relations whilst staying in the class of linear orders). More generally, in Remark \ref{rem:macpherson}, we point out that any $\mathrm{NIP}$ $1$-overlap closed finitely homogeneous structure must have automorphism group $S_\infty$. In Subsection \ref{sec:kaygraphs}, we will look more carefully at the example of parity $(k+1)$-hypergraphs which are not $k$-overlap closed.  
\end{remark}

\begin{remark}\label{rem:partialsteiner}
The following lemma plays the role of \cite[Lemma 4.1]{AKL}. While our construction remains elementary, it seems possible that much higher-powered combinatorics may be useful in our higher-arity setting. A uniform hypergraph satisfying the second condition of Definition \ref{def:closed} is called a \emph{partial Steiner system}. The problem of constructing partial Steiner systems with many edges and possibly avoiding certain configurations is an intensively studied problem in extremal combinatorics; we refer to \cite{delcourt2022finding} or \cite{glock2024conflict} for recent results and a discussion of the long history. Most of the work in extremal combinatorics has been concerned with achieving nearly the maximum number of possible edges, which in the setting of Definition \ref{def:closed} would be $\binom{n}{k+1}/\binom{r}{k+1} = O(n^{k+1})$. However, the first condition of Definition \ref{def:closed} calls for substantially fewer edges, which allows us to forbid more configurations. 
    
\end{remark}

\begin{lemma} \label{lemma:Steiner}
	Let $r,k, J \in \mathbb{N}$ with $r > k \geq 2$. There are constants $C$ and $\epsilon > 0$ such that for all $n > r$, there is an $r$-uniform hypergraph on $n$ vertices satisfying the following conditions.
	\begin{enumerate}
		\item There are at least $Cn^{k+\epsilon}$ hyperedges.
		\item No two hyperedges intersect in more than $k$ points.
		\item For all $j \leq J$, there is no configuration of $j$ hyperedges on at most $jr -k(j-1)-1$ vertices.
	\end{enumerate}
\end{lemma}
\begin{proof}
	Let $p = 1/n^{r-(k+\epsilon)}$, where $0 < \epsilon < 1$ is to be determined later. Consider the random $r$-uniform hypergraph on $n$ vertices obtained by adding each hyperedge independently with probability $p$, so the expected number of hyperedges is $c_0n^{k+\epsilon}$.
	
For a fixed set of $k+1$ vertices, the expected number of pairs of hyperedges whose intersection contains that set is at most $n^{2(r-(k+1))}p^2$. Thus expected number of intersecting pairs on all such sets is at most 
\[ \binom{n}{k+1} n^{2(r-(k+1))} n^{-2r+2k+2\epsilon} \leq \binom{n}{k+1}n^{-2+2\epsilon} \leq n^{k-1+2\epsilon}\]
	
	And for $\epsilon < 1$, we have $k-1 + 2\epsilon < k+\epsilon$, so we can remove one hyperedge from each pair.
	
	Now we need to remove hyperedges to satisfy (3). The expected total number of such configurations involving $j$ hyperedges is at most
	
	\[\sum_{i=1}^{jr-k(j-1)-1}\binom{n}{i}\binom{i^r}{j}p^j \leq c_1n^{jr -k(j-1)-1}n^{-jr+j(k+\epsilon)} = c_1n^{k-1+j\epsilon}\]
	
	As long as $J\epsilon < 1+\epsilon$, we may remove all such configurations for all $j \leq J$ while keeping enough total hyperedges.
\end{proof}
\begin{remark} \label{rem:noCon}
    If one does not need the third condition, then the proof shows we may choose any $\epsilon < 1$, and so we obtain at least $Cn^{k+1-\delta}$ hyperedges for every $\delta > 0$.
\end{remark}

We note that the bound on the number of vertices in third point is tight. For any fixed $r \geq 3$ and $j,k \geq 2$, by a result of Brown, Erd\"os, and S\'os \cite{brown1973some} (see e.g. the abstract of \cite{shangguan2020degenerate} for a more explicit statement), the maximum number of hyperedges in an $r$-uniform hypergraph in which there is no configuration of $j$ hyperedges on at most $jr-k(j-1)$ vertices is $O(n^k)$, which conflicts with the first point. However, it may be possible to avoid other sorts of configurations in the third point, which may in turn yield further classes that are \closed.


\begin{lemma}\label{lem:gettingclosed}
	Let $\LL$ be a relational language with minimum arity strictly greater than $k \geq 1$. Let $\FF$ be a set of finite $\LL$-structures, and suppose every $F \in \FF$ satisfies that there is a subset $A$ of size $\ell > k$ such that there is some $\ell$-ary relation induced on $A$, and such that for every $b \in F \backslash A$ there are distinct  $A_0, A_1 \subset A$ each of size at least $k$ and each in a relation with $b$ (and possibly additional points). Then $\CC = \mathrm{Forb}(\FF)$ is \closed. 
\end{lemma}
\begin{proof}
	Fix $r > 0$, and let $\K$ be a hypergraph as in Lemma \ref{lemma:Steiner} with $J = 3$. Let $F$ be a forbidden structure, and suppose some copy (which we henceforth identify with $F$) can be obtained by pasting structures from $\CC[r]$ into $\K$. Let $A \subset F$ be as described, so $A$ lies in a single (unique) $\K$-hyperedge $E_1$. Since $E_1$ cannot contain the whole of $F$, there is some $b \in F$ lying outside $E_1$. Let $A_0, A_1 \subset A$ be as described, so there is a unique $\K$-hyperedge $E_2$ containing $A_0b$ and another $E_3$ containing $A_1b$. Since no two hyperedges can intersect in $k+1$ points, we have that $A_0 \not\subset E_3$ and $A_1 \not\subset E_2$. But then $E_1, E_2, E_3$ are three distinct $\K$-hyperedges on at most $3r - k - (k+1)$ points, contradicting the properties of $\K$.
\end{proof}

\begin{corollary} \label{cor:k+1 irred} \leavevmode
	\begin{enumerate}
		\item \label{it1} In a language with minimum arity strictly greater than $k$, any class defined by forbidding a set of $(k+1)$-irreducible structures is \closed.
		\item \label{it2} For $1\leq k\leq n$, $k+1$-hypergraphs forbidding $P_n^{k+1}$ are \closed.
		\item \label{it3} In a language with minimum arity strictly greater than $k$, any class with disjoint $n$\-/amalgamation for all $n$ is \closed.
	\end{enumerate}
\end{corollary}
\begin{proof} The structures in \ref{it1} and \ref{it2} satisfy the condition of Lemma \ref{lem:gettingclosed}. For the first case, consider any $(k+1)$-irreducible structure $F\in\mathrm{Forb}(\mathcal{F})$. Take $k+1$ vertices in a relation $A$ and let $b\in F\setminus A$. Take two distinct subsets of $A$ of size $k$, $A_0, A_1$. By $(k+1)$-irreducibility $bA_0$ and $bA_1$ are both in relations of $F$, and so $F$ satisfies the hypothesis of Lemma \ref{lem:gettingclosed}. Meanwhile, for $\mathcal{P}_n^k+1$, consider as $A\subset P_n^{k+1}$ any $k+1$-hyperedge. The set $A$ will contain the "centre" of $P_n^{k+1}$, i.e. the vertex $a$ which forms a hyperedge with each other set of $k$-many vertices in $A\setminus \{a\}$. Now, for every $b\in P_n^{k+1}\setminus A$, take $A_0', A_1'\subset A\setminus\{a\}$ of size $k-2$. By construction of $P_{n}^{k+1}, bA_0 a$ and $bA_1a$ form hyperedges implying that $\mathcal{P}_n^k+1$ satisfies the hypothesis of Lemma \ref{lem:gettingclosed}. Finally, for \ref{it3} we may construct a $r$-uniform hypergraph $\mathbb{K}$ as in Lemma \ref{lemma:Steiner} without the need for the third condition, as discussed in Remark \ref{rem:noCon}. After pasting copies of structures in $\mathrm{Age}(\MM)[r]$ into the the $\mathbb{K}$-hyperedges, there is no obstruction to extending the resulting structure to one in $\mathrm{Age}(\MM)$. This may require adding relations (say when working with the random tournament), but can always be done since we are working with structures satisfying disjoint $n$-amalgamation for all $n$.
\end{proof}

\begin{lemma} \label{lemma:k irred}
	Let $k \geq 2$, $\LL$ be a relational language with minimum arity strictly greater than $k$, and let $\CC$ be a hereditary class of $\LL$-structures whose minimal forbidden structures are $k$-irreducible and have size at most $N$. Then $\CC$ is $k$-overlap closed.
\end{lemma}

\begin{figure}[ht]

\begin{center}
\includegraphics[width=\textwidth]{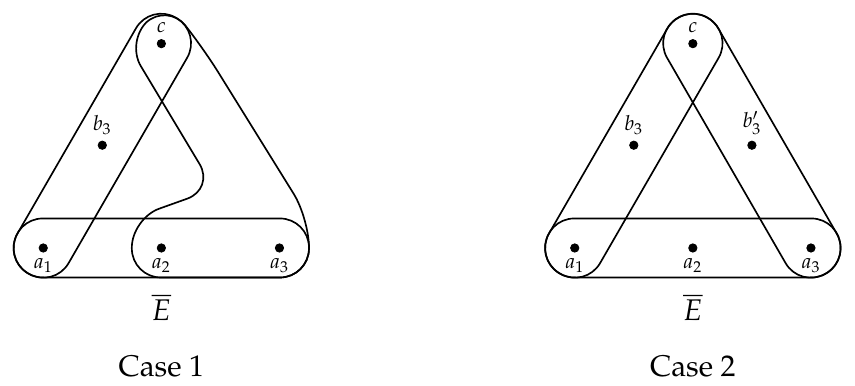}
    \caption{The two starting configurations for the case of a $3$-uniform hypergraph with $r=3$. In Case 1, $b_3$ needs to be connected to each of $\{a_2, a_3\}$, and by Lemma \ref{lemma:Steiner} (3)  each such hyperedge must contain a new point. In Case 2, $b_3$  needs to be connected to each of $\{a_2, a_3\}$, while $b'_3$ needs to be connected to each of $\{a_1, a_2\}$, and at least three such hyperedges must each contain a new point. Either case leads to indefinitely creating new hyperedges with new points until the size of the forbidden structures is exceeded.}
        \label{fig:twocases}
\end{center}
\end{figure}

\begin{proof}
	Fix $r > 0$ and let $\K$ be a hypergraph as in Lemma \ref{lemma:Steiner} for $J = N+3$. Let $F$ be a minimal forbidden structure, and suppose some copy $F^*$ can be obtained by pasting structures from $\CC[r]$ into $\K$. We will abuse terminology, and use ``$F^*$-hyperedge'' to mean some tuple of $F^*$-elements in an $\LL$-relation. Fix an $F^*$-hyperedge $E = \set{a_i}_{i \in [n]}$, and note $n \geq k+1$. For $3 \leq j \leq N+1$, we will inductively build increasing sets $K_j$ of $\K$-hyperedges. When $K_j$ is the set most recently built, we will say $v \in F^*$ satisfies $(*)$ if $v \in \bigcup K_j$, but for some distinct $\ell, m \in [n]$, no $F^*$-hyperedge containing $v$ and $a_\ell$ or $v$ and $a_m$ is a subset of an element of $K_j$. We will inductively show, for $3 \leq j \leq N+1$, that $|K_j| =j$ and one of the following two situations holds.
	\begin{enumerate}
		\item[(1)]  $|\bigcup K_j| \leq jr-k(j-1)$, $|F^* \cap \bigcup K_j| \geq j$, and there is some $b_j \in \bigcup K_j$ satisfying $(*)$
		\item[(2)] $|\bigcup K_j| \leq jr-k(j-1)+1$, $|F^* \cap \bigcup K_j| \geq j+1$, and there are distinct $b_j, b'_j \in \bigcup K_j$ satisfying $(*)$
	\end{enumerate}
	
	Note that this will imply that $|F^*| \geq N+1$, which is a contradiction.

	For the base case, let $j=3$. Let $\overline E$ be the unique $\K$-hyperedge containing $E$; we will additionally ensure that $\overline E \in K_3$. Since $\overline E$ cannot contain all of $F^*$, let $c \in F^*$ lie outside $\overline E$. We will let $\binom{[n]}{k-1}$  denote the $(k-1)$-element subsets of $[n]$, and given $I \subset [n]$ will let $A_I = \set{a_i | i \in I} \subset E$. By $k$-irreducibility, there are (not necessarily distinct) $F^*$-hyperedges $E_I \supset cA_I$ for each $I \in \binom{[n]}{k-1}$, each contained in a unique $\K$-hyperedge $\overline E_I$. Note that no $\overline E_I$ can intersect $\overline E$ in $k+1$ vertices. There are now two subcases, but first we mention an observation we will repeatedly use: given $K_j$, if a vertex $v$ is contained in a unique element $\overline {E'}$ of $K_j$ and $\overline {E'}$ intersects $E$ in at most $k-1$ vertices, then $v$ satisfies $(*)$.\\

 Now, for the subcases (illustrated in Figure \ref{fig:twocases} in a simplified context):
	
	Case 1: Suppose there is $M \in \binom{[n]}{k-1}$ such that $\overline E_M$ intersects $\overline E$ in $k$ vertices. Let $\ell \in [n]$ be such that $a_\ell \not\in \overline E_M$, and let $L \in \binom{[n]}{k-1}$ contain $\ell$. If $\overline E_L \cap (\overline E \cup \overline E_M)$ contains more than just $cA_L$, then $\set{\overline E, \overline E_M, \overline E_L}$ form three $\K$-hyperedges on at most $3r - 2k-1$ vertices, contradicting the construction of $\K$. So let $d \in E_L \bs cA_L$. Let $K_3 = \set{\overline E, \overline E_M, \overline E_L}$ and let $b_3 = d$. 
	These satisfy the conditions in (1).
	
	Case 2: Suppose there is no $M \in \binom{[n]}{k-1}$ such that $\overline E_M$ intersects $\overline E$ in $k$ vertices. For each $I \in \binom{[n]}{k-1}$, let $d_I \in E_I \bs cA_I$. If there are distinct $K, L, M \in \binom{[n]}{k-1}$ such that the intersection of each pair from $E_K, E_L, E_M$ contains some element outside $cE$, then $\set{\overline E, \overline E_K, \overline E_L, \overline E_M}$ form four $\K$-hyperedges on at most $4r-3k-1$ points, contradicting the construction of $\K$. So without loss of generality, suppose $\overline E_K \cap \overline E_L \subset cE$. Let $d_1 \in \overline E_K \bs cE$ and let $d_2 \in \overline E_L \bs cE$. Then let $K_3 = \set{\overline E, \overline E_K, \overline E_L}$, $b_3 = d_1$, and $b'_3 = d_2$. These satisfy the conditions in (2).\\

	For the inductive step, suppose we have built $K_m$ (containing $\overline E$) and wish to build $K_{m+1}$.
	
	Case 1: $K_m$ satisfies the conditions in (1) and contains $\overline E$.
	
	For $I \in \binom{[n]}{k-1}$, let $E_I$ be an $F^*$-hyperedge containing $b_mA_I$. Since $b_m$ satisfies $(*)$, there is some $J \in \binom{[n]}{k-1}$ such that $E_J$ is not a subset of any element of $K_m$. Let $d \in E_J \bs b_mA_J$ and let $\overline E_J$ be the $\K$-hyperedge containing $E_J$. If $\overline E_J$ intersects $\bigcup K_m$ in more than just $b_mA_J$,  then $K_m \cup \set{\overline E_J}$ form $m+1$ $\K$-hyperedges on at most $mr - k(m-1)+r-(k+1) = (m+1)r -km-1$ vertices, contradicting the construction of $\K$. So let $K_{m+1} = K_m \cup \set{\overline E_J}$ and let $b_{m+1} = d$. These satisfy the conditions in (1).

	Case 2: $K_m$ satisfies the conditions in (2) and contains $\overline E$.
	
	For each $I \in \binom{[n]}{k-1}$, there are (not necessarily distinct) $F^*$-hyperedges $E_I \supset b_mA_I$, each contained in a unique $\K$-hyperedge $\overline E_I$. By $(*)$ least two are not contained in $K_m$, which we call $\overline E_J$ and $\overline E_L$. 
	
	Case 2a: Suppose $\overline E_J = \overline E_L$. If $\overline E_J$ intersects $\bigcup K_m$ in more than just $b_mA_{J \cup L}$, then $K_m \cup \set{\overline E_J}$ form $m+1$ $\K$-hyperedges on at most $mr - k(m-1)+1+r-(k+2) = (m+1)r -km-1$ vertices, contradicting the construction of $\K$. In particular, $b'_m \not\in \overline E_J$. So let $K_{m+1} = K_m \cup \set{\overline E_J}$ and let $b_{m+1} = b'_m$. Then these satisfy the conditions in (1).

	Case 2b: Suppose $\overline E_J \neq \overline E_L$. If $\overline E_J$ intersects $\bigcup K_m$ in more than just $b_mA_J$ and $\overline E_L$  intersects $\bigcup K_m$ in more than just $b_mA_L$, then $K_m \cup \set{\overline E_J, \overline E_L}$ form $m+2$ $\K$-hyperedges on at most $mr - k(m-1)+1 + 2r -2(k+1) = (m+2)r -k(m+1)-1$ vertices, contradicting the construction of $\K$. So without loss of generality, suppose $\overline E_J \cap \bigcup K_m = b_mA_J$. Let $d \in E_J \bs b_mA_J$. Then let $K_{m+1} = K_m \cup \set{\overline E_J}$, $b_{m+1} = d$, and $b'_{m+1} = b'_m$. These satisfy the conditions in (2).
\end{proof}

\begin{remark} \label{rem:4aryExample}
   Already for finitely bounded free amalgamation classes of 4-uniform hypergraphs, proving 3-overlap closedness runs into an issue. We obtain a contradiction to the construction of $\K$ as in Lemma \ref{lemma:Steiner} if there are three $\K$-hyperedges on 5 vertices. Consider the 4-uniform hypergraph $F$ on vertices $\set{a_1, a_2, a_3, a_4, b, c}$ with hyperedges $E_0 = \set{a_1, a_2, a_3, a_4}$, $E_1 = \set{b, c, a_1, a_2}$, and $E_2 = \set{b, c, a_3, a_4}$. This is 2-irreducible but has 6 vertices. So the conditions in Lemma \ref{lemma:Steiner}, which we recall are in a sense tight, cannot handle the class $\mathrm{Forb}(F)$.

   More explicitly, if $\K$ is an $r$-uniform hypergraph, it may be that $r = 4$. Then $\K$ as constructed in Lemma \ref{lemma:Steiner} may contain copies of $F$, and pasting a hyperedge into every $\K$-hyperedge will yield a copy of $F$ in the resulting hypergraph.
\end{remark}

\begin{theorem} \label{thm:cre si}
Let $k \geq 1$, $\LL$ be a relational language with minimum arity strictly greater than $k$, and let $\CC$ be a hereditary class of $\LL$-structures. 
Suppose that the minimal forbidden structures of $\CC$ consist of a set of $(k+1)$-irreducible structures, and (when $k \geq 2$) $k$-irreducible structures of bounded size.
Let $\CC'$ be a hereditary class of $\LL'$-structures with labelled growth rate $O(e^{n^{k+\delta}})$ for every $\delta > 0$. Then every consistent random expansion of $\CC$ by $\CC'$ is exchangeable.
\end{theorem}
\begin{proof}
    The proofs of Corollary \ref{cor:k+1 irred} and Lemma \ref{lemma:k irred} straightforwardly combine to show $\CC$ is \closed. The result then follows by Theorem \ref{theorem:Chernoff}.
\end{proof}

\begin{remark} \label{rem:VCn}
In some cases, we can weaken the hypotheses of Theorem \ref{thm:cre si} by allowing $\CC'$ to have faster growth rate. For example, if $\CC$ is the class of all $(k+1)$-uniform hypergraphs, then by Remark \ref{rem:noCon} we may expand by any $\CC'$ whose labeled growth rate is $O(e^{n^{k+1-\delta}})$ for some $\delta > 0$.
\end{remark}

We finish this section by giving classes of expansions where our Theorems can be applied.

\begin{fact}\label{factgrowth}
    Let $\mathcal{L}'$ be a finite relational language whose relations have arity $\leq k$. Then the labelled growth of all $\mathcal{L}'$-structures is $o(e^{n^{k+\delta}})$ for any $\delta>0$
\end{fact}

\begin{proof}
    Take $m$ the size of $\mathcal{L}'$. Without loss of generality, we assume that all relations in $\mathcal{L}$ are of arity $k$. There are $n^k$ tuples of size $k$, each having $2$ possibilities for each element of $\mathcal{L}'$. Therefore, there are at most $2^{mn^k}$ $\mathcal{L}'$ structures of size $n$.
\end{proof}

Therefore, from Theorem \ref{theorem:Chernoff}  and Fact \ref{factgrowth}, we get

\begin{theorem}\label{thm:appliedforIKM}
    Let $k \geq 2$, and $\CC$ be a {\closed} hereditary class. Let $\CC'$ be a hereditary class of $\LL'$-structures where $\LL'$ consists of relations of arity $\leq k$. Then every consistent random expansion of $\CC$ by $\CC'$ is exchangeable.
\end{theorem}

Finally, we note that the results of this section can be applied locally to a given arity, or even to a pair of structures of a given arity. We illustrate with the following proposition, which gives some evidence that the conclusion of Theorem \ref{thm:cre si} might hold assuming only that $\CC$ has free amalgamation.

\begin{prop}\label{prop:edgecase}
    Let $\CC = \mathrm{Forb}(\FF)$ be a free amalgamation class whose relations have arity strictly greater than $k$. Let $\HH_1, \HH_2 \in \CC[k+1]$. Let $\CC'$ be a hereditary class of $\LL'$-structures with labelled growth rate $O(e^{n^{k+\delta}})$ for every $\delta > 0$. Then in every consistent random expansion $\mu$ of $\CC$ by $\CC'$, we have (in the notation of Definition \ref{CRE}) $\mathbb{P}_{H_1} = \mathbb{P}_{H_2}$.
\end{prop}
\begin{proof}
    It suffices to show that $\CC$ satisfies the conditions of being $k$-overlap closed only for $r = k+1$, rather than for every $r > k$. Given $F \in \FF$ with relations only of arity $k+1$, if $|F| \neq k+1$ then by 2-irreducibility $F$ is not $(k+1)$-partite, i.e. the points of $F$ cannot be colored with $k+1$ colors so that no two points in a tuple belonging to some relation receive the same color. Thus we may take $\mathbb K$ to be the complete $(k+1)$-uniform $(k+1)$-partite hypergraph with $n$ vertices in each part.
\end{proof}

\begin{remark}\label{rem:ipk} For those interested in higher-arity model theoretic properties: one can prove the conclusion of Proposition \ref{prop:edgecase} also for a homogeneous $(k+1)$-hypergraph or $(k+1)$-hypertournament (in the sense of \cite{cherlin2021ramsey}) whose relation has $\mathrm{IP}_k$ by \cite[Proposition 5.2]{ndependence}. For $k>1$, we only mention $\mathrm{IP}_k$ in passing in Question \ref{q:ipk}, and refer the reader to \cite{ndependence} for more on this property. The independence property $\mathrm{IP}$ (i.e. $\mathrm{IP}_1)$ is given in Definition \ref{def:ip}.
\end{remark}


\subsection{Some exchangeable invariant random expansions}\label{sec:applications}

In this section, we give some explicit applications of Theorem \ref{thm:cre si} to certain random expansions of the structures from Example \ref{ex:hypergraphs}. We also  discuss how results from exchangeability theory, and in particular the Aldous-Hoover Theorem and the work of Ackerman, Freer, and Patel \cite{AFP}, can be used to characterise these expansions and provide interesting examples. These results will have immediate applications to the study of invariant Keisler measures in Subsection \ref{classtime}, where invariant random expansions by lower arity relations, often with certain forbidden configurations, naturally arise.\\


\begin{notation} Writing $\mathcal{P}([k])$ for the set of subsets of $[k]$, let $f:[0,1]^{\mathcal{P}([k])}\to\{0,1\}$ be a function. We say that $f$ is \textbf{symmetric} if for any permutation of $[k]$, $\tau\in S_k$, and for all $(a_I)_{I\in\mathcal{P}([k])}$, 
\[f((a_I)_{I\in\mathcal{P}([k])})=f((a_{\tau I})_{I\in\mathcal{P}([k])}),\]
where for $I\subseteq [k]$, $\tau I$ is the subset of $[k]$ obtained by permuting the elements of $I$ by $\tau$.
\end{notation}

There is an easy way to build exchangeable $k$-hypergraphs from a symmetric Borel function $f:[0,1]^{\mathcal{P}([k])}\to\{0,1\}$. Let $(\xi_I\vert I\subseteq \mathbb{N}, |I|\leq k)$ be i.i.d. $\mathrm{Uniform}[0,1]$ random variables. Consider the Borel probability measure $\mu$ on the space of $k$-ary hypergraphs with relation $R$, obtained by setting, for $A\in[\mathbb{N}]^k$,
\begin{equation}\label{AH}
    A\in R \text{ if and only if } f((\xi_I)_{I\in\mathcal{P}(A)})=1.
\end{equation}
It is easy to see that $\mu$ is indeed exchangeable. The Aldous-Hoover theorem tells us that any exchangeable hypergraph is obtained in this way.

\begin{fact}[Aldous-Hoover Theorem {\cite{AldousT, HooverT}}]\label{AldousHoover}
Let $\mu'$ be an exchangeable $k$-hypergraph. Then, there is a symmetric Borel function $f:[0,1]^{\mathcal{P}([k])}\to\{0,1\}$ such that $\mu'$ has the same distribution as the exchangeable $k$-hypergraph $\mu$ built according to (\ref{AH}).
\end{fact}

The standard construction of the random graph by tossing coins for each pair of vertices in $\mathbb{N}$ given in Example \ref{ex:IREs} (a) can be seen as being exactly of this form. Some technical twists are needed to represent arbitrary exchangeable structures (where the relations may not be symmetric or injective). We refer the reader to \cite[S.2.1]{CraneTown} for a discussion of these details (see also \cite[Section 2.4]{AFKrP}). However, all exchangeable structures have a representation essentially of this form, which we call an Aldous-Hoover-Kallenberg representation.\\

As we just mentioned, there is an exchangeable random graph concentrating on the isomorphism type of $\mathcal{R}_2$. It was much harder to prove that there are exchangeable random graphs concentrating on the isomorphism type of the generic $K_n$-free graphs \cite{PetrovVershik}. The isomorphism type of a given countable $\mathcal{L}$-structure is a Borel set invariant under the action of $S_\infty$ on $\mathrm{Struc}_{\mathcal{L}'}(\mathbb{N})$. So, it makes sense to ask for which countable structures we can actually find an exchangeable structure concentrating on their isomorphism type. This was explored in the work of Ackerman, Freer, and Patel \cite{AFP} who characterise when an exchangeable structure concentrates on a given isomorphism type:

\begin{fact}[{\cite{AFP}}]\label{AFPfact} Let $\mathcal{N}$ be a countable $\mathcal{L}'$-structure. Then, there is an exchangeable $\mathcal{L}'$-structure concentrating on the isomorphism type of $\mathcal{N}$ if and only if $\mathcal{N}$ has trivial group-theoretic definable closure, i.e. for every finite $A\subseteq N$, the points of $A$ are the only points of $N$ fixed by the pointwise stabilizer of $A$, $\mathrm{Aut}(\mathcal{N}/A)$.
\end{fact}
Note that if $\mathcal{N}$ is homogeneous, trivial group-theoretic definable closure corresponds to trivial algebraic closure. So, there will be an exchangeable structure concentrating on the isomorphism type of a homogeneous structure $\mathcal{N}$ if and only if $\mathrm{Age}(\mathcal{N})$ has the disjoint amalgamation property.\\

Hence, under the hypotheses of Theorem \ref{theorem:Chernoff}, Fact \ref{AFPfact} yields a description of when an invariant random expansion of $\mathcal{M}$ concentrates on the isomorphism type of a given expansion $\mathcal{M}^*$:

\begin{theorem} Let $\mathcal{M}$ be a {\closed} $\mathcal{L}$-structure and $\mathcal{M}^*$ be an expansion of $\mathcal{M}$ by $\mathcal{L}'$-relations so that the growth-rate of structures in $\mathrm{Age}(\MM^*)\upharpoonright_{\mathcal{L}'}$ is $O(e^{n^{k+\delta}})$ for all $\delta>0$. Then, there is an invariant random expansion of $\mathcal{M}$ concentrating on the isomorphism type of $M^*$ if and only if $\mathcal{M}^*=\mathcal{M}*\mathcal{N}$ and $\mathrm{Age}(\MM^*)=\mathrm{Age}(\MM)*\mathrm{Age}(\mathcal{N})$, where $\mathcal{N}$ is an $\mathcal{L}'$-structure with trivial group-theoretic definable closure. 
\end{theorem}

 This description is especially important as \cite[Theorem 1.5]{IREs} proves that for an $\omega$-categorical structure $M$ with no algebraicity, weak eliminations of imaginaries, and $\mathrm{Aut}(\MM)\neq S_\infty$, every ergodic invariant random expansion of $M$ either concentrates on an orbit (i.e. the isomorphism type of some expansion of $M$) or is essentially free in the sense that it assigns measure $0$ to every orbit of $\mathrm{Aut}(\MM)$ on $\mathrm{Struc}_{\mathcal{L}'}(M)$.\footnote{Interestingly, this fails for $\mathrm{Aut}(\MM)=S_\infty$ \cite[Example 3.5]{AFKwP}.}\\


Let us now discuss some applications of Theorem \ref{theorem:Chernoff}. Firstly, from Theorem \ref{thm:appliedforIKM}, we can see that the invariant random expansions of many homogeneous $r$-hypergraphs by $(r-1)$-hypergraphs are exchangeable. 

\begin{corollary}\label{cor:irehyper}
Let $\mathcal{H}$ be an $r$-uniform hypergraph from Example \ref{ex:hypergraphs}, other than a universal homogeneous parity $k$-hypergraph. Then $\mathcal{H}$ is $(r-1)$-overlap closed, and so invariant random expansions by the class of all $(r-1)$-uniform hypergraphs are exchangeable.
\end{corollary}

\begin{remark}\label{moreoncorirehyper} By Theorem \ref{thm:cre si}, the above corollary really holds for any $r$-uniform homogeneous $r$-hypergraph whose age is of the form $\mathrm{Forb}(\mathcal{F})$, where $\mathcal{F}$ is either a set of $r$-irreducible structures, or a finite set of $r-1$-irreducible structures. This covers all finitely bounded homogeneous $3$-hypergraphs with free amalgamation.
\end{remark}

\begin{remark} The work of Crane and Towsner \cite{CraneTown} and Ackerman \cite{AckerAut} already yields exchangeability for invariant random  expansions by $(r-1)$-uniform hypergraphs of the random $r$-hypergraph $\mathcal{R}_r$. Our results are novel for all of the other hypergraphs in Corollary \ref{cor:irehyper} and all of the non-random hypergraphs mentioned in Remark \ref{moreoncorirehyper}. In particular, they cannot be recovered from the weak representation theorem in \cite[Theorem 3.15]{CraneTown}. This can be seen by considering the universal homogeneous parity $k$-hypergraph, which has non-exchangeable invariant random graph expansions (see \cite[Example 3.9]{CraneTown} and Subsection \ref{sec:kaygraphs}). 
\end{remark}

\begin{remark} Even in the case of $\mathcal{R}_r$, we obtain some results not clear from the work of Crane and Towsner and Ackerman \cite{CraneTown, AckerAut}. Namely, the exchangeability of expansions by certain classes of structures with maximum arity larger than $r-1$. Note that Remark \ref{rem:VCn} shows that for $\mathcal{R}_r$, any expansion by a class with labelled growth rate $O(e^{n^{r-\delta}})$ for some $\delta > 0$ will be exchangeable. Among relational classes with maximum arity $r$, such classes are characterized as those having finite $VC^*_{r-1}$-dimension in \cite{terry2018VC}.
\end{remark}

 As we pointed out earlier, our results recover the work of Angel, Kechris, and Lyons \cite{AKL} proving uniqueness of random orderings for various free amalgamation structures. In addition to this, there are many other interesting classes with slow growth rate for which we obtain exchangeability results when looking at invariant random expansions of structures with free amalgamation:

\begin{corollary}
Any transitive structure $\mathcal{M}$ such that $\mathrm{Age}(\MM)$ has free amalgamation, e.g. $\mathcal{M} = \mathcal{R}_r$ for $r \geq 2$, is 1-overlap closed. Thus, any invariant random expansion of $\mathcal{M}$ by a class $\mathcal{C}'$ with labelled growth rate $O(e^{n^{1+\epsilon}})$ for every $\epsilon > 0$, such as the class of finite linear orders, must be exchangeable. 
	\end{corollary}

 Again, in the case of $\mathcal{R}_r$, we may take $\mathcal{C}$ to be any binary class of bounded VC-dimension, which includes many well-studied graph classes such as any (semi-)algebraic graph class \cite{nguyen2023induced}. In this corollary, we already see that exchangeability can be extremely restrictive in some cases. For example, in the introduction we noted there is a unique exchangeable linear order. We may also take $\mathcal{C}$ to be a suitably slow-growing graph class, such as the class of interval graphs, whose limits are studied in \cite{diaconis2013interval}, the class of planar graphs, or the class of graphs of degree at most $d$ for some fixed $d$. However, in these last two examples, the number of edges in a graph with $n$ vertices is $o(n^2)$, and so there is a unique exchangeable graph with ages in these classes: the countable independent set. This is because non-empty exchangeable graphs are dense (see \cite[Section 6.5.1]{Cranebook}).\\

 When $\mathcal{C}$ is the age of a homogeneous structure in a finite language, by a theorem of Macpherson \cite[Theorem 1.1-1.2]{MacNIPgrowth}  the growth rate condition of $O(e^{n^{1+\epsilon}})$ for every $\epsilon >0$ is equivalent to $\mathrm{Flim}(\mathcal{C})$ being NIP (i.e. $\mathcal{C}$ having bounded $\mathrm{VC}$-dimension). This is a model theoretic property further discussed in Subsection \ref{sec:modelprelims}. So, for example, we get exchangeability for invariant random expansions of homogeneous structures with free amalgamation by ages of NIP homogeneous structures such as homogeneous $B, C,$ or $D$-relations (see \cite[Chapter 12]{Infpermnotes}, and \cite{adeleke1998relations}), or by Droste's $2$-homogeneous semilinear orders \cite{droste1985structure} (which are finitely homogenizable, in the sense that they have an expansion which is homogeneous in a finite relational language \cite[Example 6.1.2]{Homogeneous}). We also get uniqueness for the invariant random expansions of $1$-overlap closed homogeneous structures by the age of any structure interdefinable with one of the five reducts of  $\mathbb{Q}$ (cf. \cite{AFKwP}).

 \begin{remark}\label{rem:macpherson} Macpherson's theorem \cite[Theorem 1.1-1.2]{MacNIPgrowth} also tells us that $1$-overlap closed finitely homogeneous structures whose automorphism group is not $S_\infty$ have the independence property $\mathrm{IP}$ (i.e. they are not $\mathrm{NIP}$). In fact, any homogeneous structure $\mathcal{M}$ has an invariant random expansion by $\mathrm{Age}(\MM)$ which concentrates on the underlying copy of $\mathcal{M}$, which, unless $\mathrm{Aut}(\MM)=S_\infty$,  will be non-exchangeable. If $\mathcal{M}$ is $\mathrm{NIP}$ and finitely homogeneous, $\mathrm{Age}(\MM)$ has labelled growth rate $O(e^{n^{1+\epsilon}})$ for every $\epsilon >0$. Thus, if it was $1$-overlap closed, by Theorem \ref{theorem:Chernoff}, the aforementioned invariant random expansion by $\mathrm{Age}(\MM)$ would have to be exchangeable, implying that $\mathrm{Aut}(\MM)=S_\infty$.
 \end{remark}

Our results also yield exchangeability when we look at invariant random expansions omitting certain configurations. For example, we may want to study invariant random expansions of the universal homogeneous tetrahedron-free $3$-hypergraph $\mathcal{H}_4^3$ by graphs which omit a triangle on top of hyperedges. Since this class of expansions is contained in the class of expansions of $\mathcal{H}_4^3$ by graphs, Corollary \ref{cor:irehyper}, already gives us exchangeability. In particular, the measure of a triangle will have to be zero both over a hyperedge and over a non-hyperedge even if \textit{a priori} we allowed for the latter to be positive. This will be useful in Subsection \ref{classtime}, where we classify invariant Keisler measures.

\begin{corollary}\label{cor:tetIRE}
	The invariant random expansions of $\mathcal{H}_n^r$ by the class of $(r-1)$-uniform hypergraphs concentrating on expansions that never expand a copy of $K_{n-1}^{r}$ by a copy of $K_{n-1}^{r-1}$ is the class of exchangeable $K_{n-1}^{r-1}$-free $(r-1)$-uniform hypergraphs.
	\end{corollary}
\begin{proof} We have exchangeability by Corollary \ref{cor:irehyper}. Since the measure of the expansion of $K_{n-1}^{r}$ by a copy of $K_{n-1}^{r-1}$ is $0$, by exchangeability we must assign measure $0$ to any expansion of a substructure of $\mathcal{H}_n^r$ of size $n-1$ by $K_{n-1}^{r-1}$. Any invariant random expansion of $\mathcal{H}^r_n$ by $K_{n-1}^{r-1}$-free $(r-1)$-uniform hypergraphs will satisfy the condition on omitted structures in the expansion. 
\end{proof}

\begin{remark} By \cite{PetrovVershik} and \cite{AFP}, there is an exchangeable structure concentrating on the isomorphism type of the universal homogeneous $K_{n-1}^{r-1}$-free $(r-1)$-uniform hypergraph. There are many more such exchangeable classes of $(r-1)$-uniform hypergraphs, e.g. all exchangeable $(r-1)$-partite $(r-1)$-uniform hypergraphs. Moreover, from \cite{AFKwP}, we know there are $2^{\aleph_0}$ ergodic measures concentrating on the isomorphism types of these structures. The condition on omitted substructures can be much more restrictive on how many exchangeable structures there are, as we will see in Corollary \ref{cor:strangemeas}.
\end{remark}

In the next subsection, we will see a more interesting example where forbidding certain configurations leads to a unique non-exchangeable invariant random expansion of the specified type.
\subsection{When invariance does not imply exchangeability: parity $k$\-/hypergraphs}\label{sec:kaygraphs}

In this section, we give a family of examples of structures and associated IREs that are not exchangeable. They are given by the invariant random expansions of the universal homogeneous parity $k$-hypergraphs $\mathcal{G}_k$ by their space of hypergraphings (see Definition \ref{def:hypergraphing}). In particular, we prove that there is a unique invariant random expansion of $\mathcal{G}_k$ to this class, which is non-exchangeable. It was already noted in \cite[Example 3.9]{CraneTown} that for $\mathcal{G}_3$ there were non-exchangeable random graph expansions. 
The results of Subsection \ref{sec:wheninvimpliesexch} gives us a good heuristic for why this is the case: the labelled growth rate of the class of parity $k$-hypergraphs is similar to that of the class their hypergraphings. Proving uniqueness gives another interesting application of techniques from \cite{AKL} and will be helpful later in the Section on invariant Keisler measures in showing the peculiar behaviour of the Keisler measures of $\mathcal{G}_k$ (see Corollary \ref{cor:kayIKM}).

\begin{definition}\label{def:hypergraphing}
    Let $k\in \mathbb N$, and $\mathbf{H}$ a $k$-uniform hypergraph, a parity $k+1$-hypergraph is a $k+1$-uniform hypergraph such that the number of hyperedges on any subset of the vertices of size $k+2$ has the same parity as $k$.
    
    Given $\mathbf{H}$ a $k$-uniform hypergraph, we define  $\mathbf{G}$ the parity $k+1$-hypergraph associated to $\mathbf{H}$ as the $k+1$-hypergraph satisfying: a hyperedge $(x_1,\ldots,x_{k+1})$ is present in $\mathbf{G}$ iff the number of hyperedges in $\mathbf{H}$ with domain in $\{x_1,\ldots,x_{k+1}\}$ has the same parity as $k+1$. In such case, we say that $\mathbf{H}$ is compatible with (or is a hypergraphing of) $\mathbf{G}$. We also say that $\mathbf{G}$ is a reduct of $\mathbf{H}$ and denote by $\mathrm{red}_{\mathcal{C}_{k+1}}(\mathbf{H})$.
\end{definition}

For the rest of this section, we fix $k\geq 3$ and denote by $\mathcal{C}_{k}$ the class of parity $k$-hypergraphs.

For a given parity $k$-hypergraph $\mathbf{H}$, we denote by $\mathcal{C}_k^*(\mathbf{H})$ the set of compatible $(k-1)$\-/hypergraphs. The universal homogeneous parity $k$-hypergraph $\mathcal{G}_k$  introduced in Examples \ref{ex:hypergraphs} is the parity $k$-hypergraph associated with the universal homogeneous $(k-1)$-hypergraph $\mathcal{R}_{(k-1)}$.
 We will construct an IRE of $\mathrm{Aut}(\mathcal{G}_k)$ concentrated on the space of compatible $k$-uniform hypergraphs. Take $a\in \mathbb N$, for each $x_2,\ldots,x_k$ we put an hyperedge between $a,x_2,\ldots,x_k$ with probability $p$. The other hyperedges are added to ensure that the obtained random hypergraph to be compatible with $\binom{n}{n\delta}$. When $p=1/2$, the obtained random hypergraph is $\mathrm{Aut} (\binom{n}{n\delta})$-invariant. The aim of the rest of the section is to prove that this the only such random compatible hypergraph.

We prove  :

\begin{theorem}\label{thm:kaygraphs}
     Let $k\geq 2$, for all $\mathbf{A}\in \mathcal{C}_{k+1}$, there is a number $\rho(\mathbf{A})=2^{-\binom{\ell-1}{k-1}} \geq 0$, where $\ell$ is the size of $\mathbf{A}$, such that for all $\varepsilon>0$ there is $\mathbf{B} \in \mathcal{C}_{k+1}$ in which $\mathbf A$ embeds such that for any $\mathbf{A}^*\in \mathcal{C}_{k+1}^*(\mathbf{A}) $ and $\mathbf{B}^*\in \mathcal{C}_{k+1}^*(\mathbf{B})$ we have
\[\left|\frac{\left|N(\mathbf{A}^*,\mathbf{B}^*) \right|}{\left|N(\mathbf{A},\mathbf{B}) \right|}- \rho(\mathbf{A})\right|\leq \varepsilon.
\]
\end{theorem}

This in particular implies that there is a unique $\mathrm{Aut}(\mathcal{G}_k)$-invariant random compatible hypergraph, either by adapting the proof of Lemma \ref{lem:si} or the proof of Lemma 2.1. in \cite{AKL}.

We will use McDiarmid's inequality from \cite{McDiarmid}.

\begin{fact}[\cite{McDiarmid}]\label{Thm:McDiarmid} Let $n\in \N$,  $Z=(Z_1,\ldots,Z_n)$ be a family of independent identically distributed random variables on $\{0,1\}$ and $f\colon \{0,1\}^n \to \R$ such that there is a family $(a_i)_{i\leq n}\in \R^n$ that verifies $|f(z)-f(z')|\leq a_i$ whenever $z(k)=z'(k)$ for $k\neq i$ and $z(i)=1-z'(i)$. Then, for all $L>0$, we have

\[\mathbb{P}(|f(Z)-\mathbb{E}(f(Z))|\geq L)\leq 2\exp\left(-\frac{2L^2}{\sum_{i=1}^n a_i^2}\right).\]
\end{fact}

\begin{proof}[Proof of Theorem \ref{thm:kaygraphs}]
Let us fix $\mathbf{A}\in \mathcal{C}_{k+1}$ with $\ell$ vertices. The aim of the proof is to construct a $\mathbf{B}\in \mathcal{C}_{k+1}$ on $n\geq \ell$ vertices such that for any $\mathbf A^* \in \mathcal{C}_{k+1}^*(\mathbf{A})$ and $\mathbf{B}^*\in \mathcal{C}_{k+1}^*(\mathbf{B})$, we have:

\[\left|\frac{\left|N(\mathbf{A}^*,\mathbf{B}^*) \right|}{\left|N(\mathbf{A},\mathbf{B}) \right|}-\frac{1}{2^{\binom{\ell-1}{k-1}}}\right| \leq C\sqrt{\frac{\log(n)}{n}}\]
where $C$ is a constant depending only on $k$. 

We take $X$ a uniform random $k$-uniform hypergraph on a vertex set $V$ of size $n$, i.e for any $k$ vertices of $V$, we put an hyperedge between them with probability $1/2$, independently for each pair of vertices. We will prove that with non-zero probability, the reduct of $X$ is as required.

Let $N(\mathbf{A},X)$ denote the number of embeddings of $\mathbf{A}$ in $\mathrm{red}_{\mathcal{C}_{k+1}}(X)$. We now show that $E:=\mathbb{E}[N(\mathbf{A},X)]=2^{\binom{\ell-1}{k-1}}n(n-1)...(n-\ell+1)2^{-{\ell \choose k}}$. Indeed, fix an embedding $\varphi$ of the domain of $\mathbf A$ in $V$ and $\mathbf C$ a hypergraphing of $\mathbf A$. The probability that $\varphi$ is an embedding of $\mathbf C$ in $X$ is $2^{-{\ell \choose k}}$. There are $ n(n-1)...(n-\ell+1)$ possible $\varphi$ and $2^{\ell-1\choose k-1}$ possible $\mathbf{C}$. By summing over $\varphi$ and $\mathbf B$ we have the result.

We define
\[f(\mathbf{A},X)=\frac{N(\mathbf{A},X)}{E}\]
which is a function of $n \choose k$ independent identically distributed variables, each indicating the absence or presence of an edge in $X$. Adding or removing an edge to $X$ changes $N(\mathbf{A},X)$ by at most $\ell(\ell-1)...(\ell-k) (n(n-1)...(n-\ell+k+1))$. Indeed, this counts every possible embedding using this specific $k$-tuple of vertices. Therefore $f$ satisfies the conditions of Fact \ref{Thm:McDiarmid} with $a_i=c_in^{-k}$, where $c_1$ (as well as all the $c_j$ we will  define in the rest of the proof) is a positive constant depending only on $k$ and $\ell$. We therefore have, for any $D>0$,
\[\mathbb{P}(|f(\mathbf{A},X)-1|\geq D) \leq 2\exp\left(-\frac{2D^2}{{n \choose k} c_1^2 n^{-2k}}\right) \leq \exp\left(-c_2 D^2n^k\right).\]

Let us now set $\mathbf{A}^*$ a hypergraphing of of $\mathbf{A}$ and $X^*$ a hypergraphing of $\mathrm{red}_{\mathcal{C}_{k+1}}(X)$. We define $N^*(\mathbf{A}^*,X^*)$ to be the number of embeddings of $\mathbf{A}^*$ in $X^*$, and
note that $\mathbb{E}(N^*(\mathbf{A}^*,X^*))=\frac{E}{2^{\ell-1 \choose k-1}}$. We define
\[f^*(\mathbf{A}^*,X^*)=\frac{N(\mathbf{A}^*,X^*)}{E}.\]

Here, adding or removing an edge to $X$ changes $N^*(\mathbf{A}^*,X^*)$ by at most $\ell(\ell-1)...(\ell-k) (n(n-1)...(n-\ell+k+1))$. So as before, we have

\[\mathbb{P}\left(|f^*(\mathbf{A}^*,X^*)-\frac{1}{2^{{\ell-1 \choose k-1}}}|\geq D\right) \leq 2\exp\left(-\frac{2D^2}{{n \choose k} c_3 n^{-2k}}\right) \leq \exp\left(-c_4 D^2n^k\right).\]

Summing over all possible hypergraphings, we have that except with probability $c_5 2^{{n-1 \choose k-1}} \allowbreak \exp(-D^2c_6n^k)$, we have simultaneously 
\[|f(\mathbf{A},X)-1|<D \]
and
\[|f^*(\mathbf{A}^*,X^*)-\frac{1}{2^{{\ell-1 \choose k-1}}}|<D \]
for all expansions of $\mathbf{A}$ and $X$.

We choose $D=c_7\sqrt{\frac{\log(n)}{n}}$  with $c_7$ chosen so that $c_5 2^{{n-1 \choose k-1}}\exp(-c_6D^2 n^k)<1$ for all $n\geq 1$. This implies that there exists a (deterministic) graph $\widetilde{\mathbf{B}}$ satisfying all the above inequalities simutaneously. If we denote by  $\mathbf{B} \in \mathcal{H}$ its reduct, then we have 
\[\left|\frac{|N(\mathbf{A},\mathbf{B})|}{E}-1\right|<D \]
and
\[\left|\frac{|N(\mathbf{A}^*,\mathbf{B}^*)|}{E}-\frac{1}{2^{{\ell-1 \choose k-1}}}\right|<D, \]
for all $\mathbf{A}^*\in \mathcal{H}^*(\mathbf{A})$ and $\mathbf{B}^*\in \mathcal{H}^*(\mathbf{B})$.

The required inequality then follows. Indeed, for all $\mathbf{A}^*\in \mathcal{H}^*(\mathbf{A})$ and $\mathbf{B}^*\in \mathcal{H}^*(\mathbf{B})$:

\begin{align*}
\left|\frac{|N(\mathbf{A}^*,\mathbf{B}^*)|}{|N(\mathbf{A},\mathbf{B})|}-\frac{1}{2^{{\ell-1 \choose k-1}}}\right| \leq& \left|\frac{|N(\mathbf{A}^*,\mathbf{B}^*)|}{|N(\mathbf{A},\mathbf{B})|}-\frac{|N(\mathbf{A}^*,\mathbf{B}^*)|}{E}\right| \\ &+ \left| \frac{|N(\mathbf{A}^*,\mathbf{B}^*)|}{E} - \frac{1}{2^{{\ell-1 \choose k-1}}}\right|\\
\leq&\left|\frac{|N(\mathbf{A}^*,\mathbf{B}^*)|}{|N(\mathbf{A},\mathbf{B})|}-\frac{|N(\mathbf{A}^*,\mathbf{B}^*)|}{E}\right| + D\\
\leq&  \frac{|N(\mathbf{A}^*,\mathbf{B}^*)|}{|N(\mathbf{A},\mathbf{B})|}\cdot \left|\frac{|N(\mathbf{A},\mathbf{B})|}{E}-1\right| + D \\
\leq& 2D \\
=& 2c_7\sqrt{\frac{\log(n)}{n}}.
\end{align*}
\end{proof}

\begin{remark} For the universal homogeneous parity $3$-hypergraph, a more model-theoretic version of this proof is given in \cite[Section 4.5.3]{ColinPhD} (where it is refered to as the $2$-graph) and independently in \cite[Theorem 7.4.11]{myPhD} (where it is refered to as the two-graph). The proof relies on the fact that, fixing a vertex, $\mathcal{G}_3$ canonically embeds a copy of the random graph, so any invariant random expansion $\mu$ of $\mathcal{G}_3$ by graphings induces an invariant unary random expansion $\nu$ of $\mathcal{R}_2$, and these are well-understood. Uniqueness then follows given some additional equations that $\mu$ must satisfy which force a unique choice of $\nu$. This technique should also generalise to the universal homogeneous parity $k$-hypergraphs, though some of our original motivation for looking at parity $k$-hypergraphs for $k\geq 3$ comes from the fact that \textit{a priori} more measures might have arisen in this context: for example, for $\mathcal{G}_4$, the step of looking at the measure induced by fixing a vertex would yield a binary random invariant expansions of $\mathcal{R}_3$, and this space is much wilder than unary invariant random expansions of $\mathcal{R}_2$.
\end{remark}

\begin{remark}
    It is easy to see that the universal homogeneous parity $k$-hypergraph is not $k$-overlap closed. However it is $(k-1)$-overlap closed, being a reduct of the universal homogeneous $k$-hypergraph, which is $(k-1)$-overlap closed.
\end{remark}
\section{The connection with invariant Keisler measures}\label{sec:IKMconnection}

In this section we show how invariant Keisler measures of homogeneous structures can be viewed as a special case of invariant random expansions. This allows us to describe the spaces of invariant Keisler measures of various homogeneous structures which were not previously understood. We view the space of types of a $\mathcal{L}$-homogeneous structure as an expansion by a language $\mathcal{L}'$ which associates to each $\mathcal{L}$-relation finitely many $\mathcal{L}'$-relations of strictly lower arity (Lemma \ref{Mpstar}). This generalises some observations of Ensley \cite{EnsleyPhD} and yields a correspondence between invariant Keisler measures of $\mathcal{M}$ and a particular space of invariant random expansions of $\mathcal{M}$ (Corollary \ref{meascorr}), which is discussed in Subsection \ref{howto}. The fact that the expansion is by relations of lower arity allows us to use Theorem \ref{thm:appliedforIKM} when $\mathrm{Age}(\MM)$ is $k$-overlap closed, concluding that all invariant Keisler measures for $M$ in the variable $x$ are exchangeable (Theorem \ref{thm:IKMareSI}). In particular, the space of invariant Keisler measures for $\mathcal{M}$ in the variable $x$ corresponds to a space of $\mathrm{S}_\infty$-invariant measures concentrating on a particular age (Corollary \ref{invcorr}). This correspondence is described in detail in Subsection \ref{homcont}. This allows us to describe the spaces of invariant Keisler measures for many homogeneous structures in Subsection \ref{classtime} and obtain some further model theoretic applications in Subsection \ref{FOtime}.\\

Below we discuss previous research on Keisler measures under different model theoretic assumptions in order to explain the relevance of Keisler measures to modern model theory and put the progress that we make in context. An understanding of model theoretic properties is not required to follow the correspondence between invariant Keisler measures and invariant random expansions discussed in Subsections \ref{howto} and \ref{homcont}. A more formal introduction to model theoretic properties is given in Subsection \ref{sec:modelprelims}, but it is important to keep in mind that relatively technical and sophisticated concepts trivialise in the examples and classes of structures that we study. \\

\begin{figure}[t]
    \centering
    \includegraphics[width=0.7\textwidth]{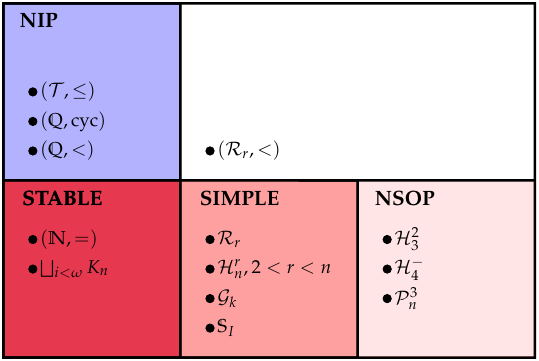}
    \caption{Simplified map of the model-theoretic universe. Note that $\mathrm{stable}\subsetneq\mathrm{simple}\subsetneq\mathrm{NSOP}$ and $\mathrm{NIP}\cap\mathrm{NSOP}=\mathrm{stable}$. The map includes examples of homogeneous (or finitely homogenisable) structures in the various classes. We already introduced the strictly simple and strictly $\mathrm{NSOP}$ ones in Example \ref{ex:hypergraphs}.  In the strictly $\mathrm{NIP}$ part of the universe we can see the  dense linear order $(\mathbb{Q},<)$, the cyclical order $(\mathbb{Q}, \mathrm{cyc})$, and Droste's $2$-homogeneous semilinear orders \cite{droste1985structure}, denoted by $(\mathcal{T}, \leq)$. In the stable part of the universe, we have an infinite set with equality and the countable disjoint union of copies of the complete graph $K_n$. An example of a homogeneous structure outside all of these classes is universal homogeneous ordered $r$-hypergraph $(\mathcal{R}_r, <)$. }
    \label{fig:universe}
\end{figure}

Model theorists study the behaviour of Keisler measures in  relation to model theoretic properties. These are properties of theories which characterise some combinatorial features of definable sets in their models. They usually yield various tools for the study of a theory such as tameness conditions on definable groups, or the existence of good notions of independence or dimensions. Our main focus in this paper lies in the following model theoretic properties: stability, NIP (not independence property), simplicity, and NSOP (not strict order property). In Figure \ref{fig:universe}, we illustrate the intersections and containments between these properties and list some $\omega$-categorical examples. These properties are also meaningful for many natural structures from algebra, number theory, and topology, some of which we mention below. We refer the curious reader to \url{https://www.forkinganddividing.com/} for a more extensive list of examples. Moreover, there are several resources on stable, NIP and simple theories that the reader may consult \cite{TZ, NIP, Kimsimp, Wagner:ST}. Strictly NSOP theories (i.e. theories which are NSOP and not simple) are in general not well-understood model theoretically though there is a hierarchy within them of $\mathrm{NSOP}_n$ theories for $n<\omega$, and much current research in the field has been dedicated to them for $n\leq 4$ \cite{chernikov2016model, mutchnik2024conantindependencegeneralizedfreeamalgamation}.\\

Keisler measures are best understood in NIP theories \cite{KeislerNIP, NIPinv, NIP, StarchNIP, GroupmeasuresNIP, Genericallystable}. For the purposes of our exposition, NIP theories may be regarded as the "non-random" part of the model theoretic universe since 
Keisler measures can be locally approximated by a weighted sum of types, which, under tame assumptions,  can be chosen to be invariant when the measure is invariant \cite[Lemma 4.8]{NIPinv}. Even stronger results hold for NIP $\omega$-categorical structures (see Section \ref{sec:nip}, and \cite{Ensley}).\\ 

Outside of an NIP context, Keisler measures are relatively poorly understood \cite{lots_of_authors, Keislerwild}, though, positive results have been obtained for arbitrary theories under the assumption of amenability (i.e. types over $\emptyset$ can be extended to global invariant Keisler measures) \cite{FOAmen, krupinski2019amenability}. Nevertheless, well-behaved invariant Keisler measures naturally appear in the context of many simple theories \cite{MS, WolfMEC}, such as pseudofinite fields \cite{HrushPseud, CVDMI} and smoothly approximable structures~\cite{FSFT}. Strictly simple theories are often considered as having some additional amount of "randomness" to stable theories (which are the intersection of simple and NIP theories).\\ 


Finally, Keisler measures have several applications in obtaining strong regularity lemmas under model theoretic tameness conditions such as stability, NIP and their higher arity generalisations \cite{MalliarisReg, DominationRegularity, CherStarchNIP, RegularityNIP, StarchNIP, GroupStableReg, chernikov2020hypergraph}. Moreover, the good behaviour of Keisler measures in pseudofinite fields was used in \cite{PillayStar} for a model theoretic proof of the algebraic regularity lemma for graphs definable in finite fields \cite{TaoReg} (see also \cite{dzamonja2022graphonsarisinggraphsdefinable}). This was further generalised in \cite{AlexisARL} to a hypergraph algebraic regularity lemma using model-theoretic techniques (cf. Appendix B in \cite{AER}).\\


In spite of the relevance of Keisler measures to modern model theory, the project of describing the space of invariant Keisler measures of a given class of theories outside of an NIP context was essentially abandoned after the work of Albert and Ensley \cite{Albert, Ensley, EnsleyPhD}. This is particularly surprising considering that the space of invariant Keisler measures for a structure is understood as a natural generalisations of its space of invariant types, which is usually easy to understand and heavily studied. However, as we explain more in Section \ref{sec:consequences}, their techniques cannot work for the structures we study. On one hand, it is well-known that Keisler measures will in general behave differently from how they do in NIP theories~\cite{Keislerwild}. On the other, Albert's techniques do not properly extend outside of a binary context. Even the most recent techniques (e.g. Theorem B.8 in \cite{AER}) require a structure to satisfy some higher independent amalgamation property, and so only help with structures with disjoint $n$-amalgamation for all $n$ in our context (due to \cite{Palacinrandom}). Indeed, even after noticing the correspondence we build in Subsections \ref{howto} and \ref{homcont} between invariant Keisler measures and invariant random expansions, characterising a given space of invariant random expansions requires results such as our exchangeability result in Theorem \ref{thm:appliedforIKM}, or, even for the random hypergraphs, the recent results of Crane and Towsner, and Ackerman \cite{CraneTown, AckerAut}.\\

For this reason, the correspondence we build, together with Theorem \ref{thm:appliedforIKM} opens up a large class of theories whose space of invariant Keisler measures can now be understood. This is explored in Section \ref{sec:consequences} and Subsection \ref{classtime}. Then in Subsection \ref{FOtime}, we will see how our results have strong implications for our understanding of invariant Keisler measures in simple theories.


\subsection{Preliminaries in model theory}\label{sec:modelprelims}

In this section, we give some preliminaries in model theory. The reader should keep in mind that (i) the material of Subsections \ref{howto} and \ref{homcont} requires only basic knowledge of model theory, such as Chapters 1-4 of \cite{TZ}; (ii) in all examples we consider, sophisticated model theoretic concepts trivialise; hence, (iii) until Section \ref{sec:nip} at the end of the paper, almost no background in model theory is needed except for the arguments explaining why certain model theoretic concepts trivialise in our context and examples. We keep the technical arguments for model theorists, but a reader unfamiliar with the discipline may skip these and lose little understanding of our ideas. Nevertheless, model theoretic concepts are central to the motivation of the following sections because the new understanding of invariant Keisler measures that we achieve offers many ideas and challenges for future research in the field: on one hand we get many positive results which manage to classify and show exchangeability for the invariant Keisler measures of many structures. On the other, we offer counterexamples to many reasonable things one could have expected of them. We use the introductions to the various subsections to overview how our work interacts with the rest of the literature in this respect.\\

Given a structure $\mathcal{M}$ and $A \subseteq M$, a \textbf{type over A in the variable $x$} is a maximal set of formulas with free variable $x$ and parameters from $A$ such that every finite subset is satisfied in $M$. In our usual setting where $M$ is a relational Fra\"{i}ss\'{e} limit, such a type simply describes all the relations holding and not holding between $x$ and $A$ while avoiding the forbidden substructures determining $\mathrm{Age}(\MM)$.\\

For a cardinal $\kappa$, we say that a structure $\mathcal{M}$ is $\bm{\kappa}$\textbf{-saturated} if all types over sets of cardinality $<\kappa$ are realised. We say that it is \textbf{strongly} $\bm{\kappa}$\textbf{-homogeneous} if partial elementary maps between subsets of $\mathcal{M}$ of cardinality $<\kappa$ extend to automorphisms of $\mathcal{M}$. The countable model of an $\omega$-categorical structure is always $\omega$-saturated and strongly $\omega$-homogeneous, and so is any homogeneous structure. In general, to introduce model theoretic concepts, we work in a sufficiently $\kappa$-saturated and strongly $\kappa$-homogeneous model $\mathbb{M}$ (see \cite[Section 6.1]{TZ}). We say that a subset $A\subseteq \mathbb{M}$ is \textbf{small} when it is of cardinality $<\kappa$ and in general, we will be working with small subsets of $\mathbb{M}$.\\

Imaginaries will be relevant in Section \ref{sec:nip} and in some of the discussion of Section \ref{sec:consequences}. The idea is to expand $\mathcal{M}$ by sorts representing the equivalence classes of $\mathcal{M}$ with respect to definable equivalence relations. The resulting structure, $\mathcal{M}^{eq}$ is usually a natural object for expressing model theoretic results at the correct level of generality. Often, one also needs to consider hyperimaginaries, i.e., objects standing for the quotients of $\mathcal{M}$ with respect to type-definable equivalence relations. The structures we study eliminate hyperimaginaries and weakly eliminate imaginaries, meaning that these objects will not be relevant to most of our proofs \cite[Theorem 1.1]{Freeam}. Still, when we mention $\mathrm{acl}^{eq}(\emptyset)$ and $\mathrm{dcl}^{eq}(\emptyset)$, we will be talking about algebraic closure and definable closure in $\mathcal{M}^{eq}$. We refer the reader to \cite[Chapter 16]{Poizat} on imaginaries and weak elimination of imaginaries and to \cite[Chapter 4]{Kimsimp} on hyperimaginaries.\\

In Subsection \ref{FOtime}, we explore the relation between two notions of smallness for definable sets in a structure. One of them consists in a definable set $X$ being \textbf{universally measure zero}, i.e. it is assigned measure zero by every invariant Keisler measure in $\mathbb{M}$. The other consists in a definable set forking over $\emptyset$. We define this below:

\begin{definition} We say that a formula $\phi(x,b)$ \textbf{divides} over $A$ if there is some $k$ and a sequence $(b_i|i<\omega)$ of realisations of $\mathrm{tp}(b/A)$ such that $\{\phi(x, b_i)\vert i<\omega\}$ is $k$-inconsistent (i.e. all $k$-subsets of $\{\phi(x, b_i)\vert i<\omega\}$ are inconsistent). We say that A partial type $\pi(x)$ (or a formula) \textbf{forks} over $A$ if it implies a disjunction of formulas dividing over $A$. The 
definable subsets  of $\mathbb{M}^{|x|}$ whose defining formula forks over $A$ form an ideal in the Boolean algebra $\mathrm{Def}_x(\mathbb{M})$. We write $F_x(A)$ for the \textbf{ideal of definable sets forking over} $A$ in the variable $x$.
\end{definition}

The idea behind dividing is that a definable set which is "tiny" in the model may be moved around by automorphisms so that it does not overlap with itself. Since we expect the union of two tiny sets to also be tiny and the disjunction of two dividing formulas does not necessarily divide, we define forking to close this notion of tininess under disjunction. From the point of view of Boolean algebra, $F_x(A)$ consists of the ideal generated by formulas dividing over over $A$ in $\mathrm{Def}_x(\mathbb{M})$. Forking yields a notion of independence which is central to model theory:

\begin{definition} For $a, b, A$ small sets from $\mathbb{M}$, we write $a\indep{A}{}b$ if $\mathrm{tp}(a/Ab)$ does not fork over $A$. We say that $a$ is non-forking independent from $b$
 over $A$ and call this notion of independence non-forking independence.
 \end{definition}

In particular, forking independence is central to the study of \textbf{simple} theories. These are theories for which forking independence is symmetric, i.e. $a\indep{B}{} c$ if and only if $c\indep{B}{} a$. More generally, by the Kim-Pillay theorem \cite{KPsimp}, simple theories can be characterised by a notion of independence satisfying various nice conditions, including the independence theorem over models, which we may consider as a weaker form of disjoint $3$-amalgamation, where disjointness has been replaced by independence. The simple theories we study in this paper satisfy some of the most desirable properties in this class. In particular, they are all supersimple of $\mathrm{SU}$-rank $1$ (see \cite{Kimsimp} for the relevant definitions). Moreover, forking independence is trivial for them:

\begin{definition} We say that a theory has \textbf{trivial forking} if
\[A\indep{C}{}B \text{ if and only if } (A\cup C)\cap (B\cup C)=C.\]
\end{definition}
A weaker property than trivial forking is being \textbf{one-based}, which, under elimination of hyperimaginaries means that non-forking independence corresponds to \textbf{weak algebraic independence}: $A\indep{C}{} B$ if and only if $\mathrm{acl}^{eq}(AC)\cap\mathrm{acl}^{eq}(BC)=\mathrm{acl}^{eq}(C)$.

Every $\omega$-categorical structure with disjoint amalgamation and $3$-DAP has trivial forking and the other aforementioned tameness conditions:

\begin{fact}[{\cite[Theorem 3.14]{Kruck}}]\label{trivialforkfact} Suppose that $\mathcal{C}$ has $k$-DAP for $k\leq 3$ and $\mathcal{M}=\mathrm{Flim}(\mathcal{C})$ is $\omega$-categorical. Then, $T$ is supersimple, of $\mathrm{SU}$-rank $1$ and with trivial forking.
\end{fact}

This also holds for the Fra\"{i}ss\'{e} limits of classes omitting $3$-irreducible structures:

\begin{fact}[{Theorem 7.22 in \cite{Freeam}, see \cite[Remark 7.3]{Kopconstr}}] \label{fact:3irred} Let $\mathcal{L}$ be a finite relational language. Let $\mathcal{C}=\mathrm{Forb}(\mathcal{F})$ and $\mathcal{M}$ be the associated Fra\"{i}ss\'{e} limit. If every structure in $\mathcal{F}$ is $3$-irreducible, then $\mathcal{M}$ is supersimple, with $\mathrm{SU}$-rank $1$ and trivial independence.
\end{fact}

More generally, Conant \cite{Freeam} studies model theoretic properties in structures with free amalgamation and proves that various complex model theoretic concepts trivialise in this context. In particular, all such structures are either supersimple of $\mathrm{SU}$-rank $1$ with trivial forking, or strictly $\mathrm{NSOP}_4$. The latter model theoretic property is a strengthening of $\mathrm{NSOP}$ which belongs to a hierarchy of $\mathrm{NSOP}_n$ properties. The model theory of $\mathrm{NSOP}_4$ is not yet fully understood. Some of the examples we discuss, such as the generic $K_n$-free graphs and the $n$-petal-free homogeneous $3$-hypergraphs, belong to this class.

\begin{fact}[Theorem 1.1 in {\cite{Freeam}}]\label{freefact} Let $\mathcal{M}$ be a countable homogeneous structure in a finite relational language whose age has free amalgamation and let $T=\mathrm{Th}(M)$. Then, $T$ has weak elimination of imaginaries. Moreover, 
\begin{itemize}
    \item EITHER: $T$ is supersimple of $\mathrm{SU}$-rank $1$ and with trivial forking;
    \item OR: $T$ is strictly $\mathrm{NSOP}_4$, i.e. it is $\mathrm{NSOP}_4$ and $\mathrm{SOP}_3$.
\end{itemize}
\end{fact}
\begin{proof} The observation that forking is trivial in the simple context follows from the fact that in $T$ is one-based \cite[Theorem 1.1]{Freeam}, which under elimination of hyperimaginaries, weak elimination of imaginaries and trivial algebraicity implies actual triviality of forking. 
\end{proof}

Theorem 1.1 in \cite{Freeam} actually yields even more model theoretic information about these structures which is not needed for this paper. A nice consequence is that simple structures with free amalgamation satisfy disjoint $3$-amalgamation:

\begin{fact}\label{3dapfree} Suppose that $\mathcal{M}$ is homogeneous in a finite relational language, with free amalgamation and simple. Then, $\mathrm{Age}(\MM)$ satisfies disjoint $3$-amalgamation.
\end{fact}
\begin{proof} We use Fact \ref{freefact}. Simple theories satisfy independent $3$-amalgamation over Lascar strong types. As mentioned earlier, this can be seen as a weak form of disjoint $3$-amalgamation where disjointness is replaced by independence and quantifier-free types are replaced by Lascar strong types. However, in our context, these notions trivialise so that $\mathrm{Age}(\MM)$ does indeed satisfy disjoint $3$-amalgamation. In fact, by $\omega$-categoricity and \cite[Corollary 5.3.5]{Kimsimp}, $\mathcal{M}$ satisfies independent $3$-amalgamation over strong types. By weak elimination of imaginaries, $\mathcal{M}$ satisfies it over algebraically closed sets. By triviality of algebraic closures, $\mathcal{M}$ satisfies independent $3$-amalgamation over finite sets. By triviality of forking, $\mathrm{Age}(\MM)$ satisfies disjoint $3$-amalgamation.
\end{proof}



\begin{remark}[Simplicity of some of the hypergraphs from Example \ref{ex:hypergraphs}]\label{rem:simplicityex} For $2<r<n$, $\mathcal{R}_n, \mathcal{H}_n^r, \mathcal{H}_{r+1}^{-}, \mathcal{G}_{(r+1)},$ and $\mathcal{P}_{n}^{r+1}$ are supersimple of $\mathrm{SU}$-rank $1$ and with trivial forking. For  $\mathcal{R}_n, \mathcal{P}_{n}^{r+1}, \mathcal{H}_{r+1}^{-}$ and $\mathcal{H}_n^r$ this just follows from Fact \ref{fact:3irred}. There is some history of interest in the hypergraphs $\mathcal{H}_n^r$ as homogeneous supersimple one-based structures failing disjoint $n$-amalgamation \cite{HrushPseud, BinKop}. The universal homogeneous parity $k$-hypergraphs $\mathcal{G}_k$ are simple with trivial forking since, as noted in  Remark \ref{rem:ndaphyper}, they have disjoint amalgamation and disjoint $3$-amalgamation. It is also possible to see these structures are supersimple of $\mathrm{SU}$-rank $1$ and with trivial algebraic closure since  since $\mathcal{G}_{k}$ is a reduct of $\mathcal{R}_{(k-1)}$ \cite[Fact 2.7]{Kopconstr}, by considering the definable $k$-hyperedge relation in $\mathcal{R}_{(k-1)}$ which holds of $k$ many vertices if and only if they have a number of $(k-1)$-hyperedges with the same parity as $k$ \cite{Thomashyp}. Parity $k$-hypergraphs are also further studied in \cite{ApproxRamsey} (where they are called kay-graphs).
\end{remark}

Whilst for a homogeneous structure with free amalgamation, forbidding $3$-irreducible substructures implies simplicity, the converse is not necessarily the case. In Theorem 7.22 of \cite{Freeam}, Conant proves that if $\mathrm{Age}(\MM)$ is obtained by forbidding injective homomorphisms from some class $\mathcal{F}$ chosen to be minimal with respect to this notion, then all structures in $\mathcal{F}$ being $3$-irreducible is equivalent to simplicity of $\mathcal{M}$ when $\mathcal{M}$ has free amalgamation. Still, this is a more restricted class than that simple homogeneous structures with free amalgamation. Hence, proving a homogeneous structure with free amalgamation is not simple might require a bit more effort. We give an example of this by proving the hypergraphs $\mathcal{P}_n^3$ are strictly $\mathrm{NSOP}_4$.

\begin{lemma}\label{lem:petalnsop} For $n\geq 3$, the $n$-petal free $3$-hypergraph $\mathcal{P}_n^3$ does not satisfy disjoint $3$-amalgamation and so is strictly $\mathrm{NSOP}_4$  by Facts \ref{freefact} and \ref{3dapfree}.
\end{lemma}
\begin{proof} Note that $P_n^3$ is $2$-irreducible and so $\mathcal{P}_n^3$ has free amalgamation. We show that it has a disjoint $3$-amalgamation problem with no solution. Consider a base $A$ of $n-2$ vertices $a_0, B_1, \dots, B_{n-2}$, where $B_k=\{b_k^1, \dots, b_k^{n-2}\}$ for $1\leq k\leq n-3$. Consider extra vertices $c_1, c_2,\dots, c_n$ and let 
\[A_{12}=A c_1 c_2, A_{13}=A c_1 (c_i\vert 3\leq i\leq n), \text{ and } A_{23}=A c_2 (c_i\vert 3\leq i\leq n).\]
We ask that $A_{12}$ has exactly one hyperedge consisting of $a_0c_1c_2$. Meanwhile, for $j\in\{1,2\}$, $A_{j3}$ has exactly the following hyperedges: for each $3\leq i\leq n$, $c_i$ forms a hyperedge with every pair in $B_{i-2}c_j$, and $a_0$ forms a hyperedge with each pair in $c_j(c_i\vert 3\leq i\leq n)$. Now, $(A_I\vert I\in [3]^2)$ forms a disjoint $3$-amalgamation problem. We claim that it has no solution. Any solution to the problem must have a $3$-hyperedge between some of the $c_i$ for $i\leq n$ since otherwise, $a_0(c_i\vert i\leq n)$ forms a copy of $P_n^3$. This hyperedge must contain $c_1c_2$ since the amalgamation problem has already specified whether there are relations or not for triplets not containing both of these vertices. So, without loss of generality, $c_1c_2c_3$ form a hyperedge in the solution to the $3$-amalgamation problem. However, now, $c_3$ forms a relation with every pair in $B_1c_1c_2$ and the amalgamation problem has specified that there are no other relations in $B_1c_1c_2$, yielding that $B_1c_1c_2c_3$ forms a copy of $P_n^3$ contradicting that we found a solution to the disjoint $3$-amalgamation problem. Hence, the $3$-amalgamation problem given by $(A_I\vert I\in [3]^2)$ has no solution, meaning that $\mathcal{P}_n^3$ does not have disjoint $3$-amalgamation and so is strictly $\mathrm{NSOP}_4$ by Facts \ref{freefact} and \ref{3dapfree}.
\end{proof}

We conclude this section by briefly mentioning two other model theoretic properties: stability and NIP.

\begin{definition} A formula $\phi(x,y)$ is stable if there is no infinite sequence $(a_i b_i)_{i<\omega}$ such that $\mathbb{M}\vDash \phi(a_i, b_j)$ if and only if $i<j$. A theory is stable if all of its formulas are stable.
\end{definition}

Stability is generally considered the ideal model theoretic property and model theorists have several results both for stable theories and for stable formulas in arbitrary theories \cite{Geomstab}. For an invariant Keisler measure $\mu$, the relation $R(a,b):=\mu(\phi(x, a)\wedge\psi(x, b))=\alpha$ is stable as an $\mathrm{Aut}(\mathbb{M})$-invariant relation \cite[Proposition 2.25]{approxsubg}, and so as a definable relation in an $\omega$-categorical setting. This means that local stability has several applications to the study of invariant Keisler measures \cite{GroupStableReg, PillayStar}. A problem we address in this paper is that of understanding measures of more complex intersections such as $\mu(\phi(x, a,b)\wedge \psi(x, a,c)\wedge \chi(x, bc))$ for which fewer tools are currently available. This is further discussed in the introduction to Section \ref{sec:consequences}.\\

Both simplicity and NIP are generalisations of stability which try to preserve some of its results. Simple theories generalise the particularly nice theory of non-forking independence of stable theories whilst allowing some amount of randomness. NIP theories preserve the "non-randomness" of stable theories, but allow some order to exist in the structure. Below we give the definition of NIP, though the only property of NIP theories we will need in this paper is isolated in Corollary \ref{cor:NIPIEP}.

\begin{definition}\label{def:ip} A formula $\phi(x, y)$ has the independence property $\mathrm{IP}$, if there are $(a_i)_{i<\omega}\allowbreak (b_I)_{I\subseteq \omega}$ such that 
\[\mathbb{M}\vDash \phi(a_i, b_I) \text{ if and only if } i\in I.\]
We say that a formula is $\mathrm{NIP}$ otherwise. A theory is $\mathrm{NIP}$ if all of its formulas are $\mathrm{NIP}$.
\end{definition}

In addition to the Examples given in Figure \ref{fig:universe}, real closed fields and algebraically closed valued fields are $\mathrm{NIP}$. Some additional homogeneous examples are given by $B, C,$ and $D$-relations (see \cite{Infpermnotes}). 




In Section \ref{sec:nip}, we provide some results describing invariant Keisler measures in NIP $\omega$-categorical structures which complement our work in simple theories and extend previous work of Ensley \cite{Ensley}.

\subsection{How to view a type as an expansion}\label{howto}
In this section, we show how we may consider invariant Keisler measures as a special case of invariant random expansions. We begin with a brief discussion of invariant Keisler measures and proceed by explaining how we may view the space of non-realised types $S'_x(M)$ as a space of expansions of $\mathcal{M}$. Throughout this section we denote by $\mathcal{M}$ a
strongly $\omega$-homogeneous $\mathcal{L}$-structure not assumed to be countable. The results in this and the following section hold in general with variable $\overline{x}$ of arbitrary arity. Nevertheless, it will be easier for notational convenience to work with a singleton variable.\\

As a motivating example, recall the universal homogeneous graph $\mathcal{R}_2$. A 1-type $p(x)$ over $\mathcal{R}_2$ describes which points of $\mathcal{R}_2$ are adjacent to $x$. This may be encoded by an expansion of $\mathcal{R}_2$ by a unary predicate $U$, where $U(a)$ holds if $p(x)$ makes $x$ adjacent to $a$. Similarly, for $\mathcal{R}_3$, 1-types can be encoded by expanding by a binary (graph) relation. Given this correspondence between types and expansions, we also obtain a correspondence between invariant measures on these spaces, i.e. between invariant Kiesler measures and invariant random expansions. In this subsection and the next, we formalize this correspondence and verify that it behaves as expected. But these examples (and those listed at the end of Example \ref{hypexample}) might already make the correspondence sufficiently clear, in which case the reader can skip to the more interesting results of the next section.

\begin{notation} We denote by $S'_x(M)$ the space of complete non-realised types in the variable $x$ over $M$.
\end{notation}

\begin{definition} An invariant Keisler measure on $\MM$ in the variable $x$ is a regular Borel probability measure\footnote{When $\MM$ is countable, any $\mathrm{Aut}(\MM)$-invariant Borel probability measure on $S_x(M)$ is regular \cite[Theorem 7.1.7]{Boga}.} on $S_x(M)$. We denote by $\mathfrak{M}_x(M)$ be the space of invariant Keisler measures on $M$ in the variable $x$.  We denote by $\mathfrak{M}'_x(M)$ the space of invariant Keisler measures on $M$ in the variable $x$ whose support contains no realised type. These are just $\mathrm{Aut}(\MM)$-invariant regular Borel probability measures on $S'_x(M)$. 
\end{definition}

\begin{remark} Invariant Keisler measures correspond to finitely additive $\mathrm{Aut}(\MM)$-invariant probability measures on the Boolean algebra of definable subsets of $M$ in the variable $x$ with parameters from $M$, $\mathrm{Def}_x(M)$ \cite[\S 7.1]{NIP}. Indeed, for a Keisler measure $\mu$ and an $\mathcal{L}(M)$-formula $\phi(x, \overline{a})$, we will write $\mu(\phi(x, \overline{a}))$ for the measure $\mu$ assigns to the clopen set 
\[[\phi(x, \overline{a})]:=\{p\in S_x(M)\vert \phi(x, \overline{a})\in p\}.\]
It makes sense to focus the study of invariant Keisler measures to $\omega$-saturated and strongly $\omega$-homogeneous models. Strong $\omega$-homogeneity yields that $\overline{a}\equiv a'$ implies $\mu(\phi(x, \overline{a}'))=\mu(\phi(x, \overline{a}))$. Hence, any two $\omega$-saturated strongly $\omega$-homogeneous models will have essentially the same spaces of invariant Keisler measures.
\end{remark}

\begin{remark}\label{rem:ergodic} For $\mathcal{M}$ countable, it is often convenient to focus on ergodic measures when studying invariant Keisler measures. We say that an invariant Keisler measure is \textbf{ergodic} when for all Borel $B\subseteq S_x(M)$, we have that if for all $\sigma \in\mathrm{Aut}(\MM)$, $\mu(B\triangle \sigma B)=0$, then $\mu(B)\in\{0,1\}$. Any invariant Keisler measure can be decomposed as an integral average of ergodic measures \cite{Choquet, me2}.  Moreover, ergodic measures concentrate on orbits: this is because $\mathrm{Aut}(\MM)$-orbits are Borel (since the action of $\mathrm{Aut}(\MM)$ on $S_x(M)$ is Borel \cite[Theorem 15.14]{KechrisDST}), and so they must be assigned value $0$ or $1$ by an ergodic measure since they are $\mathrm{Aut}(\MM)$-invariant.
\end{remark}

\begin{remark} In this section we will be showing how we can think of invariant Keisler measures in $\mathfrak{M}'_x(M)$ as invariant measures on spaces of expansions of $M$ which represent the space of non-realised types. Everything we do also works with $\mathfrak{M}_x(M)$, though we think that more clarity is achieved by looking at measures with no realised types in their support: this makes the choice of language of the expansions more natural when looking at examples. Moreover, our main focus in this paper is on countable structures for which there is little reason to focus on measures containing realised types in their support: as noted in Remark \ref{rem:ergodic}, for $\mathcal{M}$ countable we can focus on ergodic measures. If an ergodic measure $\mu$ contains a type $p$ realised by $a\in M$ in their support, $\mu$ must concentrate on the orbit of $p$, $\mathrm{Orb}(p)$, since $[x=a]=\{p\}\subseteq \mathrm{Orb}(p)$ is assiged positive measure. Since $M$ is countable and $p$ is realised, $\mathrm{Orb}(p)$ is at most countable. If it is countably infinite, by $\sigma$-additivity and $\mathrm{Aut}(\MM)$-invariance, there cannot be any ergodic measure (and so no measure) containing $p$ in their support. If $\mathrm{Orb}(p)=\{p_1, \dots, p_n\}$ for some $n\in\mathbb{N}$, then, by additivity and $\mathrm{Aut}(\MM)$-invariance, 
\[\mu=\frac{1}{n}\sum_{i=1}^n p_i.\]
Hence, the ergodic measures with realised types in their support can be fully understood. 
\end{remark}

For $\mathcal{L}'$ a relational language distinct from $\mathcal{L}$, we denote by $\mathrm{Struc}_{\mathcal{L}'}(M)$ the space of expansions of $\mathcal{M}$ to $\mathcal{L}^*:=\mathcal{L}\cup\mathcal{L}'$. As in Definition \ref{def:strucl}, $\mathrm{Struc}_{\mathcal{L}'}(M)$ is equipped with a natural topology induced from the product topology and $\mathrm{Aut}(\MM)$ acts on $\mathrm{Struc}_{\mathcal{L}'}(M)$ via a continous relativised logic action.\\




    

\begin{definition} Let $X$ and $Y$ be compact Hausdorff topological spaces and let $G$ be a topological group acting continously on each of them. A $G$-map $\Gamma:X\to Y$ is a continuous map such that for all $x\in X$ and $g\in G$, 
\[\Gamma(g\cdot x)=g\cdot \Gamma(x).\]
\end{definition}

\begin{definition}\label{repres} Let $\mathcal{M}$ be an $\mathcal{L}$-structure with quantifier elimination. Let $\mathrm{Struc}_{\mathcal{L}'}(M)$ be the space of expansions of $\mathcal{M}$ to $\mathcal{L}^*:=\mathcal{L}\cup\mathcal{L}'$.

We say that the the type space $S_x'(M)$ is \textbf{representable in} $\mathrm{Struc}_{\mathcal{L}'}(M)$ if there is an injective $\mathrm{Aut}(\MM)$-map $\Gamma:S_x'(M)\to \mathrm{Struc}_{\mathcal{L}'}(M)$.

For $S\subseteq \mathrm{Struc}_{\mathcal{L}'}(M)$, we say that $S'_x(M)$ is \textbf{represented by} $S$ in $\mathrm{Struc}_{\mathcal{L}'}(M)$ whenever $S'_x(M)$ is representable in $\mathrm{Struc}_{\mathcal{L}'}(M)$ via an $\mathrm{Aut}(\MM)$-map $\Gamma$ with range $S$.
\end{definition}

Since $\Gamma$ is continuous, $S$ is compact and so closed in $\mathrm{Struc}_{\mathcal{L}'}(M)$. Hence, $\Gamma$ is a homeomorphism 
between $S_x'(M)$ and $S$.\\

As we shall see in Lemma \ref{Gembedding}, for any $\mathcal{L}$-structure $\mathcal{M}$, we can  represent $S'_x(M)$ in a particular space of expansions of $\mathcal{M}$, which we will call expansions by the projection language. Usually, there are more natural choices of languages $\mathcal{L}'$ such that we may represent $S_x'(M)$ by expansions to $\mathcal{L}^*=\mathcal{L}\cup\mathcal{L}'$. Indeed, thinking in terms of these can be helpful in building an intuition as to what representations of non-realised types in spaces of expansions look like (and why we should want to think in terms of them). Hence, we begin with a natural example of how to represent the spaces of non-realised types of homogeneous hypergraphs.

\begin{example}[homogeneous hypergraphs]\label{hypexample} Suppose that $\mathcal{M}$ is a homogeneous uniform $k$-hypergraph for $k\geq 2$, where $R(x_1, \dots, x_k)$ denotes the hypergraph relation. We argue that we can represent non-realised types over $M$  as expansions by a $(k-1)$-ary uniform hypergraph.\\

Let $\{E\}$ be the language consisting of a single $(k-1)$-ary relation $E(x_1,\dots, x_{k-1})$. The type space $S'_x(M)$ is representable in $\mathrm{Struc}_{\{E\}}(M)$ by some $S\subseteq \mathrm{Graph}_E(M)$, where $\mathrm{Graph}_E(M)$ is the space of expansions of $M$ where $E$ forms a uniform $(k-1)$-ary hypergraph. To see this, we define a map $\Gamma:S'_x(M)\to \mathrm{Graph}_E(M)$, given by $p\mapsto \mathcal{M}^*_p$, where the latter is an expansion of $\mathcal{M}$ by a $(k-1)$-uniform hypergraph $E$, where for a $k-1$-tuple $\overline{a}$,
\[ \mathcal{M}_p^*\vDash E(\overline{a}) \text{ if and only if } R(x, \overline{a})\in p.\] 
Since $R$ is symmetric and uniform, so is $E$, yielding that we do indeed have that $\mathrm{Range}(\Gamma)\subseteq \mathrm{Graph}_E(M)$. The map is injective since two distict types will disagree over whether $R(x, \overline{a})$ for some tuple. It is also easy to see that $\Gamma$ is also a topological embedding and an $\mathrm{Aut}(\MM)$-map. For a proof of this, we refer the reader to the proof of Lemma \ref{Gembedding}, which deals with the more general case. \\

Inspecting different examples of homogeneous graphs and hypergraphs one can deduce the following conclusions:
\begin{itemize}
    \item For $\mathcal{R}_2$ the countable model of the random graph, $S'_x(\mathcal{R}_2)$ is represented by the space of  expansions of $\mathcal{R}_2$ by a unary predicate $P$, i.e. by $\mathrm{Struc}_{\{P\}}(\mathcal{R}_2)$;
    \item For $\mathcal{H}^2_3$ the countable model of the generic triangle-free graph, $S'_x(\mathcal{H}^2_3)$ is represented by the space of  expansions of $\mathcal{H}^2_3$ by a unary predicate $P$ such that for any $a,b\in\mathcal{H}^2_3$ sharing an edge we do not have that $P$ holds for both $a$ and $b$. This space lives inside of $\mathrm{Struc}_{\{P\}}(\mathcal{H}^2_3)$;
    \item For $\mathcal{R}_3$ the countable model of the universal homogeneous $3$-hypergraph, $S'_x(\mathcal{R}_3)$ is represented by the space of  expansions of $\mathcal{R}$ by a graph given by the binary relation $E$. This space lives inside $\mathrm{Struc}_{\{E\}}(\mathcal{R}_3)$;
    \item For $\mathcal{H}^3_4$ the universal homogeneous tetrahedron-free $3$-hypergraph, $S'_x(\mathcal{H}^3_4)$ is represented by the space of expansions of $\mathcal{H}^3_4$ by a graph relation $E$ such that whenever $a,b,c\in\mathcal{H}^3_4$ form a hyperedge, we have that they cannot form a triangle with respect to $E$;
    \item for $\mathcal{G}_3$ the universal homogeneous parity $3$-hypergraph, $S'_x(\mathcal{G}_3)$ is represented by the space of expansions of $\mathcal{G}_3$ by graphings (i.e. by a graph relation $E$ such that for any three points $a,b,c\in\mathcal{G}_3$, we have that there is an odd number of edges in $\{a,b,c\}$ if and only if $a,b,c$ form a hyperedge in $\mathcal{G}_3$). This follows from the parity condition on the $4$-element subsets of the parity $3$-hypergraph.
\end{itemize}
\end{example}

\begin{definition}[Projection language] For each $\mathcal{L}$-formula $\phi(x;\overline{y})$ let $R_{\phi(x;\overline{y})}$ be a relation of arity $|\overline{y}|$. Let $\mathcal{L}^{\mathrm{s}}$ be a relational language consisting of the relations $R_{\phi(x;\overline{y})}$ for each formula $\phi(x;\overline{y})$ a formula containing the variable $x$.
\end{definition}

The idea behind definition Definition \ref{Mpstar} is that we can represent types over $M$ as a particular class of expansions of $M$. As noted before Example \ref{hypexample}, usually, reflecting on the structure one is working with, there are better choices of language than $\mathcal{L}^{\mathrm{s}}$ and better choices of spaces of expansions. However, our construction has the advantage of working in a fully general setting. The general construction is similar to Example \ref{hypexample}. 

\begin{definition}\label{Mpstar} Let $\mathcal{M}$ be a $\mathcal{L}$-structure and $p\in S'_x(M)$. Consider the $\mathcal{L}^{\mathrm{s}}$-expansion of $\mathcal{M}$, $\mathcal{M}_p^*$, where for $\phi(x;\overline{y})$ an $\mathcal{L}$-formula containing the variable $x$, 
\[ \mathcal{M}_p^*\vDash R_{\phi(x;\overline{y})}(\overline{a}) \text{ if and only if } \phi(x;\overline{a})\in p.\]
\end{definition}

The construction of $\mathcal{M}^*_p$ follows a similar route to that of the Shelah expansion $\mathcal{M}^{\mathrm{Sh}}$ \cite{shelah2009dependent}. Given $\mathcal{M}\preceq \mathcal{N}$ sufficiently saturated, the Shelah expansion adds relations to $\mathcal{M}$ for each externally definable set (i.e. for each set of the form $X\cap M$, where $X$ is defined by an $\mathcal{L}(N)$-formula). Given  $|M|^+$-saturated $\mathcal{N}\succeq \mathcal{M}$ and $a\in \mathcal{N}$ realising $p\in S'_x(M)$, our expansion $\mathcal{M}^*_p$ adds relations to $\mathcal{M}$ for each externally definable set of $M$ defined over $a$.

\begin{lemma} \label{Gembedding} Let $\mathcal{M}$ be an $\mathcal{L}$-structure. Then, $S_x'(M)$ is representable in $\mathrm{Struc}_{\mathcal{L}^{\mathrm{s}}}(M)$ via the injective $\mathrm{Aut}(\MM)$-map $\Gamma:S'_x(M)\to \mathrm{Struc}_{\mathcal{L}^{\mathrm{s}}}(M)$ given by $p\mapsto \mathcal{M}_p^*$, where $\mathcal{M}_p^*$ is as in Definition \ref{Mpstar}.
\end{lemma}
\begin{proof} Firstly, note that $\Gamma$ is injective since if two non-realised types disagree this is witnessed by some $\mathcal{L}(M)$-formula $\phi(x;\overline{a})$ and so they will be sent to distinct expansions.\\

We know that the sets of the form $\left\llbracket R_{\phi(x;\overline{y})}(\overline{a}) \right\rrbracket$ form a subbasis of clopen sets for the topology on $\mathrm{Struc}_{\mathcal{L}^{\mathrm{s}}}(M)$ for $R_{\phi(x;\overline{y})}$ a relation in $\mathcal{L}^{\mathrm{s}}$ and $\overline{a}$ a tuple from $\mathcal{M}$ of arity $\overline{y}$. To prove continuity of $\Gamma$, it is sufficient to check that the preimages of these clopen sets are open. By definition of $\Gamma$, the preimage of $\left\llbracket R_{\phi(x;\overline{y})}(\overline{a}) \right\rrbracket$ is precisely the clopen set in $S_x'(M)$ given by
\[[\phi(x;\overline{a})]=\{p\in S_x'(M) \vert \phi(x;\overline{a})\in p\}. \]
Hence $\Gamma$ is continuous. Finally, it is easy to see that $\Gamma$ is an $\mathrm{Aut}(\MM)$-map.
\end{proof}

\begin{remark} There are many reasonable choices of $\mathcal{L}'$ such that we may represent $S_x'(M)$ as a space of expansions of $\mathcal{M}$ to $\mathcal{L}'$ as in Lemma \ref{Gembedding}. If $\mathcal{M}$ has quantifier elimination, it is sufficient to take $\mathcal{L}'$ containing relations $R_{\phi(x;\overline{y})}$ for the atomic $\mathcal{L}$-formulas. In Definition \ref{Mclass} we define a natural choice of $\mathcal{L}'$ when $\mathcal{M}$ is homogeneous in a finite relational language, which has the advantage of being finite when $\mathcal{L}$ is, thus making it easier to apply the results of Section \ref{sec:invexch}.
\end{remark}




\begin{remark} Recall that for a measure space $(X,\mu)$ and a measurable map $f$ from $X$ to some measurable space $Y$, the pushforward of $\mu$ by $f$ is defined to be the measure $f_\sharp(\mu)$ on $Y$ such that,
\[f_\sharp(\mu)(A)=\mu(f^{-1}(A)).\]
\end{remark}

\begin{notation} Below, by an $\mathrm{IRE}_{\mathcal{L}'}(M)$ we mean an $\mathrm{Aut}(\MM)$-invariant regular Borel probability measure on $\mathrm{Struc}_{\mathcal{L}'}(M)$. When $\mathcal{M}$ is countable, this corresponds to Definition \ref{def:IRE}. 
\end{notation}

The following is a direct consequence of Definition \ref{repres}:

\begin{corollary}\label{meascorr}  Let $S_x'(M)$ be representable by $S$ in $\mathrm{Struc}_{\mathcal{L}'}(M)$ via $\Gamma$. Then, $\Gamma$ induces a bijection $\gamma$ between $\mathfrak{M}'_x(M)$ and the space of $\mathrm{IRE}_{\mathcal{L}'}(M)$ concentrating on $S$, where for each $\mu\in\mathfrak{M}'_x(M)$, $\Gamma$ induces an isomorphism of measure spaces between $(S_x'(M), \mu)$ and $(S, \Gamma_\sharp(\mu))$, where $\Gamma_\sharp(\mu)$ is the pushforward measure induced by $\Gamma$.
\end{corollary}
\begin{proof}
    This is automatic since $\Gamma_\sharp(\mu)$ is the pushforward of $\mu$ by a bijection.
\end{proof}

\subsection{Keisler measures and IREs in a homogeneous context}\label{homcont}

We now focus on how the correspondence from the previous section specialises to a homogeneous context.

\begin{definition}
Let $\mathcal{L}$ be a relational language. Let $(P_u\vert u<v)$ be an enumeration of the unary predicates of $\mathcal{L}$. We consider the expansion of $\mathcal{L}$ to the language $\mathcal{L}^*:=\mathcal{L}\cup\mathcal{L}^{\mathrm{pr}}$ as follows: for each relation $R_i$ of arity $r_i, u< v, K\subsetneq [r_i]$, $\mathcal{L}^{\mathrm{pr}}$ has a relation $R_i^{(K,u)}$ of arity $r_i-|K|$. 
\end{definition}

 \begin{definition}\label{Mclass} Let $\mathcal{F}$ be a Fra\"{i}ss\'{e} class. Let $(P_u\vert u<v)$ be the unary relations in $\mathcal{L}$. Let $Ab\in\mathcal{F}[n+1]$, where $A$ is the induced structure on $[n]$ and $b$ is the vertex corresponding to $n+1$. We consider the $\mathcal{L}^{\mathrm{pr}}$-expansion $A^b$ of $A$ obtained as follows:

For $K\subsetneq[r_i]$ and $m_1, \dots, m_{r_i-|K|}\leq n$, let $\overline{m}=(m_1, \dots, m_{r_i-|K|})$ and $\overline{m}^K$ be the $r_i$-tuple consisting of $b$ in each position in $K$ and the $m_j$ in the other positions. Hence, we set that, for $u<v$,
\[ A^b\vDash R_i^{(K,u)}(\overline{m}) \text{ if and only if } R_i(\overline{m}^K)\wedge P_u(b).\]
We define the $\mathrm{MC}(\mathcal{F})$, \textbf{measuring class} of $\mathcal{F}$,  as the class of finite $\mathcal{L}^*$-structures isomorphic to some $A^b$ for $Ab\in\mathcal{F}$. For $\mathcal{M}=\mathrm{Flim}(\mathcal{F})$, we write  $\mathrm{MC}(\mathcal{M})$ for the measuring class of $\mathcal{F}$.
\end{definition}

 Note that for $\mathcal{M}$ a homogeneous structure, we may represent $S'_x(M)$ by the space of $\mathcal{L}^{\mathrm{pr}}$-expansions of $\mathcal{M}$, $\mathcal{M}_p^*$, where for $\overline{m}$ from $M$,
\[ \mathcal{M}_p^*\vDash R_i^{(K,u)}(\overline{m}) \text{ if and only if } R_i(\overline{m}^K)\wedge P_u(x)\in p.\]
Our choice of language and space of expansions is superficially different from that of Definition \ref{Mpstar}, though an analogue of Lemma \ref{Gembedding} also holds in this context by the same argument. We rely on quantifier elimination to consider $S'_x(M)$ as representable in a space of expansions of $\mathcal{M}$ by a language $\mathcal{L}^{\mathrm{pr}}$ which has only finitely many relations when $\mathcal{L}$ also does. Indeed, the following is a direct consequence of Lemma \ref{Gembedding}:

\begin{lemma} Let $\mathcal{M}$ be an homogeneous $\mathcal{L}$-structure. The space $S'_x(M)$ is representable by
\[\mathrm{ST}(M):=\left\{M^*\in\mathrm{Struc}_{\mathcal{L}^{\mathrm{pr}}}(M) \vert \ \mathrm{Age}(\MM^*)\subseteq \mathrm{MC}(\mathcal{M})\right\}.\]
\end{lemma}

\begin{remark}\label{rk:lowerarity} Recall that the projection language $\mathcal{L}^{\mathrm{pr}}$ has for every $\mathcal{L}$-relation of arity $k+1$ a set of relations of arity $\leq k$. In particular, if all relations of $\mathcal{M}$ have the same arity $k+1$, the space $\mathrm{ST}(M)$ is a space of expansions of $M$ by relations of arity $\leq k$.
\end{remark}

\begin{theorem}\label{thm:IKMareSI} Let $\mathcal{M}$ be homogeneous $k$-transitive in a finite $(k+1)$-ary language whose age is $k$-overlap closed. Then, any invariant Keisler measure of $\mathcal{M}$ in the variable $x$ is exchangeable.
\end{theorem}
\begin{proof} $S_\infty$-invariance of the Keisler measure follows from Remark \ref{rk:lowerarity} and Theorem \ref{thm:appliedforIKM}.  
\end{proof}

\begin{definition} Let $\mathcal{C}^*$ be a class of $\mathcal{L}^*$-structures where $\mathcal{L}^*=\mathcal{L}\cup\mathcal{L}'$. We define the \textbf{exchangeable} $\mathcal{L}'$-class for $\mathcal{C}^*$ to be the following class of finite $\mathcal{L}'$-structures 
\[\mathrm{Exc}(\mathcal{C}^*):=\{A'\in\mathcal{C}^*\upharpoonright_{\mathcal{L}'} \vert \text{ for all } A\in\mathcal{C}^*\upharpoonright_{\mathcal{L}} \; A*A'\in\mathcal{C^*}\}.\]
%
    
\end{definition}

\begin{remark} In Example \ref{hypexample}, we showed that the type space of the universal homogeneous tetrahedron-free $3$-hypergraph $\mathcal{H}_4^3$ is represented by the space of graph expansions of $\mathcal{H}_4^3$ omitting a triangle on top of a hyperedge. Let $\mathcal{L}$ consist of a ternary relation and $\mathcal{L}'$ consist of a binary relation and let $\mathcal{C}^*$ denote the class of tetrahedron-free $\mathcal{L}$ $3$-hypergraphs with an $\mathcal{L}'$-graph such that there is no $\mathcal{L}'$-triangle on top of an $\mathcal{L}$-hyperedge. We can see that $\mathrm{Exc}(\mathcal{C}^*)$ is the class of triangle-free graphs in the language $\mathcal{L}'$.
\end{remark}


\begin{lemma}\label{invcorr} Let $\mathcal{M}$ be a homogeneous $\mathcal{L}$-structure and $\mathcal{C}^*$ a class of finite $\mathcal{L}^*$-structures such that $\mathcal{C}^*\upharpoonright_{\mathcal{L}}=\mathrm{Age}(\MM)$. Suppose that all IREs in $\mathrm{IRE}(\mathcal{M}, \allowbreak \mathcal{C}^*)$ are exchangeable. 
Consider the injection $\Delta:\mathrm{Struc}(M, \mathcal{C}^*)\to\allowbreak \mathrm{Struc}(\mathbb{N}, \mathcal{C}^*\upharpoonright_{\mathcal{L}'})$ given by $M*M'\mapsto\mathbb{N}*M'$. Then, $\Delta$ induces a bijection $\delta$ between the spaces $\mathrm{IRE}(\mathcal{M}, \mathcal{C}^*)$ and $S_\infty(\mathrm{Exc}(\mathcal{C}^*))$ where, for each $\mu\in \mathrm{IRE}(\mathcal{M}, \allowbreak \mathcal{C}^*)$, $\Delta$ induces an isomorphism of the measure space $(\mathrm{Struc}(M, \allowbreak \mathcal{C}^*), \mu)$ with the space $(\mathrm{Range}(\Delta), \Delta_\sharp(\mu))$, where $\Delta_\sharp(\mu)$ is the pushforward of $\mu$ by $\Delta$, and it concentrates on $\mathrm{Struc}(\mathbb{N}, \allowbreak \mathrm{Exc}(\mathcal{C}^*))$.
\end{lemma}

\begin{proof}
   Since all IREs in $\mathrm{IRE}(\mathcal{M}, \mathcal{C}^*)$ are exchangeable, $\Delta$ has indeed range in $S_\infty(\mathrm{Exc}(\mathcal{C}^*))$. Finally, we nees to prove that $\Delta$ is a bijection, which is easy since $\Delta$ is a bijection.
\end{proof}

\begin{corollary}\label{cor:correspondence} Let $\mathcal{M}$ be a homogeneous $\mathcal{L}$-structure and suppose that every measures in $\mathrm{IRE}(M,\allowbreak \mathrm{MC}(\mathcal{M}))$ is exchangeable. Then, there is a map $\Theta\colon S'_x(M) \to \mathrm{IRE}(\mathcal{M}, \allowbreak \mathrm{MC}(\mathcal{M}))$ such that for every $\mu \in\mathfrak{M}'_x(M)$, $\Theta$ induces an isomorphism of measure spaces between $(S'_x(M), \mu)$ and $(\mathrm{Range}(\Theta), \Theta_\sharp(\mu))$, where, $\Theta_\sharp(\mu)$ is the pushforward of $\mu$ by $\Theta$.
Moreover, $\Theta_\sharp(\mu)$ concentrates on $\mathrm{Struc}(\mathbb{N}, \allowbreak \mathrm{Exc}(\mathrm{MC}(\mathcal{M})))$.
    
\end{corollary}

\begin{proof}

Consider the map $\Theta:=\Delta\circ\Gamma$, where $\Gamma$ is the map giving that $S'_x(M)$ is representable in $\mathrm{ST}(M)$, and $\Delta$ is the map defined in Lemma \ref{invcorr}, where $\mathcal{C}^*$ is taken to be $\mathrm{MC}(\mathcal{M})$.  The conclusion follows from Lemma~\ref{invcorr}.
\end{proof}

\section{Consequences}\label{sec:consequences}
The correspondence between invariant Keisler measures and invariant random expansions explored in the previous section in Corollary \ref{meascorr} and \ref{cor:correspondence} can be used together with Theorem \ref{thm:appliedforIKM} to describe the spaces of invariant Keisler measures of many homogeneous structures which were previously not understood.\\

With a modern perspective, we can see that the main tool behind
Albert's description of the spaces of invariant Keisler measures of the random graph and the generic $K_n$-free graphs lies in a connection between model theoretic notions of independence and probabilistic independence when looking at the measure of the intersection of two sets defined over independent parameters. In particular, from \cite[Proposition 2.25]{approxsubg}, we know that the relation $\mu(\phi(x,a)\wedge\psi(x,b))=\alpha$ is stable (as an invariant relation) for any invariant Keisler measure (see also \cite[Proposition 4]{taopost}). Under the additional assumption that $\mathrm{acl}^{eq}(\emptyset)=\mathrm{dcl}^{eq}(\emptyset)$, one can use basic stability theory to derive that  the value of $\mu(\phi(x,a)\wedge\psi(x,b))$ only depends on $\mathrm{tp}(a)$ and $\mathrm{tp}(b)$ and not on $\mathrm{tp}(ab)$ when $a$ and $b$ are independent (and we are working in a strongly $\omega$-homogeneous model). A version of this phenomenon was observed at various points in the development of model theory: see \cite[Lemma 8.4.2-Proposition 8.4.3]{FSFT} for smoothly approximable structures and, \cite{TomInd} for pseudofinite fields. The dependence of $\mu(\phi(x,a)\wedge\psi(x,b))$ only on $\mathrm{tp}(a)$ and $\mathrm{tp}(b)$ was the main ingredient in Pillay and Starchenko's proof of the algebraic regularity lemma \cite[Lemma 1.1]{PillayStar}. Stability of $\mu(\phi(x,a)\wedge\psi(x,b))=\alpha$ can also be deduced from stability of probability algebras, which was proven in \cite{ben2006schrodinger}, but can be seen implicitly in Ryll-Nardzewski's theorem (about invariant random expansions of $(\mathbb{Q},<)$ by unary predicates) \cite{ryll1957stationary} and in \cite{krivine1981espaces}. \\

In an $\omega$-categorical context this relation between model-theoretic and probabilistic independence can be strengthened. Firstly, it is sufficient to focus on the ergodic measures in the space of invariant Keisler measures since any other measure can be written as an integral average of them \cite{Choquet} (see \cite{me2}). From \cite{JahelT}, we get that for ergodic invariant Keisler measures over the countable model of an $\omega$-categorical theory, weak algebraic independence of the parameters $a$ and $b$ yields actual probabilistic independence \cite{me2} (cf. \cite{PieceInt}):
\begin{equation}\label{eq:indeptransf}
    \mu(\phi(x,a)\wedge\psi(x,b))=\mu(\phi(x,a))\mu(\psi(x,b)).
\end{equation}
This transfer of independence together with the ergodic decomposition, yields easy proofs of Albert's results in \cite{Albert}. It can also be used to study invariant Keisler measures in other examples such as vector spaces with bilinear forms over finite fields, which we leave to the reader as a fun exercise.\\

When one moves to ternary structures, there is no analogue of the transfer of independence of \ref{eq:indeptransf} for $\mu(\phi(x, a, b)\wedge \psi(x, b,c)\wedge\xi(x, a, c))$ (cf. \cite[Theorem B.11]{AER}). This can be clearly seen even in the random $3$-hypergraph as we point out in Remark \ref{rem:randomhypIKM}. However, 
our exchangeability results suggest that under suitable conditions, independence between the parameters $abc$ may imply that that $\mu(\phi(x, a, b)\wedge \psi(x, b,c)\wedge\xi(x, a, c))$ only depends on the types of pairs and not on the type of the triplet. A similar result is obtained under the assumption of higher amalgamation in \cite[Theorem B.8]{AER}, elaborating on the hypergraph regularity lemma of Chevalier and Levi \cite[p.7 and Corollary 4.1.3]{AlexisARL}. However, we know from the classification of the invariant Keisler measures for the universal homogeneous parity $3$-hypergraph that for simple structures this is not in general the case even under further model theoretic tameness assumptions (see Remark \ref{rem:twographIKM}). Clearly, there is more to be studied regarding the measures of more complex intersections of formulas and what criteria imply that $\mu(\phi(x, a, b)\wedge \psi(x, b,c)\wedge\xi(x, a, c))$ only depends on the types of the pairs in $\{a,b,c\}$ rather than the type of the triplet.

\subsection{Classifying spaces of measures in examples}\label{classtime}
In this section we describe the spaces of invariant Keisler measures of the homogeneous hypergraphs from Example \ref{ex:hypergraphs}. In particular, in most cases we give the space of $S_\infty$-invariant measures corresponding to the space of invariant Keisler measures for $M$ in the singleton variable $x$ according to Corollary \ref{cor:correspondence}. We use the description of the spaces of expansions representing the type spaces $S'_x(M)$ given in Example \ref{hypexample}.

\begin{corollary}\label{cor:randomhypIKM}  For $r\geq 2$, the space of invariant Keisler measures for the universal homogeneous $r$-hypergraph $\mathcal{R}_r$ corresponds to the space of $S_\infty$-invariant random $(r-1)$-hypergraphs.  
\end{corollary}

\begin{remark}\label{rem:randomhypIKM} Our result contradicts the original conjecture of Ensley \cite{EnsleyPhD}, who conjectured 
that all ergodic invariant Keisler measures for $\mathcal{R}_3$ would be Bernoulli. Moreover, it is easy to construct measures with some counterintuitive properties. For example, from Petrov and Vershik \cite{PetrovVershik} there is an $S_\infty$-invariant measure concentrating on the isomorphism type of the generic triangle-free graph. This corresponds to an invariant Keisler measure $\mu$ for $\mathcal{R}_3$ such that for all $a,b,c$ distinct,
\[\mu(R(x,a,b))>0 \text{ but } \mu(R(x, a, b)\wedge R(x, a, c)\wedge R(x, b, c))=0,\]
where $R$ is the $3$-hyperedge relation on $\mathcal{R}_3$. This gives a clear counterexample to any hope for a decomposition of the measure of the intersection $R(x, a, b)\wedge R(x, a, c)\wedge R(x, b, c)$ into the measures of the individual definable sets in disanalogy with the measure of the intersection of two sets defined over independent parameters we mentioned in the introduction to this section. Such decomposition seems possible under further assumptions of the measure such as a version of Fubini's Theorem (cf. \cite[Remark 7.2.3]{myPhD} and \cite[Theorem B.11]{AER}).
\end{remark}

The following is just a consequence of Corollary \ref{cor:tetIRE}:

\begin{corollary}\label{cor:tetfreehypIKM} For $2< r< n$, the space of invariant Keisler measures for $\mathcal{H}_n^r$ corresponds to the space of $S_\infty$-invariant random $K_{n-1}^{r-1}$-free $(r-1)$-hypergraphs.
\end{corollary}


\begin{corollary}\label{cor:kayIKM} Let $k\geq 3$. The space of invariant Keisler measures for the universal homogeneous parity $k$-hypergraph $\mathcal{G}_{k}$ consists of a unique invariant Keisler measure corresponding to the unique $\mathrm{Aut}(\mathcal{G}_{k})$-invariant random $(k-1)$-hypergraphing of $\mathcal{G}_k$ according to the correspondence described in Corollary \ref{meascorr}. In particular, for $d, a_1, \dots, a_n\in \mathcal{G}_k$ and $n\geq k-1$,
\[\mu(d/a_1, \dots, a_n)= 2^{-{n-1 \choose k-2}},\]
where $\mu(d/a_1, \dots, a_n)$ is the measure of the formula isolating $\mathrm{tp}(d/a_1, \dots, a_n)$.
\end{corollary}
\begin{proof}
   Similarly to the case of the parity $3$-hypergraph discussed in Example \ref{hypexample}, the type space $S_x(\mathcal{G}_k)$ is representable by the space of $k$-hypergraphings of $\mathcal{G}_k$. In Subsection \ref{sec:kaygraphs}, we argued there is a unique $\mathrm{Aut}(\mathcal{G}_k)$-invariant measure on such space.
\end{proof}

\begin{remark} It is easy to see that $\mathcal{G}_k$ has no $\mathrm{Aut}(\mathcal{G}_k)$-invariant types. Hence $\mathcal{G}_k$ gives an example of a structure with $\mathrm{acl}^{eq}(\emptyset)=\mathrm{dcl}^{eq}(\emptyset)$ that has a unique invariant Keisler measure and no invariant types. Without the first condition, it is easy to build such examples, e.g. with the theory of an equivalence relation with two infinite classes. 
{A vector space over a finite field with a symplectic bilinear form also exhibits this behaviour: it has no invariant types, and only two invariant Keisler measures, one of which concetrates on the type realised by the $0$ in the vector space. This can be proven using the techniques described in the introduction to Section \ref{sec:consequences}.}
\end{remark}

\begin{remark}\label{rem:twographIKM} In the universal homogeneous parity $3$-hypergraph $\mathcal{G}_3$, given an independent triplet  $abc$, the measure $\mu(R(x,a,b)\wedge R(x, a,c)\wedge R(x, b, c))$ still depends on the type of the whole triplet rather than just on the types of pairs. For $abc$ not forming a hyperedge, the formula $R(x,a,b)\wedge R(x, a,c)\wedge R(x, b, c)$ is inconsistent in $\mathcal{G}_3$, and so must be assigned measure zero. However, for $a'b'c'$ forming a hyperedge,  $R(x,a',b')\wedge R(x, a',c')\wedge R(x, b', c')$  is consistent and by Corollary \ref{cor:kayIKM}, it has measure $1/4$. In both cases the triplets are independent, since independence is trivial, and they agree on the types of pairs since $\mathcal{G}_3$ is $2$-transitive. This provides a challenge to generalising to higher arity the fact that $\mu(\phi(x,a)\wedge\psi(x,b))$ only depends on $\mathrm{tp}(a)$ and $\mathrm{tp}(b)$ for an independent pair when $\mathrm{acl}^{eq}(\emptyset)=\mathrm{dcl}^{eq}(\emptyset)$ in an $\omega$-categorical context. 
\end{remark}

Interestingly, the moment that we move out of simple structures, the "random" invariant Keisler measures (i.e. those which are not weighted sums of invariant types) easily disappear, at least in our examples.

\begin{corollary}\label{cor:strangemeas} The space of invariant Keisler measures in the singleton variable for the homogeneous $3$-hypergraphs of the form $\mathcal{H}_4^{-}$ and $\mathcal{P}_n^3$ consists of the unique invariant type asserting that $x$ forms no $3$-hyperedge with any pair of vertices from the model.
\end{corollary}
\begin{proof} We prove this for  $\mathcal{H}_4^{-}$. The proof for  $\mathcal{P}_n^3$ is essentially identical. As usual, we represent the type space of $\mathcal{H}_4^{-}$ in a singleton variable by a graph expansion of $\mathcal{H}_4^{-}$. We know that the age of $\mathcal{H}_4^{-}$ is $2$-overlap closed by Corollary \ref{cor:k+1 irred}. So we get exchangeability from Theorem \ref{thm:IKMareSI}. The associated exchangeable class consists of the class of graphs omitting a triangle and a path of length $2$. This is the class of graphs which consists of disjoint copies of edges and vertices. However, if an edge were to be assigned positive measure, the average degree of a vertex would be infinite, and therefore greater than $1$ with positive probability.
Hence, there is a unique $S_\infty$-invariant measure concentrating on the infinite empty graph $\overline{K_\omega}$. This measure corresponds to the unique invariant type in the variable $x$ for $\mathcal{H}_4^{-}$ asserting that $x$ does not form a hyperedge with any pair of vertices.
\end{proof}

\begin{remark} It is easy to see that the space of invariant Keisler measures for $\mathcal{H}_{r+1}^{-}$ for $r>3$ corresponds to the space of $S_\infty$-invariant measures concentrating on the age of $\mathcal{P}_{r}^{r-1}$. In this case the space of invariant Keisler measures is again quite rich since  $\mathcal{P}_{r}^{r-1}$ is homogeneous with trivial algebraic closure. 
\end{remark}

\subsection{Forking and universally measure zero formulas}\label{FOtime}

In the model theoretic preliminaries we noted how forking gives a natural notion of smallness for a definable set. Keisler measures also yield a natural notion of smallness: a definable set is \textbf{universally measure zero} if it is assigned measure zero by every invariant Keisler measure. Just like forking formulas,  definable sets in the variable $x$ which are universally measure zero form an ideal in $\mathrm{Def}_x(\mathbb{M})$, which we denote by $\mathcal{O}_x(\emptyset)$. We also write $\mathcal{O}(\emptyset)$ for the space of universally measure zero definable sets in an arbitrary (finite) variable. \\

The relation between the sets $F(\emptyset)$ and $\mathcal{O}(\emptyset)$ has been studied in many recent papers on Keisler measures \cite{lots_of_authors, AtticusPill, me2}.
Forking formulas are always universally measure zero, and in stable theories we have that $F(\emptyset)=\mathcal{O}(\emptyset)$. It was unclear whether this equality continued to hold in either simple or NIP theories until recent counterexamples in \cite{lots_of_authors} and \cite{AtticusPill} respectively. In \cite{me2}, the third author showed that for some simple $\omega$-categorical Hrushovski constructions $F(\emptyset)\subsetneq \mathcal{O}(\emptyset)$, whilst Lemma \ref{F=O} yields that $F(\emptyset)=\mathcal{O}(\emptyset)$ in NIP $\omega$-categorical theories. All counterexamples mentioned are built specifically to witness that $F(\emptyset)\subsetneq\mathcal{O}(\emptyset)$.\\

In this section we show instead that this phenomenon is very common in the simple $\omega$-categorical world, even in very model-theoretically tame structures (i.e. supersimple, of $\mathrm{SU}$-rank $1$, one-based with trivial forking, with elimination of hyperimaginaries and weak elimination of imaginaries). In particular, in Theorem \ref{IKMdichotomy} we show that $k$-overlap closed structures with disjoint $n$-amalgamation for $n\leq 3$ have $F(\emptyset)\subsetneq \mathcal{O}(\emptyset)$ as long as they are not random. This is easy to see for the universal homogeneous tetrahedron-free $3$-hypergraph: from Corollary \ref{cor:tetfreehypIKM} the measure $\mu(R(x, a, b)\wedge R(x, a, c)\wedge R(x, b,c))$ is $0$ when $abc$ form a hyperedge, since the formula is inconsistent and must also be $0$ when $abc$ do not form a hyperedge by exchangeability. But in the latter case, the formula $R(x, a, b)\wedge R(x, a, c)\wedge R(x, b,c)$ is consistent and so does not fork over $\emptyset$ since forking is trivial, yielding that $F(\emptyset)\subsetneq\mathcal{O}(\emptyset)$ in $\mathcal{H}^3_4$. Below we generalise this idea.

\begin{lemma}\label{syntacticfact} Let $\mathcal{C}$ be a hereditary class of relational structures with disjoint $2$-amalgamation. Suppose that disjoint $k$-amalgamation fails for some $k>2$. Then, for some $n$ there are $1$-point disjoint $n$-amalgamation problems $(A_I\vert I\subsetneq [n])$ and $(A'_I\vert I\subsetneq [n])$ over the same set $A$, where $(A_I\vert I\subsetneq [n])$ has no solution, $(A'_I\vert I\subsetneq [n])$ has a solution and $A_I=A'_I$ for $n\in I$.
\end{lemma}
\begin{proof}
    From Fact \ref{basicndap}, we know disjoint $n$-amalgamation for all $n\geq 2$ holds if and only if $1$-point disjoint $n$-amalgamation for all $n\geq 2$ holds. So let $n$ be minimal such that some $1$-point disjoint $n$-amalgamation problem $(A_I\vert I\subsetneq [n])$ over $A$ fails. $(A_{I\cup\{n\}}\vert I\subsetneq [n-1])$ is an $(n-1)$-amalgamation problem over $a_n$. By minimality of $n$, it has a solution $B_{[n]}$. For $I\subsetneq [n]$, let $A'_I=B_{[n]}\upharpoonright I$. Then, $B_{[n]}$ is a solution to $(A'_I\vert I\subsetneq [n])$, and $A'_I=A_I$ for $n\in I$, as desired.
\end{proof}

The following can be proven essentially in the same way as Lemma \ref{Gembedding}:

\begin{lemma}\label{homrep} Let $\mathcal{M}$ be homogeneous structure. For each quantifier-free $(k+1)$-type $q$ let $R_q$ be a $k$-ary relation and let $\mathcal{L}'$ consist of all these relations. To each type $p\in S'_x(M)$, we associate the $\mathcal{L}^*=\mathcal{L}\cup\mathcal{L}'$-expansion $\mathcal{M}^*_p$ where for $b_1, \dots, b_k\in M$, 
\[\mathcal{M}^*_p\vDash R_q(b_1, \dots, b_k) \text{ if and only if } p\vDash q(x, b_1, \dots, b_k).\]
Then, the type space $S_x'(M)$ is represented by the space $S$ of $\mathcal{L}'$-expansions of $M$ of the form $\mathcal{M}^*_p$.
\end{lemma}

Using standard techniques, we can show that if $\mathcal{M}$ has disjoint $n$-amalgamation for all $n$ and is in a finite relational language,  then there is an invariant Keisler measure $\mu$ in the variable $x$ assigning positive measure to every non-forking formula:

\begin{lemma}\label{nDAPlemma} Suppose $\mathcal{M}$ is a homogeneous structure in a finite relational language whose age has disjoint $n$-amalgamation for all $n$. Then, there is an invariant Keisler measure $\mu$ in the variable $x$ assigning positive measure to every non-forking formula.
\end{lemma}
\begin{proof} By Fact \ref{trivialforkfact},  $\mathcal{M}$ is supersimple with trivial forking. Using the representation in Lemma \ref{homrep}, we need to build an $\mathrm{Aut}(\MM)$-invariant measure on $S$ which assigns each $R_q$ for $q$ not implying that $x=y_k$ positive measure. We may assume without loss of generality that $\mathcal{M}$ is transitive: if $\mathcal{M}$ is not transitive, one can just run the same construction as below separately for each orbit and take a weighted sum of the resulting measures\\

We build a measure on $S$ inductively. for each quantifier-free $2$-type $q(x, y)$ we have a relation $R_q(y)$ in $\mathcal{L}'$. Choose independently at random for each vertex $a\in M$ which $R_q$ holds for it. Suppose that at step $k$ we have assigned each $k$-tuple from $M$ a relation $R_q$ for $q$ a quantifier-free $(k+1)$-type. Now, at step $k+1$ consider a $k+1$-tuple $(a_1, \dots, a_{k+1})\in M^{k+1}$. Each $k$-subset of $\overline{a}$ is assigned some relation $R_q$ for some $k+1$-type $q$. In particular, $\mathrm{qftp}(a_1, \dots, a_{k+1})$ forms a basic disjoint $k+2$-amalgamation problem together with $(q(x_{k+2}, \dots x_I)\vert I\subsetneq [k+1])$ and this problem will have some solutions, $q_1, \dots, q_m$. Choose independently at random between these which $R_q$ holds for $\overline{a}$. We continue this process for all $k$ (once we reach the maximal arity in $\mathcal{L}$ we don't need to do this anymore, though there would be no problem in running this process also if the language was countable but with finitely many relations in each arity). The resulting measure on $\mathrm{Struc}_{\mathcal{L}'}(M)$ is $\mathrm{Aut}(\MM)$-invariant, and it concentrates on $S$ since every quantifier-free $\mathcal{L}'$-formula over a tuple $\overline{a}$ which is given positive measure corresponds to a formula over $\overline{a}$ in the variable $x$ contained in some consistent $p\in S'(M)$ by disjoint $n$-amalgamation. Moreover, any consistent $\mathcal{L}$-formula with parameters from $\mathcal{M}$ is assigned positive measure by the construction. This completes the proof.
\end{proof}

\begin{theorem}\label{IKMdichotomy} Let $\mathcal{M}$ be a $k$-transitive homogeneous structure in a finite $(k+1)$-ary language which is $k$-overlap closed and has disjoint $n$-amalgamation for $n\leq 3$. Then, any invariant Keisler measure for $\mathcal{M}$ in $x$ is exchangeable. Moreover, 
\begin{enumerate}
    \item\label{case1} EITHER: $\mathrm{Age}(\MM
    )$ has disjoint $n$-amalgamation for all $n$. In this case there is an IKM $\mu$ assigning positive measure to every non-forking formula;
    \item\label{case2} OR: $\mathrm{Age}(\MM)$ fails disjoint $n$-amalgamation for some $n$. In this case,
    \[F(\emptyset)\subsetneq \mathcal{O}(\emptyset).\]
\end{enumerate}
    
\end{theorem}
\begin{proof} The first statement is Theorem \ref{thm:IKMareSI}. By Fact \ref{basicndap}, $\mathcal{M}$ is simple with trivial forking.  The first case is already dealt with in Lemma \ref{nDAPlemma}. So, suppose that disjoint $n$-amalgamation fails for some $n$. By Lemma \ref{syntacticfact}, there are $1$-point disjoint $n$-amalgamation problems $(A_I\vert I\subsetneq [n])$ and $(A'_I\vert I\subsetneq [n])$ over the same set $A$, where $(A_I\vert I\subsetneq [n])$ has no solution, $(A'_I\vert I\subsetneq [n])$ has a solution and $A_I=A'_I$ for $n\in I$. Now, take $\overline{a}\cong A_{[n-1]}$ and $\overline{b}\cong A'_{[n-1]}$ in $M$ and let $q(x, \overline{y})$ be the conjunction for $I\subsetneq [n-1]$ of the formulas isolating the quantifier-free types of $A_{I\cup\{n\}}$, where the variable $x$ denotes $a_n$. Since $(A_I\vert I\subsetneq [n])$ fails disjoint $n$-amalgamation, for any invariant Keisler measure,
\[\mu\left(q(x, \overline{a})\right)=0.\]
But by $S_\infty$-invariance of $\mu$, 
\[\mu\left(q(x, \overline{b})\right)=0.\]
The latter formula does not fork over $\emptyset$ since it has realisations and forking is trivial. Hence, it yields a non-forking formula which is universally measure zero.
\end{proof}

Theorem \ref{IKMdichotomy} suggests that most simple homogeneous structures have $F(\emptyset)\subsetneq\mathcal{O}(\emptyset)$. In fact, for each finite relational language $\mathcal{L}$, there are only finitely many $\mathcal{L}$-structures with disjoint $n$-amalgamation for all $n$. However, as long as there are relations with arity $\geq 2$, there will be uncountably many $k$-overlap closed $\mathcal{L}$-structures with disjoint $n$-amalgamation for $n\in\{2,3\}$.

\begin{corollary}\label{manycounterex} There are $2^{\aleph_0}$ homogeneous ternary structures which are supersimple, with $\mathrm{SU}$-rank $1$, one-based and with trivial forking but for which $F(\emptyset)\subsetneq\mathcal{O}(\emptyset)$.
\end{corollary}
\begin{proof} Just apply Theorem \ref{IKMdichotomy} to the $2^{\aleph_0}$ ternary homogeneous supersimple structures constructed by Koponen from Definition \ref{Kopconstr}. None of them, except for the trivial one,  satisfy disjoint $n$-amalgamation for all $n$ since they have constraints of size $>3$.
\end{proof}

\begin{remark}
    The structures in Theorem \ref{IKMdichotomy} and Corollary \ref{manycounterex} also offercounterexamples to a question of Elwes and Macpherson \cite{EM}, who asked whether $\omega$-categorical supersimple structures of finite $\mathrm{SU}$-rank are MS-measurable in the sense of Macpherson and Steinhorn \cite{MS}. A structure is MS-measurable when it is equipped with a dimension operation and a system of invariant Keisler measures satisfying various desirable properties, including a version of Fubini's Theorem. Pseudofinite fields \cite{CVDMI} and smoothly approximable structures are MS-measurable \cite{FSFT}. Previous counterexamples to the question of Elwes and Macpherson using Hrushovski constructions were given in \cite{Measam} and \cite{me}. Since these structures are not particularly model theoretically tame, being not one-based, it is natural to ask whether one-based simple $\omega$-categorical structures are MS-measurable. Since MS-measurable structures are such that $F(\emptyset)=\mathcal{O}(\emptyset)$ \cite{me2}, all the examples in Theorem \ref{IKMdichotomy} satisfying $F(\emptyset)\subsetneq\mathcal{O}(\emptyset)$ are not MS-measurable. In \cite{myPhD} the third author proved that $\mathcal{H}_4^3$ is not MS-measurable using different techniques.
\end{remark}

\section{The Invariant Extension Property and the \texorpdfstring{$\omega$}{omega}-categorical NIP setting}\label{sec:nip}

Most of our results in previous sections concern structures with the independence property. In this section we focus on invariant Keisler measures for NIP $\omega$-categorical structures.  We identify a property of NIP theories shared also by some of the $\mathrm{NSOP}_4$ examples we study (i.e. the generic $K_n$-free graphs and $\mathcal{P}_n^3$, see Corollary \ref{cor:strangemeas}), which we call the invariant extension property (IEP): every formula which is not universally measure zero is contained in an $\mathrm{Aut}(\MM/\mathrm{acl}^{eq}(\emptyset))$-invariant type. Under this assumption, we show that all ergodic invariant Keisler measures concentrate on the $\mathrm{Aut}(\MM)$-orbit of an $\mathrm{Aut}(\MM/\mathrm{acl}^{eq}(\emptyset))$-invariant type. This gives a particularly nice description of the invariant Keisler measures as integrals of weighted averages of invariant types as long as no $\mathrm{Aut}(\MM/\mathrm{acl}^{eq}(\emptyset))$-invariant type has uncountably many $\mathrm{Aut}(\MM)$-conjugates, a condition satisfised by all known examples of NIP $\omega$-categorical theories and all known examples of homogeneous structures in a finite language.\\

We build on the work \cite{Ensley} of Ensley, who gives a classification of invariant Keisler measures in NIP $\omega$\-/categorical structures under two additional technical assumptions: (i) that forking agrees with dividing and (ii) a strengthening of the statement that there are only finitely many types $p\in S_x(M)$ which do not fork over $\emptyset$. Since his assumptions are satisfied in stable $\omega$-categorical structures, Ensley already provides a full classification of invariant Keisler measures there. However, examples of NIP $\omega$\-/categorical structures where forking and dividing disagree are well-known (\cite[Example 3.12]{Ensley} and \cite[Exercise 7.1.6]{TZ}).\\





We also note that in NIP $\omega$-categorical structures, $F(\emptyset)=\mathcal{O}(\emptyset)$. This is different from the general NIP setting, where \cite{AtticusPill} gave an example where the containment is strict. Some of the results in this section can be extracted from \cite{NIPinv}. Nevertheless, we include them to clarify the behaviour of invariant Keisler measures which is exhibited in both NIP and and many strictly NSOP $\omega$-categorical structures.\\

\begin{definition} We say that $\mathcal{M}$ has the \textbf{invariant extension property} (IEP) if every $\mathcal{L}(M)$-formula which is not universally measure zero is contained in an $\mathrm{Aut}(\MM/\allowbreak\mathrm{acl}^{eq}(\emptyset))$\-/invariant type in $S_x(M)$. We say that the $\omega$\-/categorical theory $T$ has the IEP if its countable model has the IEP.
\end{definition}

In addition to the strictly $\mathrm{NSOP}_4$ examples (i.e. the $K_n$-free generic graphs, and the $n$-petal-free $3$-hypergraphs), it is easy to see that NIP $\omega$\-/categorical theories have the invariant extension property \cite{CasNIP}:

\begin{fact}\label{nonfNIP} Let $\mathcal{M}$ be $\omega$\-/categorical NIP, and $p\in S_x(M)$. Then, the following are equivalent:
\begin{itemize}
\item $p$ does not fork over $\emptyset$;
\item $p$ is $\mathrm{Aut}(\MM/\allowbreak\mathrm{acl}^{eq}(\emptyset))$-invariant.
\end{itemize}
\end{fact}
\begin{proof} From \cite[Proposition 3]{CasNIP}, we have that for $M$ $\omega$-saturated over $\emptyset$, $p$ does not fork over $\emptyset$ if and only if it is $\mathrm{bdd}(\emptyset)$-invariant. Since $\mathcal{M}$ is $\omega$\-/categorical, it is $\omega$-saturated over $\emptyset$. By \cite[Theorem 18.14.3]{CasSimp}, in small theories, over finite sets, Kim-Pillay types coincide with strong types (i.e. types over bounded closures coincide with types over imaginary algebraic closures). Hence, the result follows. 
\end{proof}

We know from \cite{lots_of_authors} that any formula which is not universally measure zero is non-forking. Hence, Fact \ref{nonfNIP} implies that NIP $\omega$\-/categorical structures have the IEP:

\begin{corollary} \label{cor:NIPIEP} If $\mathcal{M}$ is the countable model of a NIP $\omega$\-/categorical theory, then $\mathcal{M}$ has the invariant extension property.
\end{corollary}

\begin{definition} Let $\mu$ be a Borel probability measure on $X$. The \textbf{support} of $\mu$ consists of all points $p\in X$ such that every open neighbourhood of $p$ is assigned positive measure:
\[\mathrm{Supp}(\mu):=\{p\in X \vert \text{for all open } U \text{ such that } p\in U, \mu(U)>0\}.\]
\end{definition}

We show below that the IEP implies that any invariant Keisler measure is supported by $\mathrm{Aut}(\MM/\allowbreak\mathrm{acl}^{eq}(\emptyset))$-invariant types in $S_x(M)$. 

\begin{lemma}\label{invsupport}
     Let $\mathcal{M}$ be a countable $\omega$\-/categorical structure with the IEP. Let $\mu$ be an invariant Keisler measure on $\mathcal{M}$ in the variable $x$. Then, any $p\in\mathrm{Supp}(\mu)$ is $\mathrm{Aut}(\MM/\allowbreak\mathrm{acl}^{eq}(\emptyset))$-invariant.
\end{lemma}
\begin{proof}
We prove that any type which is not $\mathrm{Aut}(\MM/\allowbreak\mathrm{acl}^{eq}(\emptyset))$-invariant cannot be in $\mathrm{Supp}(\mu)$. Suppose that $p\in S_x(M)$ is not $\mathrm{Aut}(\MM/\allowbreak\mathrm{acl}^{eq}(\emptyset))$-invariant. Then, there is $\xi(x,d)$ and $\tau\in \mathrm{Aut}(\MM/\allowbreak\mathrm{acl}^{eq}(\emptyset))$ such that
\[\xi(x,d)\wedge \neg\xi(x,\tau(d))\in p.\]
This formula cannot be contained in any $\mathrm{Aut}(\MM/\allowbreak\mathrm{acl}^{eq}(\emptyset))$-invariant type. Hence, by the IEP it is universally measure zero. Hence, $p\not\in \mathrm{Supp}(\mu)$ since it is contained in an open set of measure zero.
\end{proof}

\begin{remark}\label{rem:IEPmeas} Since the action of $\mathrm{Aut}(\MM)$ on $ S_x(M)$ is Borel, being continous, $\mathrm{Aut}(\MM)$-orbits are Borel \cite[Theorem 15.14]{KechrisDST} and so ergodic measures concentrate on orbits. Hence, by Lemma \ref{invsupport}, an ergodic measure $\mu$ must concentrate on the $\mathrm{Aut}(\MM)$-orbit of a single $\mathrm{Aut}(\MM/\allowbreak \mathrm{acl}^{eq}(\emptyset))$-invariant type $p$. If such type $p$ has only finitely many $\mathrm{Aut}(\MM)$-conjugates $\{p_1,\allowbreak \dots, p_n\}$. Then, 
\[\mu=\frac{1}{n}\sum_{i=0}^n p_i,\]
    where $p_i$ is the Dirac measure concentrated on $p_i$.
Note that it is impossible for $\mu$ to concentrate on a countable orbit: it would assign every type the same measure $0$, and so by countable additivity the entire orbit would have measure $0$.

\end{remark}

Hence, we get the following result:

\begin{corollary}\label{cor:IEPchar} Let $\mathcal{M}$ be a countable $\omega$-categorical structure with the IEP. Suppose that no $\mathrm{Aut}(\MM/\allowbreak\mathrm{acl}^{eq}(\emptyset))$-invariant type has uncountably many $\mathrm{Aut}(\MM)$-conjugates. Let $\mu$ be an invariant Keisler measures on $\mathcal{M}$. Let $S$ be the set of $\mathrm{Aut}(\MM)$-conjugacy classes of $\mathrm{Aut}(\MM/\allowbreak\mathrm{acl}^{eq}(\emptyset))$-invariant types in $S_x(M)$ with finitely many conjugates. Then, there is a unique Borel probability measure $\nu$ on $S$ such that 
\[\mu=\int_{\overline{p}\in S} \mu_{\overline{p}} \ \mathrm{d}\nu,\]
where $\mu_{\overline{p}}$ is the probability measure given by
\[\mu_p=\frac{1}{n}\sum_{i=0}^n p_i,\]
where the $p_i$ are the $\mathrm{Aut}(\MM)$-conjugates in the conjugacy class $\overline{p}\in S$.
\end{corollary}
\begin{proof} The proof just follows from the characterisation of ergodic measures in Remark \ref{rem:IEPmeas} and from the ergodic decomposition \cite[p.77]{Choquet}.
\end{proof}


\begin{remark} In stable $\omega$-categorical theories, by the finite equivalence relation theorem \cite[Ch. III, Theorem 2.8]{shelah1990classification}, any 
$\mathrm{Aut}(\MM/\mathrm{acl}^{eq}(\emptyset))$-invariant type has finitely many $\mathrm{Aut}(\MM)$-conjugates, and so Corollary \ref{cor:IEPchar} already describes all the possible invariant Keisler measures, as was already proved by Ensley \cite[$\S$ 4.2]{Ensley}. We do not know any example of a NIP $\omega$-categorical theory with an $\mathrm{Aut}(\MM/\allowbreak\mathrm{acl}^{eq}(\emptyset))$-invariant type with uncountably many $\mathrm{Aut}(\MM)$-conjugates. Indeed, it is easy to see that any $\mathrm{Aut}(\MM/\mathrm{acl}^{eq}(\emptyset))$-invariant type will have only finitely many $\mathrm{Aut}(\MM)$-conjugates as long as $\mathrm{Aut}(\MM)/\allowbreak\mathrm{Aut}(\MM/\mathrm{acl}^{eq}(\emptyset))$ is finite, which in an $\omega$-categorical context corresponds to the theory of $M$ being $G$-finite (over $\emptyset$) \cite{CatMod, FOAmen}. There are no known non-$G$-finite homogeneous structures in a finite language and no known NIP $\omega$-categorical examples. So, it is possible the classification in Corollary \ref{cor:IEPchar} holds for all NIP $\omega$-categorical structures. It is also easy to see that the standard example of a non-$G$-finite $\omega$-categorical theory (the Fra\"{i}ss\'{e} limit of finite structures with an equivalence relation $E_n$ on $n$-tuples with two classes for each $n<\omega$ \cite{CatMod}) does have $\mathrm{Aut}(\MM/\mathrm{acl}^{eq}(\emptyset))$-invariant types with uncountably many conjugates.
\end{remark}

From the above characterisation of invariant Keisler measures we can see that $\omega$-categorical structures with the IEP satisfy a rather extreme version of $\mu(\phi(x, ab)\wedge\psi(x, ac)\wedge\xi(x, bc))$ only depending on the types of the pairs in $abc$:

\begin{corollary} Let $\mathcal{M}$ be $\omega$-categorical $\mathcal{L}$-structure with the IEP and $\mathrm{acl}^{eq}(\emptyset)=\mathrm{dcl}^{eq}(\emptyset)$. Let $\{\phi_i(x, \overline{y}_i) \vert i\leq n\}$ be a set of $\mathcal{L}$-formulas and $(\overline{a}_i\vert i\leq n)$, $(\overline{a}_i'\vert i\leq n)$ be tuples in the variables $y_i$ such that $\overline{a}_i\equiv \overline{a}_i'$ for $i\leq n$. Then, for any invariant Keisler measure $\mu$
\[\mu\left(\bigwedge_{i\leq n}\phi_i(x, \overline{a}_i)\right)=\mu\left(\bigwedge_{i\leq n}\phi_i(x, \overline{a}_i')\right).\]
In particular, if $\mathcal{M}$ is homogeneous $k$-transitive and all of its relations are of arity $k+1$, all of its invariant Keisler measures are exchangeable.
\end{corollary}
\begin{proof} Since $\mathrm{acl}^{eq}(\emptyset)=\mathrm{dcl}^{eq}(\emptyset)$, by Corollary \ref{cor:IEPchar} every invariant Keisler measure of $\mathcal{M}$ is an integral average of $\mathrm{Aut}(\MM)$-invariant types. Consider an $\mathrm{Aut}(\MM)$-invariant type $p$. Since $\overline{a}_i\equiv \overline{a}_i'$ for each $i\leq n$, we have that
\[\bigwedge_{i\leq n}\phi_i(x, \overline{a}_i)\in p \text{ if and only if } \bigwedge_{i\leq n}\phi_i(x, \overline{a}_i')\in p,\]
yielding the first part of the statement. The second part is just a consequence of quantifier elimination and $k$-transitivity.
\end{proof}
So, for example, we can see that, if they exist, the invariant Keisler measures for various homogeneous NIP structures such as $(\mathbb{Q}, <)$, and homogeneous $B, C,$ and $D$-relations are all exchangeable. Of course, these may not exist: for $(\mathbb{Q}, \mathrm{cyc})$, the circular order, it is easy to see that it has no invariant types and so no invariant Keisler measures.\\

Regardless of whether the description of invariant Keisler measures in Corollary \ref{cor:IEPchar} is true of all $\omega$-categorical NIP theories, we can still show of them that for every non-forking formula $\phi(x)$, there is an invariant Keisler measure assigning it positive measure. The proof of this follows \cite[Proposition 4.7]{NIPinv}, with the help of the following fact:

\begin{fact}[{\cite[Lemma 2.4]{EvansTsk}}]\label{aclfact} Let $\mathcal{M}$ be $\omega$\-/categorical. Fix the variable $x$. Then, there is finite $A\subseteq \mathrm{acl}^{eq}(\emptyset)$ such that for any $b, b'\in M^{x}$, we have
\[ b\equiv_{A} b' \text{ if and only if } b\equiv_{\mathrm{acl}^{eq}(\emptyset)} b'.\]
In particular, for any finite variable $x$, there are only finitely many types in the variable $x$ over $\mathrm{acl}^{eq}(\emptyset)$. The set $A$ can be also taken to be $\mathrm{Aut}(\MM^{eq})$-invariant setwise. 
\end{fact}

\begin{lemma}\label{F=O} Let $\mathcal{M}$ be the countable model of an $\omega$\-/categorical theory. If $\phi(x,b)$ is contained in an $\mathrm{Aut}(\MM/\allowbreak\mathrm{acl}^{eq}(\emptyset))$-invariant type, then it is not universally measure zero.\\

In particular, if $\mathcal{M}$ is NIP, an $\mathcal{L}(M)$-formula $\phi(x,b)$ forks over $\emptyset$ if and only if it is universally measure zero.
\end{lemma}
\begin{proof}
    We follow the construction from \cite[Proposition 4.7]{NIPinv}. Pick an $\mathrm{Aut}(\MM/\mathrm{acl}^{eq}(\emptyset))$-invariant type $p$ containing $\phi(x;b)$. For any formula $\psi(x,m)$, count the different extensions of $\mathrm{tp}(m/\emptyset)$ to a type over $\mathrm{acl}^{eq}(\emptyset)$ (note there are only finitely many by Fact \ref{aclfact}), and count how many such types are represented in the parameters of a formula $\psi(x,m') \in p$. Take the ratio to be $\mu(\psi(x,m))$. $Aut(M)$-invariance and additivity follow,  and we have $\mu(\phi(x;b)) > 0$ since $\phi(x;b) \in p$.
\end{proof}


\section{Final questions}\label{sec:finq}

The notions and results from this paper offer many directions for further research. Firstly, our main results can be seen as working towards a more general problem:


\begin{problem} Let $H\leq G$ be closed subgroups of $S_\infty$ and $\mathcal{C}'$ a hereditary class of structures. Under what conditions on $H, G,$ and $\mathcal{C}'$ can we say that all $H$-invariant probability measures on the space of countable structures with age in $\mathcal{C}'$ are indeed $G$-invariant?
\end{problem}

The main aim of this paper was studying meaningful classes of (oligomorphic) permutation groups $H$ and classes $\mathcal{C}'$ for which we would get such results for $G=S_\infty$. It is clear that more can be said for other interesting choices of $G$ (cf. \cite{AckerAut, CraneTown} and \cite[Lemma 10.15]{chernikov2020hypergraph}). An analogue of this problem can also be asked regarding consistent random expansions of hereditary classes.\\

At the moment, much remains open even choosing $G$ to be $S_\infty$. In particular, our techniques only work under adequate assumptions of $k$-overlap closedness. It would be interesting to understand whether one could obtain similar results under the assumption of free amalgamation, as suggested by Proposition \ref{prop:edgecase}:


\begin{question}\label{q:freeam}
Suppose that $\mathcal{M}$ is a $k$-transitive homogeneous structure in a language of arity $>k$ whose age has free amalgamation. Is every expansion by a hereditary class of $\LL'$-structures with labelled growth rate $O(e^{n^{k+\delta}})$ for every $\delta > 0$ exchangeable?     
\end{question}

There is the stronger question that asks, under the assumptions above on $\mathcal{M}$, whether its age is \closed. The class of 4-uniform hypergraphs forbidding the structure described in Remark \ref{rem:4aryExample} is a concrete open case that seems to require addressing hypergraph Tur\'an-type problems forbidding multiple configurations, possibly a separate such problem for each arity $r$ of the desired $r$-hypergraph $\mathbb{K}$. Most literature on such problems focuses on obtaining the maximal number of hyperedges and omitting a single configuration \cite{keevash2011hypergraph}. On top of offering potentially interesting (and hard) problems in asymptotic combinatorics, this directions should be of independent interest since it is really asking for which hereditary classes $\mathcal{C}$ one can obtain a denser version of the random placement construction of Ne{\v{s}}et{\v{r}}il and R{\"o}dl \cite{nesetril1978probabilistic} (which corresponds to the construction obtained in the case of $1$-overlap closedness).\\

Alternatively, one may ask whether model theoretic techniques may aid in answering Question \ref{q:freeam} whilst circumventing the potentially hard combinatorial problems that arise from our techniques. In this case, we would expect higher arity model theoretic properties to play a role. In particular, the notion of $k$-dependence has emerged as a way of characterising $k$-arity of a theory (or formula) \cite{ndependence, chernikov2020hypergraph}. Perhaps, what underlies Theorem \ref{theorem:Chernoff} is that if $\mathcal{M}$ is $k$-transitive and has no "hidden" $k$-ary structure (like the case of the parity $k+1$-hypergraphs) then invariant random expansions by classes with growth rate $O(e^{n^{k+\delta}})$ for every $\delta > 0$ are exchangeable. In this case, one may expect a positive answer to some version of the following question:



\begin{question}\label{q:ipk} Suppose that $\mathcal{M}$ is a $k$-transitive homogeneous structure. Suppose further that any $\emptyset$-definable relation of $\mathcal{M}$ with $\mathrm{NIP}_k$ is already $\emptyset$-definable in $(\mathbb{N}, =)$. Is every expansion by a hereditary class of $\LL'$-structures with labelled growth rate $O(e^{n^{k+\delta}})$ for every $\delta > 0$ exchangeable?
\end{question}

Note this is strictly more general than Question \ref{q:freeam} since a homogeneous structure with free amalgamation in a language of arity $>k$ has $\mathrm{IP}_k$ (essentially by the definition of $k$-independence and the characterisation of free amalgamation in terms of omitted substructures being $2$-irreducible). The idea behind Question \ref{q:ipk} is that $\mathrm{IP}_k$ would capture the absence of "hidden" $k$-ary structure which obstructed exchangeability in the case of the parity $k$-hypergraphs in Section \ref{sec:kaygraphs}. A difficulty with this question for $k>1$ is that relatively few examples of homogeneous structures in high arity are known, so more pathological behaviour may only be exhibited beyond the few known examples (c.f. \cite{Kopconstr, cherlin2021ramsey}). A positive answer in the case of $k=1$ seems plausible given the positive results for invariant expansions by unary predicates and linear orders in \cite{JahelT}, though even in this case the question seems non-trivial.\\

It would be good to gain a more systematic understanding of non-trivial sources of failure of exchangeability. In \cite{CRO}, it is shown that bounded degree classes of graphs admit non-exchangeable order expansions. Like the parity $k$-hypergraphs case, this can be seen as another instance where expanding by a class with a similar labelled growth rate causes a failure of exchangeability. It sounds reasonable to ask the following:

\begin{question} Can we understand more systematically in which ways non-exchangeable consistent random expansions of $\mathcal{C}$ by $\mathcal{C}'$ may arise when $\mathcal{C}$ and $\mathcal{C}'$ have similar growth rates?
\end{question}

A related problem to what we study in this paper is that of obtaining Aldous-Hoover-like representations for the invariant random expansions of different structures $\mathcal{M}$. In this paper, once we get exchangeability we can use the actual Aldous-Hoover theorem to represent the relevant invariant random expansions. However, in doing this, we must restrict the growth-rate of the class of structures we are expanding by. For example, it is obvious that the random graph $\mathcal{R}_2$ has non-exchangeable invariant graph expansions: consider the $\mathrm{Aut}(\mathcal{R}_2)$-invariant measure concentrating on a graph isomorphic to the underlying copy of $\mathcal{R}_2$. Still, \cite{AckerAut, CraneTown} obtain a representation theorem for invariant random expansions by relations of arbitrary arity for homogeneous structures with disjoint $n$-amalgamation for all $n$. Compared to the Aldous-Hoover Theorem (Fact \ref{AldousHoover}), in their representation, the Borel function $f$ determining whether a $k$-ary relation $R$ holds of a tuple $\overline{a}$ is allowed to depend also on the induced substructure of $\mathcal{M}$ on the tuple $\overline{a}$ on top of depending on the i.i.d. $\mathrm{Uniform}[0,1]$ random variables $(\xi_I)_{I\subseteq\overline{a}}$. Hence, for example, we can see exchangeability of invariant graph expansions of $\mathcal{R}_3$ as a consequence of the representation theorem of \cite{AckerAut, CraneTown} and $2$-transitivity. Thus, a natural question arises of whether the exchangeability results in our paper are also a shadow of a similar representation theorem holding under more general conditions than disjoint $n$-amalgamation for all $n$. In particular, we ask the following:

\begin{question}
    Can we prove a representation theorem of the form of Crane and Towsner's \cite[Theorem 3.2]{CraneTown}, and Ackerman's \cite{AckerAut} for invariant random expansions of homogeneous structures with free amalgamation by structures of arbitrary arity? In particular, is there such a representation for invariant random expansions of the generic triangle-free graph?
\end{question}

The case of the generic triangle-free graph $\mathcal{H}_3^2$ offers a good case study. If we had an Aldous-Hoover-like representation for its invariant random expansions such as in \cite{AckerAut, CraneTown}, we would have the somewhat surprising result that all of its invariant random expansions extend to an invariant random expansion of the random graph. Still, it looks unclear how measures not coming from such a representation may arise.\\

Finally, it would be interesting to gain a better understanding of invariant random expansions in a general model theoretic setting. In Section \ref{sec:IKMconnection} we proved that invariant Keisler measures can be seen as a special case of invariant random expansions. However, whilst invariant Keisler measures are well-understood in a $\mathrm{NIP}$ setting, this is not the case for invariant random expansions. Indeed, most of our results work specifically for structures with the independence property. It would be interesting to understand what the invariant random expansions of various stable or $\mathrm{NIP}$ structures look like. Two $\omega$-categorical case studies would be infinite dimensional vector spaces over finite fields and homogeneous $\mathrm{C}$-relations. Moreover, it would be interesting to study invariant random expansions of sufficiently saturated and strongly homogeneous models of some model-theoretically tame theories which are not necessarily $\omega$-categorical. Whilst the general problem of describing the spaces of invariant random expansions of a given structure looks very challenging \cite{Kallbooksym}, at least some techniques from model theory, especially those exploiting stability of probability spaces \cite{approxsubg, AER, chernikov2020hypergraph}, may be of help in the case of invariant random expansions beyond that of invariant Keisler measures. 

 \printbibliography

@article{nguyen2023induced,
  title={{Induced subgraph density. VI. Bounded VC-dimensio}},
  author={Nguyen, Tung and Scott, Alex and Seymour, Paul},
  journal={arXiv preprint arXiv:2312.15572},
  year={2023}
}

@article{kingman1978uses,
  title={Uses of exchangeability},
  author={Kingman, John F. C. },
  journal={The Annals of Probability},
  volume={6},
  number={2},
  pages={183--197},
  year={1978},
  publisher={Institute of Mathematical Statistics}
}

@article{diaconis2013interval,
  title={Interval graph limits},
  author={Diaconis, Persi and Holmes, Susan and Janson, Svante},
  journal={Annals of combinatorics},
  volume={17},
  pages={27--52},
  year={2013},
  publisher={Springer}
}

@misc{brown1973some,
  title={Some extremal problems on r-graphs, New directions in the theory of graphs (Proc. Third Ann Arbor Conf., Univ. Michigan, Ann Arbor, Mich, 1971)},
  author={Brown, William G. and Erd{\H{o}}s, P{\'a}l and T. S{\'o}s, Vera},
  year={1973},
  publisher={Academic Press, New York}
}

@article{shangguan2020degenerate,
  title={Degenerate Tur{\'a}n densities of sparse hypergraphs},
  author={Shangguan, Chong and Tamo, Itzhak},
  journal={Journal of Combinatorial Theory, Series A},
  volume={173},
  pages={105228},
  year={2020},
  publisher={Elsevier}
}

@article{terry2018VC,
  title={$VC_\ell$-dimension and the jump to the fastest speed of a hereditary $\mathcal{L}$-property},
  author={Terry, Caroline},
  journal={Proceedings of the American Mathematical Society},
  volume={146},
  number={7},
  pages={3111--3126},
  year={2018}
}

@inproceedings{ryll1957stationary,
  title={On stationary sequences of random variables and the de Finetti's equivalence},
  author={Ryll-Nardzewski, Czes{\l}aw},
  booktitle={Colloquium Mathematicum},
  volume={4},
  number={2},
  pages={149--156},
  year={1957},
  organization={Polska Akademia Nauk. Instytut Matematyczny PAN}
}

@article{delcourt2022finding,
  title={Finding an almost perfect matching in a hypergraph avoiding forbidden submatchings},
  author={Delcourt, Michelle and Postle, Luke},
  journal={arXiv preprint arXiv:2204.08981},
  year={2022}
}

@article{glock2024conflict,
  title={Conflict-free hypergraph matchings},
  author={Glock, Stefan and Joos, Felix and Kim, Jaehoon and K{\"u}hn, Marcus and Lichev, Lyuben},
  journal={Journal of the London Mathematical Society},
  volume={109},
  number={5},
  pages={e12899},
  year={2024},
  publisher={Wiley Online Library}
}

@article{hubicka2017ramsey,
  title={Ramsey Classes with Closure Operations},
  author={Hubicka, Jan and Ne{\v{s}}etril, Jaroslav},
  journal={arXiv preprint arXiv:1705.01924},
  year={2017}
}

@article{nesetril1978probabilistic,
  title={On a probabilistic graph-theoretical method},
  author={Ne{\v{s}}et{\v{r}}il, Jaroslav and R{\"o}dl, Vojt{\v{e}}ch},
  journal={Proceedings of the American Mathematical Society},
  volume={72},
  number={2},
  pages={417--421},
  year={1978}
}

@incollection{Wagner:ST,
	author = {Wagner, Frank O.~},
	booktitle = {Automorphisms of first-order structures},
	editor = {Kaye, R. and Macpherson, D.},
	pages = {153-180},
	publisher = {Clarendon Press},
	title = {Relational structures and dimensions},
	year = {1994}}

@inbook{EM, place={Cambridge}, series={London Mathematical Society Lecture Note Series}, title={A survey of asymptotic classes and measurable structures}, volume={2}, booktitle={Model Theory with Applications to Algebra and Analysis}, publisher={Cambridge University Press}, author={Elwes, Richard and Macpherson, Dugald}, editor={Chatzidakis, Z. and Macpherson, D. and Pillay, A. and Wilkie, A.}, year={2008}, pages={125-160}, DOI={10.1017/CBO9780511735219.004},  collection={London Mathematical Society Lecture Note Series}}

@article{MS,
copyright = {Copyright 2007, American Mathematical Society},
issn = {0002-9947},
journal = {Transactions of the American Mathematical Society},
number = {1},
issn = {0002-9947}, 
pages = {411-448},
publisher = {American Mathematical Society},
title = {One-dimensional asymptotic classes of finite structures},
volume = {360},
year = {2008},
author = {Macpherson, Dugald and Steinhorn, Charles},
address = {Providence, RI},
}

@article{Measam,
  title={Higher amalgamation properties in measured structures},
  author={Evans, David M.},
  journal={Model Theory},
  volume={2},
  number={2},
  pages={233--253},
  year={2023},
  publisher={Mathematical Sciences Publishers}
}

@book{TZ, place={Cambridge}, series={Lecture Notes in Logic}, title={A Course in Model Theory}, DOI={10.1017/CBO9781139015417}, publisher={Cambridge University Press}, author={Tent, Katrin and Ziegler, Martin}, year={2012}, collection={Lecture Notes in Logic}}

@book{Hodges, place={Cambridge}, series={Encyclopedia of Mathematics and its Applications}, title={Model Theory}, DOI={10.1017/CBO9780511551574}, publisher={Cambridge University Press}, author={Hodges, W.~}, year={1993}, collection={Encyclopedia of Mathematics and its Applications}}

@article{lots_of_authors,
  title={Invariant measures in simple and in small theories},
  author={Chernikov, Artem and Hrushovski, Ehud and Kruckman, Alex and Krupi{\'n}ski, Krzysztof and Moconja, Slavko and Pillay, Anand and Ramsey, Nicholas},
  journal={Journal of Mathematical Logic},
  volume={23},
  number={02},
  pages={2250025},
  year={2023},
  publisher={World Scientific}
}

@article{Homogeneous,
title = {A survey of homogeneous structures},
journal = {Discrete Mathematics},
volume = {311},
number = {15},
pages = {1599-1634},
year = {2011},
note = {Infinite Graphs: Introductions, Connections, Surveys},
issn = {0012-365X},
doi = {https://doi.org/10.1016/j.disc.2011.01.024},
url = {https://www.sciencedirect.com/science/article/pii/S0012365X11000422},
author = {Macpherson, Dugald}
}

@article{Kruck,
author = {Kruckman, Alex},
title = {{Disjoint $n$-Amalgamation and Pseudofinite Countably Categorical Theories}},
volume = {60},
journal = {Notre Dame Journal of Formal Logic},
number = {1},
publisher = {Duke University Press},
pages = {139 -- 160},

year = {2019},
doi = {10.1215/00294527-2018-0025},
URL = {https://doi.org/10.1215/00294527-2018-0025}
}

@book{NIP, place={Cambridge}, series={Lecture Notes in Logic}, title={A Guide to NIP Theories}, DOI={10.1017/CBO9781107415133}, publisher={Cambridge University Press}, author={Simon, Pierre}, year={2015}, collection={Lecture Notes in Logic}}

@article{KPsimp,
title = {Simple Theories},
journal = {Annals of Pure and Applied Logic},
volume = {88},
number = {2},
pages = {149-164},
year = {1997},
note = {Joint AILA-KGS Model Theory Meeting},
issn = {0168-0072},
doi = {https://doi.org/10.1016/S0168-0072(97)00019-5},
url = {https://www.sciencedirect.com/science/article/pii/S0168007297000195},
author = {Byunghan Kim and Anand Pillay}
}

@book{Kimsimp,
  title={Simplicity theory},
  author={Kim, Byunghan},
  year={2014},
  publisher={Oxford University Press}
}

@article{approxsubg,
 ISSN = {08940347, 10886834},
 URL = {http://www.jstor.org/stable/23072155},
 author = {Ehud Hruhsovski},
 journal = {Journal of the American Mathematical Society},
 number = {1},
 pages = {189--243},
 publisher = {American Mathematical Society},
 title = {Stable Group Theory and Approximate Subgroups},
 volume = {25},
 year = {2012}
}

@book{Choquet,
author = {Phelps, Robert R.},
address = {Berlin, Heidelberg},
edition = {2nd ed. 2001.},
isbn = {9783540487197},
publisher = {Springer Berlin Heidelberg},
series = {Lecture Notes in Mathematics, 1757},
title = {Lectures on Choquet's Theorem},
year = {2001},
}

@book{Boga,
author = {Bogachev, Vladimir I.},
address = {Berlin, Heidelberg},
edition = {1st ed. 2007.},
isbn = {1-280-74570-3},
publisher = {Springer Berlin Heidelberg},
series = {Measure Theory;},
title = {Measure Theory},
year = {2007},
}

@book{Poizat,
series = {Universitext},
publisher = {Springer},
isbn = {9780387986555},
year = {2000},
title = {A course in model theory : an introduction to contemporary mathematical logic},
address = {New York ; London},
author = {Poizat, Bruno},
lccn = {99053572},
}

@article{JahelT,
     author = {Jahel, Colin and Tsankov, Todor},
     title = {Invariant measures on products and on the~space of linear orders},
     journal = {Journal de l{\textquoteright}\'Ecole polytechnique {\textemdash} Math\'ematiques},
     pages = {155--176},
     publisher = {\'Ecole polytechnique},
     volume = {9},
     year = {2022},
   
}

@misc{me,
  doi = {10.48550/ARXIV.2208.06323},
  
  url = {https://arxiv.org/abs/2208.06323},
  
  author = {Marimon, Paolo},
  
  title = {On the non-measurability of $\omega$-categorical Hrushovski constructions},
  
  publisher = {arXiv},
  
  year = {2022},
  
  copyright = {arXiv.org perpetual, non-exclusive license}
}

@article{AER,
  title={Approximate equivalence relations},
  author={Hrushovski, Ehud},
  journal={Model Theory},
  volume={3},
  number={2},
  pages={317--416},
  year={2024},
  publisher={Mathematical Sciences Publishers}
}

@misc{PieceInt,
  doi = {10.48550/ARXIV.2110.05142},
  
  url = {https://arxiv.org/abs/2110.05142},
  
  author = {Chevalier, Alexis and Hrushovski, Ehud},
  
  title = {Piecewise Interpretable Hilbert Spaces},
  
  publisher = {arXiv},
  
  year = {2021},
  
  copyright = {Creative Commons Attribution 4.0 International}
}

@article{EvansTsk,
author = {Evans, David M. and Tsankov, Todor},
doi = {10.4064/fm232-1-4},
journal = {Fundamenta Mathematicae},
pages = {49--63},
title = {Free actions of free groups on countable structures and property (T)},
url = {http://dx.doi.org/10.4064/fm232-1-4},
volume = {232},
year = {2016}
}

@misc{FOAmen,
  doi = {10.48550/ARXIV.2004.08306},
  
  url = {https://arxiv.org/abs/2004.08306},
  
  author = {Hrushovski, Ehud and Krupiński, Krzysztof and Pillay, Anand},
  
  title = {On first order amenability},
  
  publisher = {arXiv},
  
  year = {2020},
  
  copyright = {arXiv.org perpetual, non-exclusive license}
}

@article{CasNIP,
  title={Lascar strong types and forking in NIP theories},
  author={Casanovas, Enrique},
  journal={Rendiconti del Seminario Matematico Universit\'{a} Politecnico Torino},
  volume={72},
  number={3},
  pages={195 – 211},
  year={2014}
}

@book{CasSimp, place={Cambridge}, series={Lecture Notes in Logic}, title={Simple Theories and Hyperimaginaries}, DOI={10.1017/CBO9781139003728}, publisher={Cambridge University Press}, author={Casanovas, Enrique}, year={2011}, collection={Lecture Notes in Logic}}

@article{Albert,
  title={Measures on the random graph},
  author={Albert, Michael H.},
  journal={Journal of the London Mathematical Society},
  volume={50},
  number={3},
  pages={417--429},
  year={1994},
  publisher={Oxford University Press}
}

@misc{me2,
      title={Invariant Keisler measures for omega-categorical structures}, 
      author={Paolo Marimon},
      year={2023},
      eprint={2211.14628},
      archivePrefix={arXiv},
      primaryClass={math.LO}
}

@misc{AtticusPill,
      title={Forking and invariant measures in NIP theories}, 
      author={Anand Pillay and Atticus Stonestrom},
      year={2023},
      eprint={2307.11037},
      archivePrefix={arXiv},
      primaryClass={math.LO}
}

@article{KeislerNIP,
  title={Measures and forking},
  author={Keisler, Howard Jerome},
  journal={Annals of Pure and Applied Logic},
  volume={34},
  number={2},
  pages={119--169},
  year={1987},
  publisher={Elsevier}
}

@article{NIPinv,
  title={On NIP and invariant measures},
  author={Hrushovski, Ehud and Pillay, Anand},
  journal={Journal of the European Mathematical Society},
  volume={13},
  number={4},
  pages={1005--1061},
  year={2011},
  publisher={European Mathematical Society Publishing House}
}

@article{Ensley,
  title={Automorphism--invariant measures on $\aleph_0$-categorical structures without the independence property},
  author={Ensley, Douglas E.},
  journal={The Journal of Symbolic Logic},
  volume={61},
  number={2},
  pages={640--652},
  year={1996},
  publisher={Cambridge University Press}
}

@article{Kallsym,
  title={Symmetries on random arrays and set-indexed processes},
  author={Kallenberg, Olav},
  journal={Journal of Theoretical Probability},
  volume={5},
  number={4},
  pages={727--765},
  year={1992},
  publisher={Springer}
}

@book{Kallbooksym,
  title={Probabilistic symmetries and invariance principles},
  author={Kallenberg, Olav},
  volume={9},
  year={2005},
  publisher={Springer}
}

@inproceedings{AFP,
  title={Invariant measures concentrated on countable structures},
  author={Ackerman, Nathanael and Freer, Cameron and Patel, Rehana},
  booktitle={Forum of Mathematics, Sigma},
  volume={4},
  pages={e17},
  year={2016},
  organization={Cambridge University Press}
}

@article{StarchNIP,
  title={NIP, Keisler measures and combinatorics},
  author={Starchenko, Sergei},
  journal={S{\'e}minaire Bourbaki},
  pages={2015--2016},
  year={2016}
}

@phdthesis{EnsleyPhD,
  title={Measures on $\aleph_0$-categorical structures},
  author={Ensley, Douglas E.~},
  year={1993},
  school={Carnegie Mellon University, Pittsburgh, PA}
}

@article{TomInd,
  title={Independence, measure and pseudofinite fields},
  author={Toma\v{s}i\'{c}, Ivan},
  journal={Selecta Mathematica, New Series},
  volume={12},
  number={2},
  pages={271--306},
  year={2006},
  publisher={Springer Alemania}
}

@article{CVDMI,
issn = {0075-4102},
journal = {Journal f\"{u}r die Reine und Angewandte Mathematik},
pages = {107--135},
volume = {427},
publisher = {Walter De Gruyter},
year = {1992},
title = {Definable sets over finite fields},
copyright = {1992 INIST-CNRS},
address = {BERLIN},
author = {Chatzidakis, Zo\'{e} and Van Den Dries, Lou and Macintyre, Angus},
}

@article{HrushPseud,
  title={Pseudo-finite fields and related structures},
  author={Hrushovski, Ehud},
  journal={Model theory and applications},
  volume={11},
  pages={151--212},
  year={1991},
  publisher={Quaderni di Matematica}
}

@misc{PillayStar,
  doi = {10.48550/ARXIV.1310.7538},
  
  url = {https://arxiv.org/abs/1310.7538},
  
  author = {Pillay, Anand and Starchenko, Sergei},
  
  title = {Remarks on Tao's algebraic regularity lemma},
  
  publisher = {arXiv},
  
  year = {2013},
  
  copyright = {arXiv.org perpetual, non-exclusive license}
}

@misc{TaoReg,
  doi = {10.48550/ARXIV.1211.2894},
  
  url = {https://arxiv.org/abs/1211.2894},
  
  author = {Tao, Terence},
  
  title = {Expanding polynomials over finite fields of large characteristic, and a regularity lemma for definable sets},
  
  publisher = {arXiv},
  
  year = {2012},
  
  copyright = {arXiv.org perpetual, non-exclusive license}
}

@misc{AlexisARL,
  doi = {10.48550/ARXIV.2204.01158},
  
  url = {https://arxiv.org/abs/2204.01158},
  
  author = {Chevalier, Alexis and Levi, Elad},
  
  title = {An Algebraic Hypergraph Regularity Lemma},
  
  publisher = {arXiv},
  
  year = {2022},
  
  copyright = {Creative Commons Attribution 4.0 International}
}

@article{CraneTown,
  title={Relatively exchangeable structures},
  author={Crane, Harry and Towsner, Henry},
  journal={The Journal of Symbolic Logic},
  volume={83},
  number={2},
  pages={416--442},
  year={2018},
  publisher={Cambridge University Press}
}

@article{crane2018relative,
  title={Relative exchangeability with equivalence relations},
  author={Crane, Harry and Towsner, Henry},
  journal={Archive for Mathematical Logic},
  volume={57},
  pages={533--556},
  year={2018},
  publisher={Springer}
}

@article{Palacinrandom,
 ISSN = {00224812, 19435886},
 URL = {https://www.jstor.org/stable/26600291},
 author = {Daniel Palac\'{i}n},
 journal = {The Journal of Symbolic Logic},
 number = {4},
 pages = {1409--1421},
 publisher = {[Association for Symbolic Logic, Cambridge University Press]},
 title = {Generalised amalgamation and homogeneity},
 urldate = {2024-02-08},
 volume = {82},
 year = {2017}
}

@article{Freeam,
 ISSN = {00224812, 19435886},
 URL = {http://www.jstor.org/stable/26358469},
 author = {Conant, Gabriel},
 journal = {The Journal of Symbolic Logic},
 number = {2},
 pages = {648--671},
 publisher = {[Association for Symbolic Logic, Cambridge University Press]},
 title = {An Axiomatic Approach to Free Amalgamation},
 urldate = {2024-02-08},
 volume = {82},
 year = {2017}
}

@article{IREs,
  title={Stabilizers for ergodic actions and invariant random expansions of non-archimedean Polish groups},
  author={Jahel, Colin and Joseph, Matthieu},
  journal={arXiv preprint arXiv:2307.06253},
  year={2023}
}

@book{Descriptive,
  title={The descriptive set theory of Polish group actions},
  author={Becker, Howard and Kechris, Alexander S.},
  volume={232},
  year={1996},
  publisher={Cambridge University Press}
}

@article{PetrovVershik,
  title={Uncountable graphs and invariant measures on the set of universal countable graphs},
  author={Petrov, Fedor and Vershik, Anatoly},
  journal={Random Structures \& Algorithms},
  volume={37},
  number={3},
  pages={389--406},
  year={2010},
  publisher={Wiley Online Library}
}

@book{Taointro,
  title={An introduction to measure theory},
  author={Tao, Terence},
  volume={126},
  year={2011},
  publisher={American Mathematical Soc.}
}

@article{Kopconstr,
  title={On constraints and dividing in ternary homogeneous structures},
  author={Koponen, Vera},
  journal={The Journal of Symbolic Logic},
  volume={83},
  number={4},
  pages={1691--1721},
  year={2018},
  publisher={Cambridge University Press}
}

@article{BinKop,
 ISSN = {00224812, 19435886},
 URL = {http://www.jstor.org/stable/26358443},
 author = {Vera Koponen},
 journal = {The Journal of Symbolic Logic},
 number = {1},
 pages = {183--207},
 publisher = {[Association for Symbolic Logic, Cambridge University Press]},
 title = {Binary primitive homogeneous simple structures},
 urldate = {2024-06-11},
 volume = {82},
 year = {2017}
}

@article{Thomashyp,
title = {Reducts of random hypergraphs},
journal = {Annals of Pure and Applied Logic},
volume = {80},
number = {2},
pages = {165-193},
year = {1996},
issn = {0168-0072},
doi = {https://doi.org/10.1016/0168-0072(95)00061-5},
url = {https://www.sciencedirect.com/science/article/pii/0168007295000615},
author = {Simon Thomas}
}

@misc{ApproxRamsey,
      title={All These Approximate Ramsey Properties}, 
      author={Nadav Meir and Aris Papadopoulos},
      year={2023},
      eprint={2307.14468},
      archivePrefix={arXiv},
      primaryClass={math.CO}
}

@phdthesis{myPhD,
	author = {Marimon, Paolo},
	school = {Department of Mathematics, Imperial College London},
	title = {Measures and amalgamation properties in $\omega$-categorical structures},
	year = {2023}}

@article{Lachlanhyp,
  title={On countable homogeneous 3-hypergraphs},
  author={Akhtar, Reza and Lachlan, Alistair H.},
  journal={Archive for Mathematical Logic},
  volume={34},
  number={5},
  pages={331--344},
  year={1995},
  publisher={Springer}
}

@article{LachWood,
  title={Countable ultrahomogeneous undirected graphs},
  author={Lachlan, Alistair H. and Woodrow, Robert E.},
  journal={Transactions of the American Mathematical Society},
  pages={51--94},
  year={1980},
  publisher={JSTOR}
}

@incollection {McDiarmid,
    AUTHOR = {McDiarmid, Colin},
     TITLE = {On the method of bounded differences},
 BOOKTITLE = {Surveys in combinatorics, 1989 ({N}orwich, 1989)},
    SERIES = {London Math. Soc. Lecture Note Ser.},
    VOLUME = {141},
     PAGES = {148--188},
 PUBLISHER = {Cambridge Univ. Press, Cambridge},
      YEAR = {1989},
      ISBN = {0-521-37823-0},
   MRCLASS = {05C80 (60E15 60F10 60G42)},
  MRNUMBER = {1036755},
MRREVIEWER = {Alan\ M.\ Frieze},
}

@article{AKL,
  title={Random orderings and unique ergodicity of automorphism groups},
  author={Angel, Omer and Kechris, Alexander S. and  Lyons, Russell },
  journal={Journal of the European Mathematical Society},
  volume={16},
  number={10},
  pages={2059--2095},
  year={2014}
}

@InProceedings{Aldous,
author="Aldous, David J.",
editor="Hennequin, P. L.",
title="Exchangeability and related topics",
booktitle="{\'E}cole d'{\'E}t{\'e} de Probabilit{\'e}s de Saint-Flour XIII --- 1983",
year="1985",
publisher="Springer Berlin Heidelberg",
address="Berlin, Heidelberg",
pages="1--198",
isbn="978-3-540-39316-0"
}

@misc{AckerAut,
      title={Representations of Aut(M)-Invariant Measures}, 
      author={Nathanael Ackerman},
      year={2021},
      eprint={1509.06170},
      archivePrefix={arXiv},
      primaryClass={math.LO}
}

@article{AldousT,
title = {Representations for partially exchangeable arrays of random variables},
journal = {Journal of Multivariate Analysis},
volume = {11},
number = {4},
pages = {581-598},
year = {1981},
issn = {0047-259X},
doi = {https://doi.org/10.1016/0047-259X(81)90099-3},
url = {https://www.sciencedirect.com/science/article/pii/0047259X81900993},
author = {David J. Aldous}
}

@article{HooverT,
  title={Relations on Probability Spaces and Arrays of Random Variables},
  author={Hoover, Douglas N.},
  journal={Preprint. Institute for Advanced Study. Princeton},
url={https://www.stat.berkeley.edu/users/aldous/Research/hoover.pdf},
  year={1979}
}

@inproceedings{deFinetti1,
  title={Funzione caratteristica di un fenomeno aleatorio},
  author={De Finetti, Bruno},
  booktitle={Atti del Congresso Internazionale dei Matematici: Bologna del 3 al 10 de settembre di 1928},
  pages={179--190},
  year={1929}
}

@misc{CRO,
      title={Consistent random vertex-orderings of graphs}, 
      author={Paul Balister and B\'{e}la Bollob\'{a}s and Svante Janson},
      year={2015},
      eprint={1506.03343},
      archivePrefix={arXiv}
}

@Article{AFKwP,
    Author = {Nathanael {Ackerman} and Cameron {Freer} and Aleksandra {Kwiatkowska} and Rehana {Patel}},
    Title = {A classification of orbits admitting a unique invariant measure.},
    FJournal = {{Annals of Pure and Applied Logic}},
    Journal = {{Ann. Pure Appl. Logic}},
    ISSN = {0168-0072},
    Volume = {168},
    Number = {1},
    Pages = {19--36},
    Year = {2017},
    Publisher = {Elsevier (North-Holland), Amsterdam},
    DOI = {10.1016/j.apal.2016.08.003},
    MSC2010 = {03C98 37L40 60G09 20B27}
}

@misc{AFKrP,
      title={Properly ergodic structures}, 
      author={Nathanael Ackerman and Cameron Freer and Alex Kruckman and Rehana Patel},
      year={2017},
      eprint={1710.09336},
      archivePrefix={arXiv}
}

@book{Cranebook,
  title={Probabilistic foundations of statistical network analysis},
  author={Crane, Harry},
  year={2018},
  publisher={Chapman and Hall/CRC}
}

@article{Kallenbergsome,
  title={Some highlights from the theory of multivariate symmetries},
  author={Kallenberg, Olav},
  journal={Rend. Mat. Appl.(7)},
  volume={28},
  pages={19--32},
  year={2008}
}

@article{GroupmeasuresNIP,
  title={Groups, measures, and the NIP},
  author={Hrushovski, Ehud and Peterzil, Ya’acov and Pillay, Anand},
  journal={Journal of the American Mathematical Society},
  volume={21},
  number={2},
  pages={563--596},
  year={2008}
}

@article{Genericallystable,
  title={Generically stable and smooth measures in NIP theories},
  author={Hrushovski, Ehud and Pillay, Anand and Simon, Pierre},
  journal={Transactions of the American Mathematical Society},
  volume={365},
  number={5},
  pages={2341--2366},
  year={2013}
}

@article{CherStarchNIP,
  title={Definable regularity lemmas for NIP hypergraphs},
  author={Chernikov, Artem and Starchenko, Sergei},
  journal={The Quarterly Journal of Mathematics},
  volume={72},
  number={4},
  pages={1401--1433},
  year={2021},
  publisher={Oxford University Press UK}
}

@article{RegularityNIP,
  title={Structure and regularity for subsets of groups with finite VC-dimension},
  author={Conant, Gabriel and Pillay, Anand and Terry, Caroline},
  journal={Journal of the European Mathematical Society},
  volume={24},
  number={2},
  pages={583--621},
  year={2021}
}

@inproceedings{GroupStableReg,
  title={A group version of stable regularity},
  author={Conant, Gabriel and Pillay, Anand and Terry, Caroline},
  booktitle={Mathematical Proceedings of the Cambridge Philosophical Society},
  volume={168},
  number={2},
  pages={405--413},
  year={2020},
  organization={Cambridge University Press}
}

@article{DominationRegularity,
  title={Domination and regularity},
  author={Pillay, Anand},
  journal={Bulletin of Symbolic Logic},
  volume={26},
  number={2},
  pages={103--117},
  year={2020},
  publisher={Cambridge University Press}
}

@article{MalliarisReg,
  title={Regularity lemmas for stable graphs},
  author={Malliaris, Maryanthe and Shelah, Saharon},
  journal={Transactions of the American Mathematical Society},
  volume={366},
  number={3},
  pages={1551--1585},
  year={2014}
}

@article{Keislerwild,
  title={Keisler measures in the wild},
  author={Conant, Gabriel and Gannon, Kyle and Hanson, James},
  journal={Model Theory},
  volume={2},
  number={1},
  pages={1--67},
  year={2023},
  publisher={Mathematical Sciences Publishers}
}

@book{FSFT,
  title={Finite structures with few types},
  author={Cherlin, Gregory and Hrushovski, Ehud},
  year={2003},
  publisher={Princeton University Press}
}

@article{WolfMEC,
  title={Multidimensional exact classes, smooth approximation and bounded 4-types},
  author={Wolf, Daniel},
  journal={The Journal of Symbolic Logic},
  volume={85},
  number={4},
  pages={1305--1341},
  year={2020},
  publisher={Cambridge University Press}
}

@book{Oligomperm, place={Cambridge}, series={London Mathematical Society Lecture Note Series}, title={Oligomorphic Permutation Groups}, publisher={Cambridge University Press}, author={Cameron, Peter J.}, year={1990}, collection={London Mathematical Society Lecture Note Series}}

@book{KechrisDST,
  title={Classical descriptive set theory},
  author={Kechris, Alexander},
  volume={156},
  year={2012},
  publisher={Springer Science \& Business Media}
}

@article{CatMod,
  title={On the category of models of a complete theory},
  author={Lascar, Daniel},
  journal={The Journal of Symbolic Logic},
  volume={47},
  number={2},
  pages={249--266},
  year={1982},
  publisher={Cambridge University Press}
}

@book{Geomstab,
  title={Geometric stability theory},
  author={Pillay, Anand},
  year={1996},
  publisher={Oxford University Press}
}

@article{MacNIPgrowth,
  title={Infinite permutation groups of rapid growth},
  author={Macpherson, Dugald},
  journal={Journal of the London Mathematical Society},
  volume={2},
  number={2},
  pages={276--286},
  year={1987},
  publisher={Oxford University Press}
}

@article{SBM,
title = {Stochastic blockmodels: First steps},
journal = {Social Networks},
volume = {5},
number = {2},
pages = {109-137},
year = {1983},
issn = {0378-8733},
doi = {https://doi.org/10.1016/0378-8733(83)90021-7},
url = {https://www.sciencedirect.com/science/article/pii/0378873383900217},
author = {Paul W. Holland and Kathryn Blackmond Laskey and Samuel Leinhardt}
}

@inproceedings{BorgGraphons,
  title={Graphons: A nonparametric method to model, estimate, and design algorithms for massive networks},
  author={Borgs, Christian and Chayes, Jennifer},
  booktitle={Proceedings of the 2017 ACM Conference on Economics and Computation},
  pages={665--672},
  year={2017}
}

@article{Bollobassparse,
  title={Sparse graphs: metrics and random models},
  author={Bollob{\'a}s, B{\'e}la and Riordan, Oliver},
  journal={Random Structures \& Algorithms},
  volume={39},
  number={1},
  pages={1--38},
  year={2011},
  publisher={Wiley Online Library}
}

@article{BorgsparrseI, 
  title={An $L^p$ theory of sparse graph convergence I: Limits, sparse random graph models, and power law distributions},
  author={Borgs, Christian and Chayes, Jennifer and Cohn, Henry and Zhao, Yufei},
  journal={Transactions of the American Mathematical Society},
  volume={372},
  number={5},
  pages={3019--3062},
  year={2019}
}

@article{BorgAH,
  author  = {Christian Borgs and Jennifer Chayes and Henry Cohn and Nina Holden},
  title   = {Sparse Exchangeable Graphs and Their Limits via Graphon Processes},
  journal = {Journal of Machine Learning Research},
  year    = {2018},
  volume  = {18},
  number  = {210},
  pages   = {1--71},
  url     = {http://jmlr.org/papers/v18/16-421.html}
}

@book{Infpermnotes,
  title={Notes on infinite permutation groups},
  author={Bhattacharjee, Meenaxi and M{\"o}ller, R{\"o}gnvaldur G and Macpherson, Dugald and Neumann, Peter M.},
  year={2006},
  publisher={Springer}
}

@book{adeleke1998relations,
  title={Relations Related to Betweenness: Their Structure and Automorphisms},
  author={Adeleke, Samson Adepoju and Neumann, Peter M},
  volume={623},
  year={1998},
  publisher={American Mathematical Soc.}
}

@book{droste1985structure,
  title={Structure of partially ordered sets with transitive automorphism groups},
  author={Droste, Manfred},
  volume={334},
  year={1985},
  publisher={American Mathematical Society}
}

@article{lovasz2006limits,
  title={Limits of dense graph sequences},
  author={Lov{\'a}sz, L{\'a}szl{\'o} and Szegedy, Bal{\'a}zs},
  journal={Journal of Combinatorial Theory, Series B},
  volume={96},
  number={6},
  pages={933--957},
  year={2006},
  publisher={Elsevier}
}

@book{lovasz2012large,
  title={Large networks and graph limits},
  author={Lov{\'a}sz, L{\'a}szl{\'o}},
  volume={60},
  year={2012},
  publisher={American Mathematical Soc.}
}

@article{borgs2008convergent,
  title={Convergent sequences of dense graphs I: Subgraph frequencies, metric properties and testing},
  author={Borgs, Christian and Chayes, Jennifer and Lov{\'a}sz, L{\'a}szl{\'o} and S{\'o}s, Vera T. and Vesztergombi, Katalin},
  journal={Advances in Mathematics},
  volume={219},
  number={6},
  pages={1801--1851},
  year={2008},
  publisher={Elsevier}
}

@article{borgs2012convergent,
  title={Convergent sequences of dense graphs II. Multiway cuts and statistical physics},
  author={Borgs, Christian and Chayes, Jennifer and Lov{\'a}sz, L{\'a}szl{\'o} and S{\'o}s, Vera T. and Vesztergombi, Katalin},
  journal={Annals of Mathematics},
  pages={151--219},
  year={2012},
  publisher={JSTOR}
}

@article{janson2007graph,
  title={Graph limits and exchangeable random graphs},
  author={Janson, Svante and Diaconis, Persi},
  journal={Rendiconti di Matematica e delle sue Applicazioni. Serie VII},
  pages={33--61},
  year={2008}
}

@article{krupinski2019amenability,
  title={Amenability, definable groups, and automorphism groups},
  author={Krupi{\'n}ski, Krzysztof and Pillay, Anand},
  journal={Advances in Mathematics},
  volume={345},
  pages={1253--1299},
  year={2019},
  publisher={Elsevier}
}

@misc{mutchnik2024conantindependencegeneralizedfreeamalgamation,
      title={Conant-independence and generalized free amalgamation}, 
      author={Scott Mutchnik},
      year={2024},
      eprint={2210.07527},
      archivePrefix={arXiv},
      primaryClass={math.LO},
      url={https://arxiv.org/abs/2210.07527}, 
}

@article{chernikov2016model,
  title={On model-theoretic tree properties},
  author={Chernikov, Artem and Ramsey, Nicholas},
  journal={Journal of Mathematical Logic},
  volume={16},
  number={02},
  pages={1650009},
  year={2016},
  publisher={World Scientific}
}

@phdthesis{ColinPhD,
	author = {Jahel, Colin},
	school = {Universit\'{e} Claude Bernard Lyon 1},
	title = {Some progress on the unique ergodicity problem},
	year = {2021}}

@misc{dzamonja2022graphonsarisinggraphsdefinable,
      title={Graphons arising from graphs definable over finite fields}, 
      author={Mirna Džamonja and Ivan Tomašić},
      year={2022},
      eprint={1707.06296},
      archivePrefix={arXiv},
      primaryClass={math.LO},
      url={https://arxiv.org/abs/1707.06296}, 
}

@article{chernikov2020hypergraph,
  title={Hypergraph regularity and higher arity VC-dimension},
  author={Chernikov, Artem and Towsner, Henry},
  journal={arXiv preprint arXiv:2010.00726},
  year={2020}
}

@article{ben2006schrodinger,
  title={Schr{\"o}dinger’s cat},
  author={Ben Yaacov, Ita\"{i}},
  journal={Israel Journal of Mathematics},
  volume={153},
  pages={157--191},
  year={2006},
  publisher={Springer}
}

@article{krivine1981espaces,
  title={Espaces de Banach stables},
  author={Krivine, Jean-Louis and Maurey, Bernard},
  journal={Israel Journal of Mathematics},
  volume={39},
  pages={273--295},
  year={1981},
  publisher={Springer}
}

@online{taopost, 
title={A spectral theory proof of the algebraic regularity lemma}, 
author={Tao, Terence}, 
year={2013},
url={https://terrytao.wordpress.com/2013/10/29/a-spectral-theory-proof-of-the-algebraic-regularity-lemma/}
}

@article{shelah2009dependent,
  title={Dependent first order theories, continued},
  author={Shelah, Saharon},
  journal={Israel Journal of Mathematics},
  volume={173},
  number={1},
  pages={1--60},
  year={2009},
  publisher={Springer}
}

@inbook{Aldous_2010, place={Cambridge}, series={London Mathematical Society Lecture Note Series}, title={More uses of exchangeability: representations of complex random structures}, booktitle={Probability and Mathematical Genetics: Papers in Honour of Sir John Kingman}, publisher={Cambridge University Press}, author={Aldous, David J.}, editor={Bingham, N. H. and Goldie, C. M.Editors}, year={2010}, pages={35–63}, collection={London Mathematical Society Lecture Note Series}}

@article{keevash2011hypergraph,
  title={Hypergraph turan problems},
  author={Keevash, Peter},
  journal={Surveys in combinatorics},
  volume={392},
  pages={83--140},
  year={2011},
  publisher={Cambridge University Press Cambridge}
}

@article{ndependence,
author = {Artem Chernikov and Daniel Palacin and Kota Takeuchi},
title = {{On $n$-Dependence}},
volume = {60},
journal = {Notre Dame Journal of Formal Logic},
number = {2},
publisher = {Duke University Press},
pages = {195 -- 214},
year = {2019},
doi = {10.1215/00294527-2019-0002},
URL = {https://doi.org/10.1215/00294527-2019-0002}
}

@inproceedings{cherlin2021ramsey,
  title={Ramsey expansions of 3-hypertournaments},
  author={Cherlin, Gregory and Hubi{\v{c}}ka, Jan and Kone{\v{c}}n{\`y}, Mat{\v{e}}j and Ne{\v{s}}et{\v{r}}il, Jaroslav},
  booktitle={Extended Abstracts EuroComb 2021: European Conference on Combinatorics, Graph Theory and Applications},
  pages={696--701},
  year={2021},
  organization={Springer}
}

@book{alon2015probabilistic,
  title={The probabilistic method},
  author={Alon, Noga and Spencer, Joel H.},
  year={2015},
  publisher={John Wiley \& Sons}
}

@book{shelah1990classification,
  title={Classification theory: and the number of non-isomorphic models},
  author={Shelah, Saharon},
  year={1990},
  publisher={Elsevier}
}

@article{austin2014hierarchical,
  title={A hierarchical version of the de Finetti and Aldous-Hoover representations},
  author={Austin, Tim and Panchenko, Dmitry},
  journal={Probability Theory and Related Fields},
  volume={159},
  pages={809--823},
  year={2014},
  publisher={Springer}
}

@article{jung2021generalization,
  title={A generalization of hierarchical exchangeability on trees to directed acyclic graphs},
  author={Jung, Paul and Lee, Jiho and Staton, Sam and Yang, Hongseok},
  journal={Annales Henri Lebesgue},
  volume={4},
  pages={325--368},
  year={2021}
}

@article{seidel1991survey,
  title={A survey of two-graphs},
  author={Seidel, Johan Jacob},
  journal={Geometry and Combinatorics},
  pages={146--176},
  year={1991},
  publisher={Elsevier}
}

@article{pemantle1992automorphism,
  title={Automorphism invariant measures on trees},
  author={Pemantle, Robin},
  journal={The Annals of Probability},
  pages={1549--1566},
  year={1992},
  publisher={JSTOR}
}

\end{document}